\documentclass[12pt,reqno,a4paper]{amsart}
\usepackage{extsizes}
\usepackage[margin=1in]{geometry}

\calclayout
\usepackage{graphicx}
\usepackage[utf8]{inputenc}
\usepackage{nicefrac}
\usepackage{mathabx}
 \usepackage{subcaption}
\usepackage{amsrefs}
\usepackage{esint}
\usepackage{mathtools}
\mathtoolsset{showonlyrefs}
\usepackage{wrapfig}
\usepackage{xcolor}

\newtheorem{theorem}{Theorem}[section]
\newtheorem{lemma}[theorem]{Lemma}

\theoremstyle{definition}
\newtheorem{definition}[theorem]{Definition}

\newtheorem{prop}[theorem]{Proposition}
\newtheorem{cor}[theorem]{Corollary}

\allowdisplaybreaks
\theoremstyle{remark}
\newtheorem{remark}[theorem]{Remark}

\numberwithin{equation}{section}

\allowdisplaybreaks
\newcommand{\diam}{\mathrm{diam}}

\newcommand{\mres}{\mathbin{\vrule height 1.6ex depth 0pt width
0.13ex\vrule height 0.13ex depth 0pt width 1.3ex}}
\usepackage{amsmath}
\usepackage{braket} % \Set{} is available
\makeatletter
\@namedef{subjclassname@2020}{%
  \textup{2020} Mathematics Subject Classification}
\makeatother

\begin{document}

\title[Elliptic equations with Hausdorff measures]{On elliptic equations involving surface measures}

% Info Anna 
%\author{Anna Dall'Acqua}
%\address{Universität Ulm, Institut für angewandte Analysis, 89081 Ulm}
%\email{anna.dallacqua@uni-ulm.de}
\author{Marius Müller}
\address{Universität Augsburg, Institut für Mathematik, 86135 Augsburg}
%%    Current address
%\curraddr{Department of Mathematics and Statistics,
%Case Western Reserve University, Cleveland, Ohio 43403}
\email{marius1.mueller@uni-a.de}
%\author{Shinya Okabe}
%\address{Mathematical Institute, Tohoku University, Sendai, 980-8578, Japan}
%\email{shinya.okabe@tohoku.ac.jp}
%\author{Kensuke Yoshizawa}
%\address{Institute of Mathematics for Industry, Kyushu University, Fukuoka, 819-0395, Japan}
%\email{k-yoshizawa@imi.kyushu-u.ac.jp}
%%    Address of record for the research reported here

%%    \thanks will become a 1st page footnote.
%\thanks{.}
%%    Information for first author
%\author{Tatsuya Miura}
%%    Address of record for the research reported here
%\address{Department of Mathematics, Louisiana State University, Baton
%Rouge, Louisiana 70803}
%%    Current address
%\curraddr{Department of Mathematics and Statistics,
%Case Western Reserve University, Cleveland, Ohio 43403}
%\email{xyz@math.university.edu}
%%    \thanks will become a 1st page footnote.
%%    Information for second author
%\author{Marius MÃŒller}
%\address{Helmholtzstrasse 18, 89081 Ulm, Germany }
%\email{marius.mueller@uni-ulm.de}

%    General info

\date{\today}


\keywords{Interfaces, Elliptic equations, Regularity theory, Measure-valued PDE \\
\textit{Acknowledgements.} The author would like to thank the anonymous referee for helpful suggestions.}

\begin{abstract}
We show optimal Lipschitz regularity for very weak solutions of the (measure-valued) elliptic PDE  $-\mathrm{div}(A(x) \nabla u) = Q \; \mathcal{H}^{n-1} \mres \Gamma$ in a smooth domain $\Omega \subset \mathbb{R}^n$. Here $\Gamma$ is a $C^{1,\alpha}$-regular hypersurface,  $Q\in C^{0,\alpha}$ is a density on $\Gamma$, and the coefficient matrix $A$ is symmetric, uniformly elliptic and $W^{1,q}$-regular $(q > n)$. We also discuss optimality of these assumptions on the data. The equation can be understood as a special coupling of two $A$-harmonic functions with an interface $\Gamma$. As such it plays an important role in several free boundary problems, as we shall discuss. 
%This article studies optimal regularity of the equation $(-\Delta)^2 u = Q \; \mathcal{H}^{n-1} \mres \Gamma$, where $\Gamma$ is an $(n-1)$-dimensional submanifold of $\mathbb{R}^n$, $\mathcal{H}^{n-1}$ is the Hausdorff measure, and $Q$ is some suitably regular density. 
%The article extends findings in \cite{PoissonMarius}, where the second-order equation $-\Delta u = Q \; \mathcal{H}^{n-1} \mres \Gamma$ is studied. As an application we derive (optimal) $W^{3,\infty}$-regularity for solutions of the biharmonic Alt-Caffarelli problem in two dimensions. 
\end{abstract}
\maketitle

%\tableofcontents

\section{Introduction} \label{section:intro}

In this article we examine regularity of solutions of the following \emph{measure-valued Dirichlet problem} 

\begin{equation}\label{eq:1.1}
    \begin{cases}
    -\mathrm{div}(A(x) \nabla u) = Q \; \mathcal{H}^{n-1} \mres \Gamma & \textrm{in }\Omega,  \\ \qquad \qquad \quad \quad  u = 0 & \textrm{on } \partial \Omega,
    \end{cases}
\end{equation}
on some smooth bounded domain $\Omega \subset \mathbb{R}^n$, $n \geq 2$. Here,  $\mathcal{H}^{n-1} \mres \Gamma$ denotes the $n-1$-dimensional Hausdorff measure restricted to an \emph{interface} $\Gamma = \partial \Omega'$ for some suitably regular domain $\Omega' \subset \subset \Omega$ and  $Q : \Gamma \rightarrow \mathbb{R}$ is some suitably regular \emph{density}. Further, the (nonconstant) coefficient matrix $(A(x))_{x \in \Omega}$ is suitably regular as well as \emph{symmetric and uniformly elliptic}  in the sense that
\begin{equation}\label{eq:unifellip}
    a_{ij} = a_{ji} \quad \forall i,j = 1,...,n  \quad \textrm{and} \quad \exists \lambda > 0 : \xi^T A(x) \xi \geq \lambda |\xi|^2 \quad \forall x \in \Omega \quad \forall \xi \in \mathbb{R}^n.  
\end{equation}
%For technical reasons we also require that $\mathrm{div}(A) \in L^\infty(\Omega)$.
Such equations appear often in the description of \emph{free boundary problems}, as we shall illustrate in Section \ref{sec:applic}.

The notion of solutions we study is the following
\begin{definition}[Very weak solutions]\label{def:weaksol}
Let $A \in W^{1,2}(\Omega; \mathbb{R}^{n\times n})$ be symmetric and uniformly elliptic and $Q \in L^1(\Gamma)$.
 A function $u \in L^2(\Omega)$ is called a \emph{very weak solution} of \eqref{eq:1.1} if 
 \begin{equation}
   - \int_\Omega  u(x)  \; \mathrm{div}(A(x) \nabla \phi(x))  \; \mathrm{d}x = \int_\Gamma Q(y) \phi(y) \; \mathrm{d}\mathcal{H}^{n-1}(y)  \quad \forall \phi \in C^2(\overline{\Omega}) : \phi \vert_{\partial \Omega} = 0.
\end{equation}
\end{definition}
In Appendix \ref{app:weakform}, we see that under some suitable further smoothness assumptions on the coefficients $Q,A,\Gamma$ one can equivalently demand that $u \in W_0^{1,2}(\Omega)$ and 
\begin{equation}\label{eq:weaaksol}
    \int_\Omega (A(x) \nabla u(x) , \nabla \phi(x) ) \; \mathrm{d}x = \int_\Gamma Q(y) \phi(y) \; \mathrm{d}\mathcal{H}^{n-1}(y) \quad \forall \phi \in C_0^\infty(\Omega). 
\end{equation}
Solutions in the sense of \eqref{eq:weaaksol} are called \emph{weak solutions}, and might appear more intuitive at first sight --- in particular since they can easily be seen to be unique. 
%While existence theory for weak solutions is often easier than for very weak solutions, the concept has a downside:
%In some applications one wants to require as little a priori regularity on $u$ as possible. 
The concept of very weak solutions is however the established definition of general \emph{measure-valued Poisson problems}, cf. \cite[Section 3.1]{Ponce} and will also prove helpful during the course of our proof.  
%We show in section (REFERENZ) that such solution always exists under the given hypotheses. 

%While the latter notion might seem more natural to some readers, it is equivalent to the notion given in Definition \ref{def:weaksol}, cf. Appendix \ref{app:weakform}.

We are interested in the optimal regularity of solutions,  under suitable further assumptions on the data  $Q$, $A$, $\Gamma$. 
%Our main focus is to look at optimal assumptions on $\Gamma$ and $Q$. For $A$ we assume $A \in W^{1,q}(\overline{\Omega};\mathbb{R}^{n\times n})$ for convenience. 
%Equations like \eqref{eq:1.1} are important as they appear in many contexts in \emph{free boundary problems}, cf. (TODO).
%Lipschitz (i.e. $W^{1,\infty}$-)regularity is the best expectable regularity for these problems and actually the only interesting case. Indeed, $W^{2,1}$-regularity is impossible as then $\mathrm{div}(A(x) \nabla u) \in L^1$ but $Q \; \mathcal{H}^{n-1} \mres \Gamma \not \in L^1$, contradicting \eqref{eq:1.1}. On contrary $W^{1,p}$-regularity for each $p \in (1,\infty)$ is straightforward to show via \emph{duality method}, we will expand on the details in Section \ref{sec:prelim}. Also in the classical spaces one can not go beyond Lipschitz regularity
%. Indeed, for $A = Id_{n \times n}$ (REFERENZ; Proposition 3.8) shows that $u \not \in C^1(\Omega)$ unless $Q = 0$. 
The main result in this article is  

\begin{theorem}[Main Theorem]\label{thm:main}
Fix $\alpha > 0$ and $q > n$. Let $\Omega \subset \mathbb{R}^n$ be a smooth bounded domain and $\Omega' \subset \subset \Omega$ be a $C^{1,\alpha}$-subdomain.  If  $\Gamma = \partial \Omega'$, $Q \in C^{0,\alpha}(\Gamma)$ and $A \in W^{1,q}(\Omega;\mathbb{R}^{n\times n})$ satisfies \eqref{eq:unifellip}, then the unique solution $u$ of \eqref{eq:1.1}
lies in $W^{1,\infty}(\Omega)$.
\end{theorem}

The class $W^{1,\infty}(\Omega)$ is the best expectable regularity in terms of Sobolev spaces. Indeed $W^{2,1}$-regularity is not possible, since for each $u \in W^{2,1}$ and $A \in W^{1,q}, (q> n)$ the left hand side of \eqref{eq:1.1} lies in $L^1$, whereas the (measure-valued) right hand side does not lie in $L^1$ (unless $Q = 0$). Therefore the optimal Sobolev class has to be below $W^{2,1}$. Moreover, we will see that $W^{1,p}$-regularity for each $p \in (1,\infty)$ is true and a  somewhat immediate consequence of the observations in \cite{Auscher}, which deal with equations of the form $-\mathrm{div}(A \nabla u ) = \mathrm{div}(F)$. To make these results applicable, a needed prerequisite is $A \in VMO(\Omega;\mathbb{R}^{n\times n})$. Recall that for any $q >n$ one has $W^{1,q} \hookrightarrow C^{0,\gamma} \hookrightarrow VMO$ with $\gamma = 1-\frac{n}{q}$. 
All in all, $W^{1,\infty}$-regularity is the only open question in terms of Sobolev regularity. 

It could theoretically happen that $W^{1,\infty}$-regularity can be further improved upon in \emph{classical function spaces} $C^{k,\alpha}$ (recalling that $W^{1,\infty}(\Omega) = C^{0,1}(\overline{\Omega})$). However, such improvement can not be obtained --- we will show that $C^1$-regularity is impossible, cf. Section \ref{sec:notC1}.

If $A = Id_{n\times n}$ (or $A$ is smooth), \eqref{eq:1.1} has already been studied by means of \emph{potential theory}, cf. \cite[Theorem 14.V]{Miranda} and \cite{Giraud}. In this special case the conclusion of Theorem \ref{thm:main} is already known to hold true.

We also show that the regularity assumptions on $Q, A, \Gamma$ are optimal for the conclusion of the theorem. A previous article of the author (\cite{PoissonMarius}) asserts (falsely) that if $A= Id_{n \times n}$ then the conclusion of Theorem \ref{thm:main} holds true even if $\Omega'$ is only a Lipschitz subdomain. Such result is not true (by counterexample, cf. \cite{PoissonMariusErratum}). Notice that the present article is self-contained and does not use any results in \cite{PoissonMarius}. 

We discuss the optimality of our assumptions by giving counterexamples to the conclusion of Theorem \ref{thm:main} if $\Omega'$ is only a Lipschitz domain and if $A \in W^{1,q}$ only for some $q<n$, cf. Section \ref{sec:optimalitydisc}. This is remarkable in the following way: In \cite{Kilpelainen} solutions of the general (Radon-)measure-valued problem 
\begin{equation}\label{eq:measeqgen}
    \begin{cases}
      -\mathrm{div}(A(x) \nabla u(x) ) = \mu & \textrm{in } \Omega, \\  \qquad \qquad \quad  \quad \quad \;    u = 0 & \textrm{on } \partial \Omega
    \end{cases}
\end{equation}
are studied and it is shown (cf. \cite[Theorem 2.9]{Kilpelainen}) that for each $\alpha \in (0,1)$
\begin{equation}
    \mu(B_r(x)) \leq C r^{n-2+ \alpha} \; \forall r > 0,x \in \mathbb{R}^n \; \Leftrightarrow  \; \textrm{The weak solution $u$ of \eqref{eq:measeqgen} lies in $C^{0,\alpha}(\overline{\Omega})$}.
\end{equation}
Our counterexample will reveal that ``$\Rightarrow$'' does in general not hold true in the case of $\alpha = 1$. This reveals that growth properties of the measure are in this case not sufficient to make regularity assertions. The geometry of $\mu$ plays a much bigger role.

Because of the mild requirements on our coefficient matrix we need an approach that is \emph{detached} from potential theory. We will first establish results for smooth data $Q,A,\Gamma$ by means of a comparison of solutions with the \emph{(signed) distance function}. This comparison yields also $BV$-estimates for $\nabla u$, which are crucial for our approach.

We will then perform a blow-up analysis and use special test functions to obtain a-priori estimates in terms of the data. An approximation procedure will then imply the result.

Because of the absence of potential theory, our approach provides methodological progress that can in the future also be applied for nonlinear measure-valued equations.
%(TODO: KILPELAINENS QUESTION) 

\section{Preliminaries}
\label{sec:prelim}

\subsection{Notation}

In the following we will always assume unless stated otherwise that $A$ is symmetric and uniformly elliptic, cf. \eqref{eq:unifellip}.
For an arbitrary closed set $C \subset \mathbb{R}^n$, we say that $f \in C^{0,\alpha}(C)$ if there exists some constant $H > 0$ such that 
\begin{equation}\label{eq:Hoeldiest}
    |f(x)-f(y)| \leq H |x-y|^\alpha \quad \forall x,y \in C.
\end{equation}
We call $ [f]_{C^{0,\alpha}(C)}:= \inf \{ H > 0 : \eqref{eq:Hoeldiest} \textrm{ holds} \} $ the \emph{Hölder seminorm} of $f$ and define the \emph{Hölder norm} $||f||_{C^{0,\alpha}(C)} := \sup_{x \in C} |f(x)| + [f]_{C^{0,\alpha}(C)}$. We remark that if $C$ is a $C^1$-submanifold of $\mathbb{R}^n$ then there would also be an alternative way to define the Hölder-seminorm, namely taking the \emph{intrinsic distance} $\mathrm{dist}_C(x,y)$ instead of $|x-y|$ in \eqref{eq:Hoeldiest}. These notions of Hölder continuity are only equivalent on \emph{chord arc submanifolds}, cf. \cite{Blatt}. We will thus consistently adhere to the notion in \eqref{eq:Hoeldiest}.

In the sequel $\Omega$ always denotes a smooth bounded domain and $\Omega' \subset \subset \Omega$ always denotes a suitably smooth and compactly contained subdomain. Moreover, $\Gamma = \partial \Omega'$ denotes its boundary.  Further we introduce the shorthand notation $\Omega'' := \Omega \setminus \overline{\Omega'}$.

In what follows we will assume tacitly and without loss of generality that $\Gamma =\partial \Omega'$ is connected, since otherwise we could look at the (finitely many) connected components of $\Gamma$ separately. (If we do so, it however needs to be ensured that all connected components $\Gamma$ can also be regarded as boundaries of subdomains of $\Omega$. This is true by the Jordan-Brouwer separation theorem.)  

We say (for $k \in \mathbb{N}_0$ and $\alpha \geq 0$)  that $\Gamma= \partial \Omega'$ is a \emph{$C^{k,\alpha}$-boundary} if it can be written as 
\begin{equation}\label{eq:Gaama}
   \Gamma = \bigcup_{i = 1}^M  O_i [ (U_i \times V_i) \cap \{ (x', f_i(x')): x' \in U_i \}] =: \bigcup_{i = 1}^M \Gamma_i
\end{equation}
where $M \in \mathbb{N}$ is finite, $O_i\in \mathbb{R}^{n\times n}$ are orthogonal matrices, $U_i \subset \mathbb{R}^{n-1}$ are open balls, $V_i \subset \mathbb{R}$ are open intervals and $f_i\in C^{k,\alpha}(4\overline{U_i};V_i)$ are such that the subgraph of $f_i$ lies in $\Omega'$. Here $4 U_i$ denotes the ball with the same center and the quadruple radius of $U_i$. This enlargement of $U_i$ will be needed for technical purposes, cf. Appendix \ref{app:C1alpha}. %discussion after \eqref{eq:gradsigndisti}. 
We say that $\Gamma$ is \emph{represented} by $(O_i,U_i,V_i,f_i$, $i = 1,...,M)$ if $O_i,U_i,V_i,f_i$ are as in \eqref{eq:Gaama} and $\Gamma_{i+1} \cap \Gamma_i \neq \emptyset$ for all $i = 1,...,M$. We remark that the condition that $\Gamma_{i+1} \cap \Gamma_i \neq \emptyset$ can always be arranged by relabeling, since $\Gamma_i \subset \Gamma$ is open for all $i$ and $\Gamma$ is connected. We set  
%If $\nu_{\Omega'}$ denotes the outward pointing unit normal of $\Gamma$ we define 
\begin{align}
    [\Gamma]_{k,\alpha}  & := %\mathrm{diam}(\Gamma) 
    %+ \inf \left\lbrace ||\eta||_{C^{k-1,\alpha}(\mathbb{R}^n)} : \eta \textrm{ is a $C^{k-1,\alpha}$-extension of $\nu_{\Omega'}$ on $\mathbb{R}^n$} \right\rbrace 
       \inf \left\lbrace \sum_{i = 1}^M (1+ \sup_{x\in V_i}|x| + |U_i|+ ||f_i||_{C^{k,\alpha}(4\overline{U_i})}) : \Gamma  \textrm{ rep. by $O_i,U_i,V_i,f_i$, $i =1,...,M$} \right\rbrace.
\end{align}
We use the shorthand notation $\Gamma = \partial \Omega' \in C^{k,\alpha}$ if $\Gamma$ is a $C^{k,\alpha}$-boundary of $\Omega'\subset \subset \Omega$ and we denote the outward pointing unit normal $\nu_{\Omega'}$ simply by $\nu$ for the sake of simplicity. The signed measure $\mu = Q \; \mathcal{H}^{n-1} \mres \Gamma$ is defined via 
\begin{equation}
    \mu(A) := \int_{A\cap \Gamma} Q \; \mathrm{d}\mathcal{H}^{n-1} \quad \forall A \subset \mathbb{R}^n  \textrm{ Borel set.}
\end{equation}
From now on, also $(\cdot,\cdot): \mathbb{R}^n \times \mathbb{R}^n \rightarrow \mathbb{R}$ denotes the standard Euclidean inner product.

%We have already discussed the weak formulation with \emph{Dirichlet boundary values}. Later we will also use different boundary values, which is why we use the following definition 
 
% \begin{definition}
%Let $D \subset \mathbb{R}^2$ be a $C^1$-domain, $h \in L^2(\partial D)$ and $A \in W^{1,2}(D;\mathbb{R}^{n\times n})$. A function $u \in L^2(D)$ is  called a  \emph{weak solution} of 
%\begin{equation}
%    \begin{cases}
%        - \mathrm{div}(A(x)\nabla u) = 0 & \textrm{in } D \\ u = h & \textrm{on } \partial D
%    \end{cases}
%\end{equation}
%if 
%\begin{equation}
%    \int_D u \;  \mathrm{div}(A(x) \nabla \phi) \; \mathrm{d}x + \int_{\partial D} h (A(x) \nabla \phi , \nu_D) \; \mathrm{d}\mathcal{H}^{n-1} = 0 \quad \forall \phi \in C^2(\overline{D}) \cap  W_0^{1,2}(D). 
%\end{equation}
%\end{definition}
%(UNIQUENESS OF SUCH SOLUTION)

\subsection{Existence, Uniqueness and $W^{1,p}$-regularity for $p< \infty$}

First we establish existence and $W^{1,p}$-regularity for all $p \in (1,\infty)$. To this end we employ the \emph{duality method}. We show that the measure term on the right hand side of \eqref{eq:1.1} lies the \emph{dual space} $(W_0^{1,p'}(\Omega))^*$ for any $p' \in (1,\infty)$ and then illustrate how this can be used to show that $u \in W_0^{1,p}(\Omega)$ for $p \in (1,\infty)$ such that $\frac{1}{p}+ \frac{1}{p'} = 1$. This is related to the study of the equation $-\mathrm{div}(A \nabla u) = \mathrm{div}(F)$, conducted under minimal assumptions in \cite{Auscher}. 
%Actually, the measure $\mu = Q \; \mathcal{H}lies in $(W_0^{1,1}(\Omega))^*$ and the case $p' = 1$ would correspond to the case $p = \infty$. However the method that we use, based on an observation in \cite{Auscher}, fails in this limit case. 
\begin{lemma}\label{lem:measterm}
Let $\Gamma=\partial \Omega' \in C^{0,1}$ and  $Q \in L^\infty(\Gamma)$. Then the distribution
\begin{equation}\label{eq:distriT}
    T(\psi) := \int_\Gamma Q \psi \; \mathrm{d}\mathcal{H}^{n-1}(x)  \quad (\psi \in C_0^\infty(\Omega))
\end{equation}
extends to  an element of  $W_0^{1,1}(\Omega)^*$ and $||T||_{(W_0^{1,1}(\Omega)^*} \leq C(n) [\Gamma]_{0,1}||Q||_{L^\infty(\Gamma)}$. Moreover there exists some $F \in L^\infty(\Omega;\mathbb{R}^n)$ such that $Q \mathcal{H}^{n-1} \mres \Gamma = \mathrm{div}(F)$ distributionally, i.e. 
\begin{equation}\label{eq:Fdarst}
    \int_\Gamma Q \psi \; \mathrm{d}\mathcal{H}^{n-1} = \int_\Omega (F, \nabla \psi) \; \mathrm{d}x \quad \forall \psi \in C_0^\infty(\Omega).
\end{equation}
Further, one can choose $F$ in such a way that $||F||_{L^\infty(\Omega)} \leq C(n) [\Gamma]_{0,1} ||Q||_{L^\infty(\Gamma)}$.
\end{lemma}
\begin{proof} To prove that $T \in W_0^{1,1}(\Omega)^*$ it suffices to show (by applying \cite[Proposition 17.17]{Ponce} with $\mu_{\pm}= Q^{\pm} \mathcal{H}^{n-1} \mres \Gamma$ and $k  =1$  and \cite[Proposition B.3]{Ponce} with $\delta=\infty$) that 
\begin{equation}\label{eq:mugrowth}
    \int_{\Gamma \cap B_r(x)}  |Q|  \; \mathrm{d}\mathcal{H}^{n-1}(x) \leq C r^{n-1} \quad \forall x \in \mathbb{R}^n \quad \forall r > 0. 
\end{equation}
We first estimate 
\begin{equation}\label{eq:doublingprop}
    \int_{\Gamma \cap B_r(x)}  |Q|  \; \mathrm{d}\mathcal{H}^{n-1}(x) \leq ||Q||_{L^\infty} \mathcal{H}^{n-1}(\Gamma \cap B_r(x)). 
\end{equation}
%Observe that for each $x \in \Gamma$ there exists some $r_x \in (0,1)$  such that  $\Gamma \cap B_{r_x}(x)$ can up to rotation be covered by  a Lipschitz graph $\{ (x',f(x')) : x' \in W_x \}$ such that $||Df||_\infty \leq [\Gamma]_{0,1} + 1$ and $\Gamma \cap B_{r_x}(x) = (W_x \times \mathbb{R}^{n-1}\times \{0 \}) \cap \mathrm{} . It must then hold that $W_x \subset B_{2r_x}(z)$ for some $z \in \mathbb{R}^{n-1}$ as 
%\begin{equation}
%    \mathrm{diam}(W_x) \leq \mathrm{diam}(\{(x',f(x')): x' \in W_x\}) = \mathrm{diam}(\Gamma \cap B_{r_x}(x)) = 2 r_x
%\end{equation}
%\begin{align}
%    \mathcal{H}^{n-1}(\Gamma \cap B_{r_x}(x))  & = \int_{W_x} \sqrt{1+ |Df(x')|^2} \; \mathrm{d}x'
%\\&  \leq \sqrt{1 + ([\Gamma]_{0,1}+1)^2} |W_x| \leq   \sqrt{1 + ([\Gamma]_{0,1}+1)^2} \alpha_{n-1} r_x^{n-1}.
%\end{align}
Notice that $\Gamma$ is covered by finitely many  sets of the form 
\begin{equation}
  E_i :=  O_i[ (U_i \times V_i) \cap \{(x',f_i(x')) : x' \in U_i \}  \quad i= 1,..., M,
\end{equation}
where $O_i$ is a rotation matrix, $U_i \subset \mathbb{R}^{n-1}$ is an  open set and $V_i \subset \mathbb{R}$ is open and $f_i \in C^{0,1}(4\overline{U_i})$ are such that $ \sum_{i = 1}^M ( 1 + ||Df_i||_{L^\infty}) \leq [\Gamma]_{0,1}+ \epsilon$ for some arbitrary but fixed $\epsilon > 0$. Now let $r > 0$ and $x \in \Gamma$ be arbitrary. For all $i = 1,...,M$ there holds  (with $\tilde{x}_i := O_i^T x$ and $y_i'$ denoting the first $n-1$ components of $\tilde{x}_i$) 
\begin{align}
    O_i^T (E_i \cap (\Gamma \cap B_r(x)) )  &  \subset \{(x', f_i(x')) : x' \in U_i, (x',f_i(x')) \in B_r(\tilde{x}_i)  \}  \\ & \subset  \{(x', f_i(x')) : x' \in U_i, x' \in B_r(y_i') \}. 
\end{align}
Hence (with $\mathcal{L}^{n-1}$ denoting the Lebesgue measure on $\mathbb{R}^{n-1}$) one finds
\begin{align}\label{eq:growthprop}
    \mathcal{H}^{n-1}(\Gamma \cap B_r(x)) & \leq \sum_{i = 1}^M \int_{U_i \cap B_r(y_i')} \sqrt{1+ |Df_i(x')|^2} \; \mathrm{d}x' \\ & \leq \sum_{i = 1}^M \sqrt{1 + ||Df_i||_{\infty}^2} \mathcal{L}^{n-1}(B_r(y_i)) \leq \mathcal{L}^{n-1}(B_1(0)) ([\Gamma]_{0,1}+ \epsilon) r^{n-1}.
\end{align}
Together with \eqref{eq:doublingprop} we infer that 
\begin{equation}
    \int_{\Gamma \cap B_r(x)} |Q| \; \mathrm{d}\mathcal{H}^{n-1}(x) \leq  C(n) ||Q||_{L^\infty} ([\Gamma]_{0,1}+ \epsilon) r^{n-1}.
 \end{equation} 
 This and \cite[Proposition 17.17]{Ponce} yield that $||T||_{W_0^{1,1}(\Omega)^*} \leq  C(n) ||Q||_{L^\infty} ([\Gamma]_{0,1}+ \epsilon)$. Arbitraryness of $\epsilon > 0$ yields the desired bound for $||T||_{W_0^{1,1}(\Omega)^*}$.
 By \cite[Lemma 6.6]{Torres} we infer the existence of $F \in L^\infty(\Omega; \mathbb{R}^n)$ such that $||T||_{(W_0^{1,1}(\Omega))^*} = ||F||_{L^\infty(\Omega)}$ and 
 \begin{equation}
     T(\psi) = \int_\Omega (F, \nabla \psi) \; \mathrm{d}x \quad \forall \psi \in W_0^{1,1}(\Omega). 
 \end{equation}
 This together with the definition of $T$ in \eqref{eq:distriT} implies \eqref{eq:Fdarst}. Finally, we have 
 \begin{equation}
     ||F||_{L^\infty(\Omega)}= ||T||_{(W_0^{1,1}(\Omega))^*} \leq C(n) [\Gamma]_{0,1} ||Q||_{L^\infty(\Gamma)}. 
 \end{equation}
 The claim follows.
\end{proof}

\begin{remark}
The proof exposes a special \emph{growth property} of the measure $\mu = Q \; \mathcal{H}^{n-1} \mres \Gamma$, namely $|\mu|(B_r(x)) \leq Cr^{n-1}$, holding true whenever $\Gamma = \partial \Omega' \in C^{0,1}$. 
\end{remark}

\begin{remark}
We recall that if $\Gamma = \partial \Omega' \in C^{0,1}$ then there exists a continuous \emph{trace operator} $\mathrm{tr}_{\partial \Omega'} : W^{1,1}(\Omega') \rightarrow L^1(\Gamma)$. One readily checks that the extension of $T$, defined as in \eqref{eq:distriT} to $W_0^{1,1}(\Omega)$ is given by 
\begin{equation}\label{eq:Ttrace}
    T(\psi) = \int_\Gamma Q \; \mathrm{tr}_{\partial \Omega'} (\psi) \; \mathrm{d}\mathcal{H}^{n-1}(x).
\end{equation}
We will from now on write $\psi$ instead of $\mathrm{tr}_{\partial \Omega'} (\psi)$ and tacitly regard the evaluation as an evaluation in the sense of Sobolev traces. 
\end{remark}

Next we prove the asserted $W^{1,p}$-regularity. To this end we will need that under suitable regularity requirements on $A$, very weak solutions of the equation are also \emph{weak solutions}. Which regularity requirements are exactly needed is discussed in Appendix \ref{app:weakform}. 

\begin{lemma}\label{lem:W1preg}
Let $\Gamma = \partial \Omega' \in C^{0,1}$, $Q \in L^\infty(\Omega)$ and $A \in W^{1,q}(\Omega; \mathbb{R}^{n\times n})$, $q > n$. Then there exists a unique very weak solution $u \in L^2(\Omega)$ to 
\begin{equation}
    \begin{cases}
        -\mathrm{div}(A(x) \nabla u ) = Q \; \mathcal{H}^{n-1} \mres \Gamma  & \textrm{in } \Omega, \\ \qquad \qquad \quad \; \; \; u = 0 & \textrm{on } \partial \Omega 
    \end{cases}
\end{equation}
and it lies in $W_0^{1,p}(\Omega)$ for all $ p \in (1,\infty)$. Moreover, there exists a constant $C = C(p, n, \Omega)$ such that
\begin{equation}
    ||u||_{W_0^{1,p}(\Omega)} \leq  C [\Gamma]_{0,1}  ||Q||_{L^\infty}. 
    \end{equation}
%If $\mathrm{div}(A) \in L^\infty$ then the solution is unique. 
\end{lemma}
\begin{proof}
\textbf{Step 1.} Existence. We show first that there exists $u \in W_0^{1,2}(\Omega)$ such that
\begin{equation}\label{eq:LaxMilgr}
    \int_{\Omega} (A \nabla u, \nabla \phi) \; \mathrm{d}x = \int_\Gamma Q \phi \; \mathrm{d}\mathcal{H}^{n-1} \quad \forall \phi \in W_0^{1,2}(\Omega).
\end{equation}
To this end we employ the Lax-Milgram lemma. One readily checks that (due to the fact that $A \in W^{1,q} \subset L^\infty$) one has that $a : W_0^{1,2}(\Omega) \times W_0^{1,2}(\Omega) \rightarrow \mathbb{R}$ given by 
\begin{equation}
    a(u,v) := \int_\Omega (A \nabla u, \nabla v) \; \mathrm{d}x 
\end{equation}
defines a continuous and coercive bilinear form on the Hilbert space $H := W_0^{1,2}(\Omega)$. Next let $T$ be the distribution in \eqref{eq:distriT}. The previous lemma shows that $T \in W_0^{1,1}(\Omega)^* \subset H^*$. Therefore the Lax-Milgram lemma yields that there exists a unique solution $u \in H$ such that $a(u,v) = T(v)$ for all $v \in H$. Using \eqref{eq:Ttrace} we conclude  that $u$ solves \eqref{eq:LaxMilgr}. Lemma \ref{lem:weakveryweak} shows that $u$ is also a very weak solution. \\
%We next show that $u$ is also a very weak solution. Indeed, fix $\phi \in C^{2}(\overline{\Omega}) \cap W_0^{1,2}(\Omega)$ and choose some sequence $(\eta_j)_{j \in \mathbb{N}} \subset C_0^\infty(\Omega)$ such that $\eta_j \rightarrow u$ in $W_0^{1,2}(\Omega)$.  We compute
%\begin{align}
 %   \int_\Omega Q \phi \; \mathrm{d}\mathcal{H}^{n-1}  & = \int_{\Omega} (A \nabla u, \nabla \phi) \; \mathrm{d}x = \lim_{j \rightarrow \infty} \int_{\Omega} (A \nabla \eta_j, \nabla \phi) \; \mathrm{d}x \\ &  =
  %  \lim_{j \rightarrow \infty} \int_{\Omega} ( \nabla \eta_j, A \nabla \phi) \; \mathrm{d}x = - \lim_{j \rightarrow \infty} \int_{\Omega}  \eta_j \; \mathrm{div}(A \nabla \phi) \mathrm{d}x.
% \end{align}
 %Passing to the limit we infer 
% \begin{equation}
  %% \end{equation}
 \textbf{Step 2.} Uniqueness. Suppose that $w \in L^2(\Omega)$ is another very weak solution. By Lemma \ref{lem:weakveryweak}  we have 
 \begin{equation}
     - \int_\Omega u \; \mathrm{div}(A \nabla \phi) \; \mathrm{d}x = - \int_\Omega w \; \mathrm{div}(A \nabla \phi) \; \mathrm{d}x= T(\phi) \quad \forall \phi \in W^{2,2}(\Omega) \cap W_0^{1,2}(\Omega),
 \end{equation}
 yielding 
 \begin{equation}\label{eq:kernel}
      \int_\Omega (u-w) \; \mathrm{div}(A \nabla \phi) \; \mathrm{d}x = 0 \quad \forall \phi \in W^{2,2}(\Omega) \cap W_0^{1,2}(\Omega).
 \end{equation}
 Next we let $g \in L^2(\Omega)$ be arbitrary.  Lemma \ref{lem:stronsolution} implies that there exists some $\phi_g \in W^{2,2}(\Omega) \cap W_0^{1,2}(\Omega)$ such that $\mathrm{div}(A \nabla \phi_g) = g$ pointwise almost everywhere. Using $\phi = \phi_g$ in  \eqref{eq:kernel} we find
 \begin{equation}
     \int_\Omega (u-w) g \; \mathrm{d}x = 0 \quad \forall g \in L^2(\Omega).
 \end{equation}
 This implies $u- w =0$ a.e. \\
 \textbf{Step 3.} Regularity. Let $p \in [2,\infty)$ be arbitrary. Let $F \in L^\infty(\Omega;\mathbb{R}^n)$ be as in \eqref{eq:Fdarst}. By \cite[Theorem 1]{Auscher} there exists some $u_p \in W_0^{1,p}(\Omega)$ such that 
 \begin{equation}\label{eq:u_p}
     \int_\Omega (A \nabla u_p , \nabla \phi ) \; \mathrm{d}x  = \int_\Omega (F, \nabla \phi) \; \mathrm{d}x \quad \forall \phi \in C_0^\infty(\Omega). 
 \end{equation}
 We infer from linearity and \eqref{eq:Fdarst}
 \begin{equation}
     \int_\Omega (A \nabla (u_p - u) , \nabla \phi ) \; \mathrm{d}x = \int_\Omega (F, \nabla \phi) \; \mathrm{d}x - \int_\Gamma Q \phi \; \mathrm{d}\mathcal{H}^{n-1} = 0   \quad \forall \phi \in C_0^\infty(\Omega).
 \end{equation}
 Since $u_p - u \in W_0^{1,2}(\Omega)$ and $A \in L^\infty(\Omega)$ we infer by approximation 
 \begin{equation}
     0 = \int_\Omega (A \nabla (u_p - u) , \nabla (u_p - u) ) \; \mathrm{d}x  = a(u_p-u, u_p - u),
 \end{equation}
 yielding that $u_p = u$ by coercivity of the bilinear form $a$. Hence $u= u_p \in W_0^{1,p}(\Omega)$. Thus \eqref{eq:u_p}, \cite[Theorem 1.1]{Auscher} and Lemma \ref{lem:measterm} yield the estimate
 \begin{equation}
     ||u||_{W_0^{1,p}} = ||u_p||_{W_0^{1,p}} \leq C(p,n,\Omega)||F||_{L^p} \leq C(p,n,\Omega) ||F||_{L^\infty} \leq C(p,n, \Omega) ||Q||_\infty [\Gamma]_{0,1}. 
 \end{equation}
\end{proof}

%\begin{remark}
%The regularity discussion in the previous Lemma builds on the \emph{duality method}. This means the idea that if $F \in W^{-1,p}(\Omega)$ for some $p \in (1,\infty)$ then  there exists 
%\end{remark}

\subsection{Impossibility of $C^1$-regularity}\label{sec:notC1}
In this section we discuss why $C^1$-regularity is impossible to obtain for this problem. This will be a consequence of the classical Hopf-Oleinik boundary point lemma for elliptic operators, discovered by Hopf and independently by Oleinik, cf. \cite{Hopf} for elliptic equations in $C^2$-smooth domains. Meanwhile there have also been found versions of this lemma for less regular, e.g. $C^{1,\alpha}$-regular,  boundaries, cf. \cite[Theorem 1.4]{Mazya}.
\begin{lemma}[{Hopf-Oleinik \cite[Special case of Theorem 1.4]{Mazya}}]\label{lem:25}
 Suppose that $D$ is a $C^{1,\alpha}$ domain for some $\alpha > 0$ and  $A \in C^1(D; \mathbb{R}^{n\times n})$ is symmetric and uniformly elliptic. Let 
 $v \in C^2(D)\cap C^0(\overline{D})$ be nonconstant and satisfy $\mathrm{div}(A \nabla v) = 0$ pointwise on $D$. Suppose that  $x_0 \in \partial D$ is such that $v(x_0) = \max_{x \in \overline{D}} v(x)$. Then 
 \begin{equation}
     \lim_{t \rightarrow 0-} \frac{v(x_0+t\nu_D(x_0))- v(x_0)}{t} > 0.
 \end{equation}
\end{lemma}
We remark that the existence of some $x_0 \in \partial D$ with the properties demanded in the previous lemma follows from the \emph{elliptic maximum principle.}, cf. \cite[Theorem 1, Section 6.4.1]{Evans}. We also remark that in the situation of the previous lemma there holds for any $y_0 \in \partial D$
\begin{equation}
    v(y_0) = \min_{x \in \overline{D}} v(x) \quad \Rightarrow \quad  \lim_{t \rightarrow 0-} \frac{v(y_0+t\nu_D(y_0))- v(y_0)}{t} < 0.
\end{equation}
This is readily checked when applying the previous lemma to $-v$. 

Based on this lemma we prove the following result 

\begin{lemma}\label{lem:notC1}
 Let $\Omega \subset \mathbb{R}^n$ be a smooth domain and $u \in L^2(\Omega)$ be the solution of \eqref{eq:1.1} with $A \in C^{1,1}(\overline{\Omega};\mathbb{R}^n)$, $Q \in L^\infty(\mathbb{R}^n)$ and $\Gamma = \partial \Omega' \in C^{1,\alpha}$. Then $u \not \in C^1(\Omega)$ unless $Q = 0$. 
\end{lemma} 
\begin{proof}
Assume that $u \in C^1(\Omega)$ and $Q \not \equiv 0$. We remark that also by Lemma \ref{lem:W1preg} $u \in C(\overline{\Omega}$). In particular $\min_{x \in \overline{\Omega}} u(x)$ and $\max_{x \in \overline{\Omega}} u(x)$ are attained. 
We claim that there holds $\mathrm{div}(A \nabla u ) = 0$ in $\Omega' \cup \Omega''$. Indeed, for each $\phi \in C_0^\infty(\Omega')$ one has (after extending $\phi$ by zero on $\Omega$) 
\begin{equation}
    \int_{\Omega'} (A \nabla u , \nabla \phi) \; \mathrm{d}x = \int_\Omega (A \nabla u , \nabla \phi) \; \mathrm{d}x = \int_\Gamma Q \phi \; \mathrm{d}\mathcal{H}^{n-1} = 0,
\end{equation}
since $\phi \vert_\Gamma = 0$. One infers that $\mathrm{div}(A\nabla u) =0 $ weakly in $\Omega'$. One can proceed analogously for $\phi \in C_0^\infty(\Omega'')$.
We infer from \cite[Theorem 9.19]{GilTru} and the assumptions on $A$ that $u \in W^{3,q}_{loc}(\Omega') \cap W^{3,q}_{loc}(\Omega'')$ for each $q < \infty$. In particular, by Sobolev embedding, $u \in C^2(\Omega')  \cap C^2(\Omega'')$. As a consequence $u \in C^2(\Omega') \cap C(\overline{\Omega'})$ and $u \in C^2(\Omega'') \cap C(\overline{\Omega''})$.  Hence $u$ is on both sets $\Omega'$ and $\Omega''$ a classical solution of $\mathrm{div}(A \nabla u) = 0$.  We next show the following\\
\textbf{Intermediate claim.} One value $\max_{x \in \overline{\Omega}} u(x)$ or $\min_{x \in \overline{\Omega}} u(x)$ is attained on $\Gamma$. By the elliptic maximum principle one has that $\max_{x \in \overline{\Omega'}} u(x)$ is attained on $\partial \Omega' = \Gamma$ and  $\max_{x \in \overline{\Omega''}} u(x)$ is attained on $\partial \Omega'' = \Gamma \cup \partial \Omega$. The same can be observed for the minima. 
One concludes that $\max_{x \in \overline{\Omega}} u(x)$ and $\min_{x \in \overline{\Omega}} u(x) $ must be attained on $\Gamma \cup \partial \Omega$. If both are only attained on $\partial \Omega$ then $u \vert_{\partial \Omega} = 0$ yields 
$
    \max_{x \in \overline{\Omega}} u(x) = \min_{x \in \overline{\Omega}} u(x) = 0,
$
and as a result $u \equiv 0$ in $\Omega$. A contradiction to \eqref{eq:1.1} as $Q \not \equiv 0$. The intermediate claim is shown. To proceed we assume without loss of generality that $\max_{x \in \overline{\Omega}} u(x)$ is attained at some $x_0 \in \Gamma$. % Notice that $u(x_0) >0$. 
We note that $u \not \equiv \mathrm{const.}$ in $\Omega''$. Indeed, assuming the opposite would yield $u \equiv 0$ in $\overline{\Omega''}$ and hence also on $\Gamma$. This would  however imply that $u\vert_{\Omega'}$ is a classical solution of the Dirichlet problem
\begin{equation}
    \begin{cases}
      \mathrm{div}(A \nabla u) = 0 & \textrm{in }\Omega', \\ \qquad  \quad  \; \; \;   u = 0 & \textrm{on } \Gamma = \partial \Omega'
    \end{cases}
\end{equation}
and thus (by uniqueness for the Dirichlet problem) $u \vert_{\Omega'} = 0$, resulting in $u \equiv 0$ and contradicting $Q \not \equiv 0$. %We infer (since also $u \vert_{\Omega''} = 0$ that $u \equiv 0$ on $\Omega$, which again contradicts $Q \not \equiv 0$.
Applying Lemma \ref{lem:25} with $D= \Omega''$  and using that $u \in C^1(\Omega)$ we infer that 
\begin{equation}\label{eq:geq}
  0 <  \lim_{t \rightarrow 0-} \frac{u(x_0+t\nu_{\Omega''}(x_0))- u(x_0)}{t} = (\nabla u(x_0), \nu_{\Omega''}(x_0)).  
\end{equation}
On the other hand we have 
%as $x_0 \in \Gamma$ and $\Gamma = \partial \Omega'$ and the outward normal $\nu_{\Omega'}$ points in opposite direction of $\nu_{\Omega''}$ 
\begin{equation}\label{eq:leq}
    (\nabla u(x_0), \nu_{\Omega''}(x_0)) = - (\nabla u(x_0) , \nu_{\Omega'}(x_0) ) = - \lim_{t \rightarrow 0-} \frac{u(x_0+ t \nu_{\Omega'}(x_0)) - u(x_0)}{t} \leq 0, 
\end{equation}
where the last inequality is due to the fact that $u$ attains a maximum at $x_0$ and thus  $u(x_0 + t \nu_{\Omega'} (x_0) )  \leq u(x_0)$ for all $t \leq 0$. 
%small enough, which implies that the quotient in the last lemma is nonnegative for all $t \leq 0$. 
Since \eqref{eq:geq} and \eqref{eq:leq} contradict each other we obtain that our assumption can not be satisfied. Hence $u \not \in C^1(\Omega)$.
\end{proof}

\section{Proof of the main theorem}
\subsection{Regularity for smooth initial data}\label{sec:smoothinput}

In this section we establish $W^{1,\infty}$-regularity of solutions for (suitably) smooth data $Q,A,\Gamma$.  We will achieve this by comparison with a \emph{prototype solution} that involves the \emph{signed distance function} $d_\Gamma$, a notion that we recall in Appendix \ref{app:signdist}. The comparison with the signed distance reveals not just $W^{1,\infty}$-regularity but also that solutions $u$ of \eqref{eq:1.1} with suitably smooth initial data satisfy $\nabla u \in BV(\Omega)$. This will be important as it allows us to form \emph{traces} of gradients of solutions in the sequel. 

In this section we will always require that $\Gamma = \partial \Omega' \in C^k$ for some $k \geq 2$. We further fix some $\epsilon \in (0,\epsilon_0)$ where $\epsilon_0 = \epsilon_0(\Gamma) < \mathrm{dist}(\Gamma, \partial \Omega)$ is chosen as in Appendix \ref{app:signdist}, i.e. such that $d_\Gamma \in C^k(\overline{B_{\epsilon_0}(\Gamma)})$. Here we denote $B_{\epsilon_0}(\Gamma) := \{ x \in \mathbb{R}^n : \mathrm{dist}(x,\Gamma) < \epsilon_0 \}$.
%\begin{equation}
%    \Gamma = \partial \Omega' \in C^k  \quad \Rightarrow d_\Gamma \in C^k(\overline{B}_\epsilon(\Gamma)). 
%\end{equation}

We will make frequent use of the formula in Lemma \ref{lem:distprep}, i.e. that distributionally on $C_0^\infty(B_\epsilon(\Gamma))'$ one has for any suitably smooth function $Q_0$ (with $\nu:= \nu_{\Omega'}$) 
\begin{equation}
     \partial^2_{ij}\left( \frac{Q_0}{2}|d_\Gamma| \right)  = Q_0 \nu_i \nu_j \mathcal{H}^{n-1}\mres \Gamma + \partial^2_{ij}( \frac{Q_0}{2} d_\Gamma) ( \chi_{\Omega''} - \chi_{\Omega'}).
\end{equation}
%where $\nu= \nu_{\Omega'}$ denotes the outward pointing unit normal for the rest of this article. 
%We recall from Lemma \ref{lem:distprep} that the signed distance function 
%\begin{equation}
%\end{equation}

\begin{lemma}[Lipschitz-regularity for suitably smooth input data]\label{lem:31}
Let $\Gamma = \partial \Omega' \in C^3$, $Q \in W^{2,s}(\Omega)$ and $A =(A_{ij}) \in W^{2,s}(\Omega; \mathbb{R}^{n\times n})$ for some $s > n$. Let $u \in L^2(\Omega)$ be the unique very weak solution to \eqref{eq:1.1}. Then $u \in W^{1,\infty}(\Omega)$. 
\end{lemma}
\begin{proof}
Let $Q,A,\Gamma$ be as in the statement. 
Let $\tilde{Q} := \frac{Q}{(A\nu, \nu)}$ and notice that $\tilde{Q} \in W^{2,s}(B_\epsilon(\Gamma))$ as $\nu$ possesses a $C^2$-extension on $B_\epsilon(\Gamma)$ (since $\nu = \nabla d_\Gamma \vert_{\Gamma}$ and $\Gamma = \partial \Omega' \in C^3$). Now we look at the following expression for arbitrary $\phi \in C_0^\infty(B_\epsilon(\Gamma))$. 
\begin{align}
    \int_{B_\epsilon(\Gamma)} & \left(A \nabla (\tfrac{1}{2}\tilde{Q} |d_\Gamma|) , \nabla \phi \right) \; \mathrm{d}x  =  - \int_{B_\epsilon(\Gamma)} \tfrac{1}{2} \tilde{Q} |d_\Gamma| \mathrm{div}(A \nabla \phi) \; \mathrm{d}x
    \\ & =  - \int_{B_\epsilon(\Gamma)} \tfrac{1}{2} \tilde{Q} |d_\Gamma| \sum_{i,j=1}^n  \partial_i (A_{ij} \partial_j \phi)   \; \mathrm{d}x
    \\ & = - \sum_{i,j=1}^n\int_{B_\epsilon(\Gamma)} \tfrac{1}{2} \tilde{Q} |d_\Gamma| \partial_i (A_{ij}) \partial_j \phi \; \mathrm{d}x - \sum_{i,j=1}^n \int_{B_\epsilon(\Gamma)} \tfrac{1}{2} A_{ij} \tilde{Q} |d_\Gamma|  \partial^2_{ij} \phi \; \mathrm{d}x 
    \\ & =  \sum_{i,j=1}^n \int_{B_\epsilon(\Gamma)} \partial_{j}\left( \tfrac{1}{2} \tilde{Q} |d_\Gamma| \partial_i A_{ij} \right) \phi \; \mathrm{d}x -
    \sum_{i,j=1}^n \int_{B_\epsilon(\Gamma)} \tfrac{1}{2} A_{ij} \tilde{Q} |d_\Gamma|  \partial^2_{ij} \phi \; \mathrm{d}x. 
\end{align}
Using Lemma \ref{lem:distprep} in the last summand (with $Q_0=  \tilde{Q} A_{ij}$) we find 
\begin{align}\label{eq:allterms}
     \int_{B_\epsilon(\Gamma)} &  \left(A \nabla (\tfrac{1}{2}\tilde{Q} |d_\Gamma|) , \nabla \phi \right) \; \mathrm{d}x  =  \sum_{i,j=1}^n \int_{B_\epsilon(\Gamma)} \partial_{j}\left( \tfrac{1}{2} \tilde{Q} |d_\Gamma| \partial_i A_{ij} \right) \phi \; \mathrm{d}x 
     \\ & - \sum_{i,j=1}^n \int_{\Gamma} A_{ij} \tilde{Q} \nu_i \nu_j  \phi \; \mathrm{d}\mathcal{H}^{n-1} - \sum_{i,j=1}^n \int_{B_\epsilon(\Gamma)} \partial^2_{ij}\left( \frac{\tilde{Q}A_{ij}}{2}d_\Gamma\right) (\chi_{\Omega'}- \chi_{\Omega''}) \phi \; \mathrm{d}x.
\end{align}
Notice next that 
\begin{equation}
    \sum_{i,j=1}^n A_{ij} \tilde{Q} \nu_i\nu_j  = (A\nu,\nu) \tilde{Q} = Q. 
\end{equation}
Further observe that Lipschitz continuity of $|d_\Gamma|$ and the prerequisites on $Q$ and $A$ imply
\begin{equation}
    \partial_{j}\left( \tfrac{1}{2} \tilde{Q} |d_\Gamma| \partial_i A_{ij} \right) \in L^s(B_\epsilon(\Gamma))
\end{equation}
as well as 
\begin{equation}
    \partial^2_{ij}\left( \frac{\tilde{Q}A_{ij}}{2}d_\Gamma\right) (\chi_{\Omega'}- \chi_{\Omega''}) \in L^s(B_\epsilon(\Gamma)).
\end{equation}
We infer from this and \eqref{eq:allterms} that for all $\phi \in C_0^\infty(B_\epsilon(\Gamma))$ one has
\begin{equation}
    \int_{B_\epsilon(\Gamma)}   \left(A \nabla (\tfrac{1}{2}\tilde{Q} |d_\Gamma|) , \nabla \phi \right) \; \mathrm{d}x = -\int_\Gamma Q \phi \; \mathrm{d}x  +  \int_{B_\epsilon(\Gamma) } g \phi \; \mathrm{d}x 
\end{equation}
for $g := \partial_{j}\left( \tfrac{1}{2} \tilde{Q} |d_\Gamma| \partial_i A_{ij} \right) -  \partial^2_{ij}\left( \frac{\tilde{Q}A_i^j}{2}d_\Gamma\right) (\chi_{\Omega'}- \chi_{\Omega''}) \in L^s(B_\epsilon(\Gamma))$. Now let $u \in L^2(\Omega)$ be the solution of \eqref{eq:1.1}. With Lemma \ref{lem:weakveryweak} and the previous formula we have that $u \in W_0^{1,2}(\Omega)$ and for each $\phi \in C_0^\infty(B_\epsilon(\Gamma))$ one has 
\begin{equation}\label{eq:argu1}
    \int_{B_\epsilon(\Gamma)} \left( A \nabla ( u + \tfrac{1}{2}\tilde{Q}|d_\Gamma|) , \nabla \phi \right) \; \mathrm{d}x = \int_{B_\epsilon(\Gamma)} g \phi \; \mathrm{d}x.
 \end{equation}
 Elliptic regularity (cf. \cite[Theorem 9.15]{GilTru}) now implies that $u + \tfrac{1}{2}\tilde{Q}|d_\Gamma| \in W^{2,s}_{loc}(B_\epsilon(\Gamma)) \subset W^{1,\infty}_{loc}(B_\epsilon(\Gamma))$. Since however also $ \tfrac{1}{2}\tilde{Q}|d_\Gamma| \in W^{1,\infty}(B_\epsilon(\Gamma))$ we infer that $u = u + \tfrac{1}{2}\tilde{Q}|d_\Gamma| - \tfrac{1}{2}\tilde{Q}|d_\Gamma| \in W^{1,\infty}_{loc}(B_\epsilon(\Gamma))$. To complete the proof we show  that $u \in W^{1,\infty}(\Omega \setminus B_\frac{\epsilon}{2}(\Gamma))$. To this end we notice that on 
 $\partial B_\frac{\epsilon}{2}(\Gamma)$ we have 
 \begin{equation}
     u = u+ \frac{1}{2}\tilde{Q} |d_\Gamma| -  \frac{1}{2}\tilde{Q} \frac{\epsilon}{2}.
 \end{equation}
 Since $u+ \frac{1}{2}\tilde{Q} |d_\Gamma| \in W^{2,s}_{loc}(B_\epsilon(\Gamma))$ we infer that $u \big\vert_{\partial B_\frac{\epsilon}{2}(\Gamma))}$ is a restriction of a $W^{2,s}$-function on $\partial B_{\frac{\epsilon}{2}}(\Gamma)$. Let now $\varphi: \partial B_{\frac{\epsilon}{2}}(\Gamma) \cup \partial \Omega \rightarrow \mathbb{R}$ be given by
 \begin{equation}
     \varphi(x) = \begin{cases}
     u(x) & x \in \partial B_\frac{\epsilon}{2}(\Gamma), \\ 0 & x \in \partial \Omega.
     \end{cases}
 \end{equation}
 Then $\varphi$ is the restriction of a function in  $W^{2,s}(\Omega \setminus B_{\frac{\epsilon}{2}}(\Gamma))$ to $\partial (\Omega \setminus B_{\frac{\epsilon}{2}}(\Gamma)) = \partial \Omega \cup \partial B_{\frac{\epsilon}{2}}(\Gamma)$.
Moreover
 $u\big\vert_{\Omega \setminus B_\frac{\epsilon}{2}(\Gamma)}$ solves  (weakly)
 \begin{equation}\label{eq:argu2}
     \begin{cases}
    -  \mathrm{div}(A \nabla u)  = 0 & \textrm{in } \Omega \setminus B_{\frac{\epsilon}{2}}(\Gamma), \\  \qquad \quad \quad \; \;  u = \varphi  & \textrm{on } \partial (\Omega \setminus B_{\frac{\epsilon}{2}}(\Gamma)).  
     \end{cases}
 \end{equation}
 This and the regularity of $A,\varphi, \Gamma$ imply (cf.  \cite[Lemma 9.15]{GilTru}) that $u \in W^{2,s}(\Omega \setminus B_\frac{\epsilon}{2}(\Gamma)) \hookrightarrow W^{1,\infty}(\Omega  \setminus B_\frac{\epsilon}{2}(\Gamma))$.
\end{proof}

\begin{remark}\label{rem:C3C2}
The prerequsite $\Gamma = \partial \Omega' \in C^3$ has been only used once in the beginning of the previous proof to deduce that $ \tilde{Q} = \frac{Q}{(A \nu , \nu)} \in W^{2,s}$. If $A = \mathrm{Id}_{n\times n}$, we have that $\tilde{Q} =Q \in W^{2,s}$ even if one has only $\Gamma = \partial \Omega' \in C^2$.
\end{remark}

The properties of the signed distance function and the findings in the previous proof also imply the announced $BV$-regularity statement.
\begin{cor}[$BV$-regularity of the gradient]
Let $\Gamma = \partial \Omega' \in C^3$, $Q \in W^{2,s}(\Omega)$ and $A \in W^{2,s}(\Omega; \mathbb{R}^{n\times n})$ for some $s > n$. Let $u \in L^2(\Omega)$ be the very weak solution to \eqref{eq:1.1}. Then $\nabla u \in BV(\Omega)$.
\end{cor}
\begin{proof}
We know from the discussion after \eqref{eq:argu2} that $u \in W^{2,s}(\Omega \setminus B_{\frac{\epsilon}{2}}(\Gamma))$, and hence in particular $\nabla u \in W^{1,s}(\Omega \setminus B_{\frac{\epsilon}{2}}(\Gamma)) \subset BV(\Omega \setminus B_{\frac{\epsilon}{2}}(\Gamma))$. Next we show that $u \in BV(B_\epsilon(\Gamma))$ which will then prove the claim. 
The discussion after \eqref{eq:argu1} yields that $u+\tfrac{1}{2}\tilde{Q} |d_\Gamma| \in W^{2,s}(B_\epsilon(\Gamma))$ for some $\tilde{Q} \in W^{2,s}(B_\epsilon(\Gamma))$. Hence $\nabla \left(u+\tfrac{1}{2}\tilde{Q} |d_\Gamma| \right) \in BV(B_\epsilon(\Gamma))$. Using that by Corollary \ref{cor:signdistBV} $\nabla \left( \tfrac{1}{2} \tilde{Q} |d_\Gamma| \right) \in BV(B_\epsilon(\Gamma))$ we find 
\begin{equation}
    \nabla u = \nabla \left( u + \tfrac{1}{2}\tilde{Q} |d_\Gamma| \right) - \nabla \left( \tfrac{1}{2} \tilde{Q} |d_\Gamma| \right) \in BV(B_\epsilon(\Gamma)). 
\end{equation}
The claim follows.
\end{proof}
\begin{remark}
For the same reason as in Remark \ref{rem:C3C2}, the prerequisite $\Gamma = \partial \Omega' \in C^3$ can be replaced by $\Gamma = \partial \Omega' \in C^2$ if $A = \mathrm{Id}_{n\times n}$. 
\end{remark}
Using the same comparison techniques one also obtains more regularity away from the interface $\Gamma$.
\begin{cor}[Further regularity away from the interface]\label{cor:33}
Let $\Gamma = \partial \Omega' \in C^3$ and $Q \in W^{2,s}(\Omega)$, $A \in W^{2,s}(\Omega;\mathbb{R}^{n\times n})$ for some $s >n$. Let $u \in L^2(\Omega)$ be the weak solution to \eqref{eq:1.1}. Then $u \in W^{2,s}(\Omega') \cap W^{2,s}(\Omega'')$.
\end{cor}
\begin{proof}
We have seen in \eqref{eq:argu1} and the discussion below that for $\tilde{Q} = \frac{Q}{(A\nu,\nu)} \in W^{2,s}(B_\epsilon(\Gamma))$ there holds
\begin{equation}
    u + \tfrac{\tilde{Q}}{2} |d_\Gamma| \in W^{2,s}(B_\epsilon(\Gamma)) \subset W^{2,s}(\Omega' \cap B_\epsilon(\Gamma)) \cap W^{2,s}(\Omega'' \cap B_\epsilon(\Gamma)) .
\end{equation}
Moreover, Lemma \ref{lem:F3} implies that 
\begin{equation}
    \tfrac{\tilde{Q}}{2} |d_\Gamma| \in W^{2,s}(\Omega' \cap B_\epsilon(\Gamma)) \cap W^{2,s}(\Omega'' \cap B_\epsilon(\Gamma)). 
\end{equation}
This implies that $u \in W^{2,s}(\Omega' \cap B_\epsilon(\Gamma)) \cap W^{2,s}(\Omega'' \cap B_\epsilon(\Gamma)) $. Since the discussion below \eqref{eq:argu2} also yields that $u \in W^{2,s}(\Omega \setminus B_{\frac{\epsilon}{2}}(\Gamma)) \subset W^{2,s}(\Omega' \setminus B_{\frac{\epsilon}{2}}(\Gamma)) \cap W^{2,s}(\Omega'' \setminus B_{\frac{\epsilon}{2}}(\Gamma))$ the claim follows. 
\end{proof}

\subsection{Blow-up arguments}
 
We have observed in the last section that for suitably regular data $Q,A,\Gamma$ solutions of \eqref{eq:1.1} lie in $W^{1,\infty}(\Omega)$ and their gradients lie in $BV(\Omega)$. By Lemma \ref{lem:notC1}, better global regularity can not be expected. However on $\Omega'$ and $\Omega''$, solutions are $A$-harmonic and therefore more regular (cf. Corollary \ref{cor:33}).  Hence there must occur a \emph{loss of regularity} on $\Gamma$. In this section we use blow-up techniques to understand the behavior on $\Gamma$ precisely.

\begin{lemma}[A Taylor-type expansion]\label{lem:Taylor}
Suppose that $\Gamma = \partial \Omega' \in C^1,$ $A,Q \in C^{0,\alpha}(\Omega)$ for some $\alpha > 0$ and $u \in L^2(\Omega)$ is a solution of \eqref{eq:1.1}. If $u \in W^{1,\infty}(\Omega)$ and $\nabla u \in BV(\Omega)$ then for $\mathcal{H}^{n-1}$ a.e. $x_0 \in \Gamma$  there exists some $\theta(x_0) \in \mathbb{R}^n$ such that 
\begin{equation}
    u(x)= u(x_0) + (\theta(x_0), x-x_0) - \tfrac{1}{2}\tfrac{Q(x_0)}{ (A(x_0) \nu(x_0), \nu(x_0))} |(x-x_0,\nu(x_0))| + o(|x-x_0|). 
\end{equation}
Moreover, one has (for $\mathcal{H}^{n-1}$ a.e. $x_0 \in \Gamma$)
\begin{equation}\label{eq:meanvaluelimit}
    \theta(x_0) = \lim_{r \rightarrow 0} \fint_{B_r(x_0)} \nabla u \; \mathrm{d}x. 
\end{equation}
In particular, the map $\theta: \Gamma \rightarrow \mathbb{R}^n$, $x_0 \mapsto \theta(x_0)$ is Borel measurable and lies in $L^\infty(\Gamma)$.
\end{lemma}
\begin{proof}
Define for $r>0$ the function $u_r: \mathbb{R}^n \rightarrow \mathbb{R}^n$ via
\begin{equation}
    u_r(x) := \begin{cases}
        \frac{u(x_0+rx) - u(x_0)}{r} & x_0 + r x \in \Omega, \\ -\frac{u(x_0)}{r} &  \textrm{otherwise}.
    \end{cases}
\end{equation}
Observe that $u_r$ is Lipschitz continuous for all $r>0$, $u_r(0)=0$ for all $r> 0$ and $||\nabla u_r||_{L^\infty(\mathbb{R}^n)} \leq || \nabla u ||_{L^\infty(\Omega)}$. We conclude by Lemma \ref{lem:B6} that for each sequence $r_j \rightarrow 0$ there exists a subsequence $r_{l_j} \rightarrow 0$ and some $\bar{u} \in W^{1,\infty}_{loc}(\mathbb{R}^n)$ with $\nabla \bar{u} \in L^\infty(\mathbb{R}^n)$ such that $u_{r_{l_j}} \rightarrow \bar{u}$ locally uniformly and $\nabla u_{r_{l_j}} \rightarrow \nabla \bar{u}$ weakly in $W^{1,2}(B_R(0))$ for all $R > 0$. 
As a shorthand notation we write $\bar{r}_j = r_{l_j}$. 
We next describe $\bar{u}$ by means of the equation it solves. For arbitrary $\phi \in C_0^\infty(\mathbb{R}^n)$ there holds 
\begin{align}
     \int_{\mathbb{R}^n}  &(A(x_0) \nabla \bar{u}(x), \nabla \phi(x) ) \; \mathrm{d}x    = \lim_{j \rightarrow \infty} \int_{\mathbb{R}^n} (A(x_0) \nabla u_{\bar{r}_j}(x)  ,\nabla \phi(x))  \; \mathrm{d}x
    \\ &  = \lim_{j \rightarrow \infty} \int_{x_0 + \bar{r}_j x \in \Omega } (A(x_0) \nabla u(x_0 + \bar{r}_j x) , \nabla \phi(x)  ) \; \mathrm{d}x
    \\ & = \lim_{j \rightarrow \infty} \left( \int_{x_0 + \bar{r}_j x \in \Omega } (A(x_0+\bar{r}_j x) \nabla u(x_0 + \bar{r}_j x) , \nabla \phi(x)  ) \; \mathrm{d}x \right. \\ & \qquad \quad  \left. + \int_{x_0 + \bar{r}_j x \in \Omega } ((A(x_0)- A(x_0+\bar{r}_j x)) \nabla u(x_0 + \bar{r}_j x) , \nabla \phi(x)  ) \; \mathrm{d}x \right) .
\end{align}
The absolute value of the second summand is now estimated by 
\begin{equation}
    [A]_{C^{0,\alpha}} \bar{r}_j^\alpha \mathrm{diam}(\mathrm{spt}(\phi))^\alpha ||\nabla u||_{L^\infty} \int_{\mathbb{R}^n}|\nabla \phi| \; \mathrm{d}x = o(1) \quad (j \rightarrow \infty).
\end{equation}
Hence we obtain 
\begin{align}
    \int_{\mathbb{R}^n}  &(A(x_0) \nabla \bar{u}(x), \nabla \phi(x) ) \; \mathrm{d}x  =  \lim_{j \rightarrow \infty} \int_{x_0 + \bar{r}_j x \in \Omega } (A(x_0+\bar{r}_j x) \nabla u(x_0 + \bar{r}_j x) , \nabla \phi(x)  ) \; \mathrm{d}x
    \\ & = \lim_{j \rightarrow \infty} \frac{1}{\bar{r}_j^n} \int_{\Omega } (A(z) \nabla u(z) , \nabla \phi(\tfrac{z-x_0}{\bar{r}_j})  ) \; \mathrm{d}z  = \lim_{j \rightarrow \infty} \frac{1}{\bar{r}_j^{n-1}} \int_{\Omega } (A(z) \nabla u(z) , \nabla \phi_{\bar{r}_j}(z)  ) \; \mathrm{d}z,
\end{align}
where $\phi_{\bar{r}_j}(z) := \phi(\frac{z-x_0}{\bar{r}_j})$. Using Lemma \ref{lem:weakveryweak} we obtain
\begin{align}
    \int_{\mathbb{R}^n}  &(A(x_0) \nabla \bar{u}(x), \nabla \phi(x) ) \; \mathrm{d}x  =  \lim_{j \rightarrow \infty} \frac{1}{\bar{r}_j^{n-1}}\int_{\Gamma}  Q \phi_r \; \mathrm{d}\mathcal{H}^{n-1}(x)
    \\ & = \lim_{j \rightarrow \infty} \frac{1}{\bar{r}_j^{n-1}}\int_{\Gamma}  Q(x) \phi(\tfrac{x-x_0}{\bar{r}_j}) \; \mathrm{d}\mathcal{H}^{n-1}(x)
     = \lim_{j \rightarrow \infty} \int_{\frac{\Gamma-x_0}{\bar{r}_j}}  Q(x_0+ \bar{r}_j z) \phi(z) \; \mathrm{d}\mathcal{H}^{n-1}(z)
    \\ & = \lim_{j \rightarrow \infty} \left(\int_{\frac{\Gamma-x_0}{\bar{r}_j}} Q(x_0) \phi(z) \; \mathrm{d}\mathcal{H}^{n-1}(z) - \int_{\frac{\Gamma-x_0}{\bar{r}_j}} (Q(x_0)-Q(x_0+\bar{r}_jz)) \phi(z) \; \mathrm{d}\mathcal{H}^{n-1}(z) \right).
\end{align}
The absolute value of the last summand is bounded by 
\begin{equation}
    [Q]_{C^{0,\alpha}} \bar{r}_j^\alpha \mathrm{diam}(\mathrm{spt}(\phi))^\alpha \int_{\frac{\Gamma - x_0}{\bar{r}_j}} |\phi(z)| \; \mathrm{d}\mathcal{H}^{n-1}(z) = o(1) \quad (j \rightarrow \infty),
\end{equation}
where we used that for all $\psi \in C_0^0(\mathbb{R}^n)$ there holds
\begin{equation}
    \int_{\frac{\Gamma - x_0}{\bar{r}_j}} \psi(z) \; \mathrm{d}\mathcal{H}^{n-1}(z) \rightarrow \int_{T_{x_0}\Gamma} \psi(y) \; \mathrm{d}\mathcal{H}^{n-1}(y) \quad (j \rightarrow \infty). 
\end{equation}
Using this once more we obtain 
\begin{equation}
     \int_{\mathbb{R}^n}  (A(x_0) \nabla \bar{u}(x), \nabla \phi(x) ) \; \mathrm{d}x = \lim_{j \rightarrow \infty}  Q(x_0) \int_{\frac{\Gamma-x_0}{\bar{r}_j}}   \phi \; \mathrm{d}\mathcal{H}^{n-1} = Q(x_0) \int_{T_{x_0}\Gamma} \phi \; \mathrm{d}\mathcal{H}^{n-1}.
\end{equation}
As a consequence, $\bar{u} \in W^{1,\infty}_{loc}(\mathbb{R}^n)$ is a weak solution of 
\begin{equation}\label{eq:weakblowup}
    \begin{cases}
        -\mathrm{div}(A(x_0) \nabla \bar{u}) = Q(x_0) \mathcal{H}^{n-1}\mres T_{x_0} \Gamma  & \textrm{on } \mathbb{R}^n, \\ \bar{u}(0) = 0 , \nabla{\bar{u}} \in L^\infty(\mathbb{R}^n). 
    \end{cases}
\end{equation}
We claim next that there exists some $\theta(x_0) \in \mathbb{R}^n$ such that
\begin{equation}
    \bar{u}(z) = (\theta(x_0),z) -  \tfrac{1}{2} \tfrac{Q(x_0)}{(A(x_0) \nu(x_0), \nu(x_0))} |(\nu(x_0), z)| .
\end{equation}
To this end we first compute that $v(z) :=  -\tfrac{1}{2} \tfrac{Q(x_0)}{(A(x_0) \nu(x_0), \nu(x_0))}  |(\nu(x_0), z)|$ defines a weak solution of \eqref{eq:weakblowup}. One readily checks that
\begin{equation}
    \nabla v(z)  = -\tfrac{1}{2} \tfrac{Q(x_0)}{(A(x_0) \nu(x_0), \nu(x_0))}  \mathrm{sgn}(\nu(x_0), z) \nu(x_0) .
\end{equation}
With this formula we compute for each $\phi \in C_0^\infty( \mathbb{R}^n)$ (with the shorthand notation $B := \tfrac{1}{2} \tfrac{Q(x_0)}{(A(x_0) \nu(x_0), \nu(x_0))}$)
\begin{align}
   & \int_{\mathbb{R}^n} ( A(x_0) \nabla v, \nabla \phi) \; \mathrm{d}x   = - B \int_{T_{x_0}\Gamma} \int_{\mathbb{R}}  \mathrm{sgn}( \nu(x_0), w + t \nu(x_0) ) (A(x_0) \nu(x_0), \nabla \phi) \; \mathrm{d}t \; \mathrm{d}w \\ &
   = -B \int_{T_{x_0}\Gamma} \int_{\mathbb{R}}  \mathrm{sgn}(t) (A(x_0) \nu(x_0), \nabla \phi) \; \mathrm{d}t \; \mathrm{d}w 
   \\ & = - B\left( \int_{{T_{x_0}\Gamma} + (0,\infty)\nu(x_0)} \mathrm{div}( \phi A(x_0) \nu(x_0))  \; \mathrm{d}x  -  \int_{{T_{x_0}\Gamma} + (-\infty,0)\nu(x_0)} \mathrm{div}( \phi A(x_0) \nu(x_0))  \; \mathrm{d} x \right) \\ & 
   = 2B  \int_{T_{x_0}\Gamma} \phi ( A(x_0) \nu(x_0), \nu(x_0)) \; \mathrm{d}\mathcal{H}^{n-1}   =   Q(x_0) \int_{T_{x_0}\Gamma} \phi \; \mathrm{d}\mathcal{H}^{n-1}.
\end{align}
This and \eqref{eq:weakblowup} yield that  $ w := \bar{u}+ \tfrac{1}{2}\tfrac{Q(x_0)}{(A(x_0) \nu(x_0) , \nu(x_0) )}|(\nu(x_0), \cdot )|$ solves (weakly) 
\begin{equation}\label{eq:equationw}
    \begin{cases}
        -\mathrm{div}(A(x_0) \nabla w) =  0 & \textrm{in } \mathbb{R}^n, \\ w(0) = 0 , \nabla w  \in L^\infty(\mathbb{R}^n). 
    \end{cases}
\end{equation}
We claim that this implies that there exists $\theta(x_0) \in \mathbb{R}^n$ such that $w(z) = (\theta(x_0),z)$ for all $z \in \mathbb{R}^n$. Observe first that by elliptic regularity each solution lies in $C^\infty(\mathbb{R}^n)$. Differentiating the equation one sees that for all $i = 1,...,n$ the map $\partial_i w$ satisfies  $-\mathrm{div}(A(x_0) \nabla (\partial_i w)) = 0$. Moreover, \eqref{eq:equationw} yields that $\partial_i w$ is bounded on $\mathbb{R}^n$. Liouville's theorem for elliptic operators (with constant coefficents, cf. \cite{Priola}),  implies now that $\partial_i w$ is constant. We conclude that $\nabla w$ is constant and therefore $w(z) = (\theta(x_0), z) + D$ for some $\theta(x_0) \in \mathbb{R}^n$ and some $D \in \mathbb{R}$. Notice however that $D= 0$ as $w(0)= 0$. We infer from the definiton of $w$ that 
\begin{equation}\label{eq:balu}
    \bar{u}(x) = (\theta(x_0),x) - \tfrac{1}{2} \tfrac{Q(x_0)}{(A(x_0) \nu(x_0) , \nu(x_0))} | ( \nu(x_0) , x) | .
\end{equation}
%Recalling the definition of $\bar{u}$ we have that for each sequence $r_j \rightarrow 0$ there exists a subsequence $r_{l_j} \rightarrow 0$ such that 
%\begin{equation}
 % \lim_{j \rightarrow \infty}  \frac{u(x_0 + r_{l_j} x) - u(x_0)}{r_{l_j}} = (\theta(x_0),x) - \tfrac{1}{2} \tfrac{Q(x_0)}{(A(x_0) \nu(x_0) , \nu(x_0))} | ( \nu(x_0) , x) |  
%\end{equation}
%locally uniformly in $\mathbb{R}^n$. 
We remark that $\theta(x_0)$ could in our construction depend on the chosen subsequence $(\bar{r}_j)_{j \in \mathbb{N}}$. We will next show that $\theta(x_0)$ is actually independent of the chosen subsequence for $\mathcal{H}^{n-1}$ a.e. $x_0 \in \Gamma$. This would then imply that $u_r \rightarrow \bar{u}$  locally uniformly as $r \rightarrow 0$ for $\mathcal{H}^{n-1}$ a.e. $x_0 \in \Gamma$. 
%In order to show the asserted Taylor expansion one needs to ensure independence of the chosen subsequence. If we can show that $\theta(x_0)$ does not depend on the subsequence $\bar{r}_j$ then we have obtained that each subsequence $(r_j)_{j \in \mathbb{N}}$ has a subsequence $(\bar{r}_j)_{j \in \mathbb{N}}$ such that $u_{\bar_{r}_j}$ converges uniformly to $\bar{u}$ given as in  \eqref{eq:balu}, implying that $u_r \rightarrow \bar{u}$ uniformly as $r \rightarrow 0$.
%To this end we show \eqref{eq:meanvaluelimit}. 

\textbf{Intermediate claim.} We claim that 
\begin{equation}
    \theta(x_0) = \lim_{j \rightarrow \infty}  \fint_{B_{\bar{r}_j}(x_0)} \nabla u \; \mathrm{d}x.
\end{equation}
Since by \cite[Theorem 5.20]{EvGar} for $\mathcal{H}^{n-1}$ a.e. $x_0 \in \Gamma$ there exists 
\begin{equation}
    \lim_{r \rightarrow 0} \fint_{B_r(x_0)} \nabla u \; \mathrm{d}x 
\end{equation}
the intermediate claim will imply that for $\mathcal{H}^{n-1}$ almost every $x_0 \in \Gamma$ the value $\theta(x_0) $ is independent of the chosen subsequence and 
\begin{equation}
    \theta(x_0) = \lim_{r \rightarrow 0} \fint_{B_r(x_0)} \nabla u \; \mathrm{d}x  \quad \textrm{for } \mathcal{H}^{n-1} \;  \textrm{a.e.} \; x_0 \in \Gamma. 
\end{equation}
To show the intermediate claim we compute with the Gauss divergence theorem defining $\alpha_n := |B_1(0)|$
\begin{align}
     &\fint_{B_{\bar{r}_j}(x_0)} \nabla u   \; \mathrm{d}x  = \frac{1}{\alpha_n\bar{r}_j^n} \int_{B_{\bar{r}_j}(x_0)} \nabla (u-u(x_0)) \; \mathrm{d}x  \\ & = \frac{1}{\alpha_n\bar{r}_j^n} \int_{\partial B_{\bar{r}_j}(x_0)} (u(x)-u(x_0)) \cdot \frac{x-x_0}{\bar{r}_j} \; \mathrm{d}\mathcal{H}^{n-1}
    \\ & = \frac{1}{\alpha_n\bar{r}_j} \int_{\partial B_1(0)} (u(x_0 + \bar{r}_j z )-u(x_0)) \cdot z \; \mathrm{d}\mathcal{H}^{n-1}
    \\ &  = \frac{1}{\alpha_n}\int_{\partial B_1(0)} \left( \frac{u(x_0 + \bar{r}_j z )-u(x_0)}{\bar{r}_j} \right) \cdot z \; \mathrm{d}\mathcal{H}^{n-1} \underset{j \rightarrow  \infty }{\longrightarrow} \frac{1}{\alpha_n}\int_{\partial B_1(0)} \bar{u}(z) \cdot z \; \mathrm{d}\mathcal{H}^{n-1}
    \\ & = \frac{1}{\alpha_n}\int_{\partial B_1(0)} \left( ( \theta(x_0), z) - \frac{1}{2} \tfrac{Q(x_0)}{(A(x_0) \nu(x_0), \nu(x_0))} |( \nu(x_0),z)| \right) \cdot z \; \mathrm{d}\mathcal{H}^{n-1}.
\end{align}
Notice that $z \mapsto \frac{1}{2} \tfrac{Q(x_0)}{(A(x_0) \nu(x_0), \nu(x_0))} |( \nu(x_0),z)| $ is even with respect to the transformation $z \mapsto -z$. Since however $z \mapsto z$ is odd with respect to this transformation and $-\partial B_1(0) = \partial B_1(0)$ we obtain that the second summand vanishes. Therefore
\begin{equation}
    \lim_{j \rightarrow \infty } \fint_{B_{\overline{r}_j}(x_0)} \nabla u \; \mathrm{d}x = \frac{1}{\alpha_n}\int_{\partial B_1(0)} (\theta(x_0),z) \cdot z = \frac{1}{\alpha_n} \int_{B_1(0)} \nabla (\theta(x_0),z) \; \mathrm{d}\mathcal{H}^{n-1}(z) = \theta(x_0). 
\end{equation}
As explained above this shows independence of $\theta(x_0)$ of the chosen subsequence for $\mathcal{H}^{n-1}$ a.e. $x_0 \in \Gamma$ and this yields that for $\mathcal{H}^{n-1}$ a.e. $x_0 \in \Gamma$ one has (locally uniformly in $r$)
\begin{equation}
    \lim_{r \rightarrow 0 }  u_r(x) = (\theta(x_0), x) - \tfrac{1}{2} \tfrac{Q(x_0)}{(A(x_0) \nu(x_0), \nu(x_0)} |(\nu(x_0) , x)| .
\end{equation}
Since $u_r(x) = \frac{u(x_0 + rx)- u(x_0)}{r}$ for all $x$ close to $x_0$ we conclude
\begin{equation}
    u(x_0 + r x)  = u(x_0) + (\theta(x_0), rx) -\tfrac{1}{2} \tfrac{Q(x_0)}{(A(x_0) \nu(x_0), \nu(x_0)} |(\nu(x_0) , rx)| + o(r),
\end{equation}
whereupon $y := x_0 + rx$ yields the desired formula
\begin{equation}
    u(y) = u(x_0) + (\theta(x_0),y-x_0) -  \tfrac{1}{2} \tfrac{Q(x_0)}{(A(x_0) \nu(x_0), \nu(x_0)} |(\nu(x_0) , y-x_0)| + o(|y-x_0|). 
\end{equation}
\end{proof}
This blow-up result enables us to look at the \emph{boundary trace} of $\nabla u$ on $\Gamma$. This boundary trace result becomes important for the study of $||\nabla u||_{L^\infty(\Omega)}$ with the aid of \emph{maximum principles}.
%We remark that the trace  is defined only since $\nabla u$ lies in $BV(\Omega)$. If we had only given that $\nabla u \in L^\infty(\Omega)$ then we would not be able to form boundary traces. 
%This shows once again that our comparison with the signed distance function was advantegeous.
We recall from \cite[Theorem 2.10]{Giusti} that for each $f \in BV(D)$   one has %$\mathcal{H}^{n-1}$ a.e. on $\partial D$ 
\begin{equation}\label{eq:Giusti}
    \mathrm{tr}_{\partial D}( f) (z)   = \lim_{r \rightarrow 0} \fint_{B_r(z) \cap D} f(y) \; \mathrm{d}y \quad \textrm{for $\mathcal{H}^{n-1}$ a.e. $z \in \partial D$.}
\end{equation}
This notion is also consistent with the notion of traces of Sobolev functions in the sense that if $f \in W^{1,1}(D)$ then the $BV$-trace coincides with the classical Sobolev trace of $f$. 

\begin{lemma}[Boundary traces of the gradient]\label{lem:traces}
Suppose that $\Gamma = \partial \Omega' \in C^1$, $A,Q \in C^{0,\alpha}$ and $u \in L^2(\Omega)$ is a very weak solution of \eqref{eq:1.1}. Let $\theta: \Gamma \rightarrow \mathbb{R}$ be as in Lemma \ref{lem:Taylor}. If $u \in W^{1,\infty}(\Omega)$ and $\nabla u \in BV(\Omega)$ then $\mathcal{H}^{n-1}$ a.e. on $\Gamma$ one has
\begin{equation}
    \mathrm{tr}_{\Omega'}(\nabla u ) =  \theta + \tfrac{1}{2} \tfrac{Q}{(A \nu, \nu) } \nu ,
\end{equation}
\begin{equation}
    \mathrm{tr}_{\Omega''}(\nabla u ) =  \theta - \tfrac{1}{2} \tfrac{Q}{(A \nu, \nu) } \nu.
\end{equation}
%where $\theta$ is as in Lemma \ref{lem:Taylor}. 
\end{lemma}
\begin{proof}
We only show the first formula, the second formula is analogous. The main tool we use is \eqref{eq:Giusti}.
%By Lemma (TODO) one has for $i =1,...,n$ and $\mathcal{H}^{n-1}$ a.e. $z \in \Gamma$
%\begin{equation}
  %  \mathrm{tr}_{\Omega'}(\partial_i u) (z) = \lim_{r \rightarrow 0} \fint_{B_r(z) \cap \Omega'} \partial_i u(x) \; \mathrm{d}x 
%\end{equation}
Using the Gauss divergence theorem (which is allowed as by \cite[Proposition 2.5.4]{Carbone} $\Omega' \cap B_r(z)$ is a Lipschitz domain) we find
\begin{align}
    & \lim_{r \rightarrow 0} \fint_{B_r(z) \cap \Omega'} \partial_i u(x) \; \mathrm{d}x   = \lim_{r \rightarrow 0} \tfrac{1}{|B_r(z) \cap \Omega'| } \int_{B_r(z) \cap \Omega'} \partial_i u(x) \; \mathrm{d}x  \\ & = \lim_{r \rightarrow 0} \tfrac{1}{|B_r(z) \cap \Omega'| } \int_{\partial (B_r(z) \cap \Omega') }  u(x) \nu_i(x)  \; \mathrm{d}\mathcal{H}^{n-1}(x)  
    \\ & = \lim_{r \rightarrow 0} \tfrac{1}{|B_r(z) \cap \Omega'| }  \left( \int_{\partial (B_r(z) \cap \Omega') }  (u(x)-u(z))  \nu_i(x)  \; \mathrm{d}\mathcal{H}^{n-1}(x)  + u(z) \int_{\partial (B_r(z) \cap \Omega') }   \nu_i \; \mathrm{d}\mathcal{H}^{n-1} \right)
    \\ & = \lim_{r \rightarrow 0} \tfrac{1}{|B_r(z) \cap \Omega'| }  \left( \int_{\partial (B_r(z) \cap \Omega') }  (u(x)-u(z))  \nu_i(x)  \; \mathrm{d}\mathcal{H}^{n-1}(x)  + u(z) \int_{ (B_r(z) \cap \Omega') }   \partial_i (1) \; \mathrm{d}\mathcal{H}^{n-1} \right)
    \\ & = \lim_{r \rightarrow 0} \tfrac{1}{|B_r(z) \cap \Omega'| }  \int_{\partial (B_r(z) \cap \Omega') }  (u(x)-u(z))  \nu_i(x)  \; \mathrm{d}\mathcal{H}^{n-1}(x)
    \\ & =  \lim_{r \rightarrow 0} \tfrac{1}{|B_r(z) \cap \Omega'| }  \int_{\partial (B_r(z) \cap \Omega') } [(\theta(z),x-z) - \tfrac{1}{2} \tfrac{Q(z)}{(A(z) \nu(z), \nu(z))} |(\nu(z),x-z)| \\ & \qquad \qquad \qquad\qquad\qquad \qquad \qquad \qquad\qquad\qquad \qquad \qquad \; \; \;   + o(|x-z|) ]  \nu_i(x)  \; \mathrm{d}\mathcal{H}^{n-1}(x).
\end{align}
Now note that 
\begin{equation}
    \lim_{r \rightarrow 0} \tfrac{1}{|B_r(z) \cap \Omega'|} \left\vert \int_{\partial (B_r(z) \cap \Omega') } o(|x-z|) \nu_i(x) \; \mathrm{d}\mathcal{H}^{n-1}(x) \right\vert \leq \limsup_{r \rightarrow 0} o(1) \frac{r \mathcal{H}^{n-1}(\partial (B_r(z) \cap \Omega'))}{|B_r(z) \cap \Omega'|}.
\end{equation}
Arguing as in  \eqref{eq:growthprop} we obtain
\begin{equation}
    \mathcal{H}^{n-1}(\partial (\Omega' \cap B_r(z))) \leq \mathcal{H}^{n-1} ( \partial \Omega' \cap B_r(z) ) + \mathcal{H}^{n-1}(\Omega' \cap B_r(z)) \leq C([\partial \Omega']_{0,1})  r^{n-1} + \omega_n r^{n-1} 
\end{equation}
and thus
\begin{align}
     \lim_{r \rightarrow 0} \tfrac{1}{|B_r(z) \cap \Omega'|} \int_{\partial (B_r(z) \cap \Omega') } o(|x-z|) \nu_i(x) \; \mathrm{d}\mathcal{H}^{n-1}(x) & \leq C\limsup_{r\rightarrow 0} \frac{o(1) r^n}{|B_r(z) \cap \Omega'| }  \\ & =  C\limsup_{r\rightarrow 0} o(1) \frac{1 }{|B_1(0) \cap \tfrac{\Omega'- z}{r}| } = 0,
\end{align}
where we used that $\frac{1 }{|B_1(0) \cap \tfrac{\Omega'- z}{r}| }$ is uniformly bounded in $r$ by \cite[Theorem 5.13]{EvGar}. We obtain therefore that 
\begin{align}
       & \mathrm{tr}_{\Omega'}(\partial_i u)(z)  \label{eq:quot} \\  & = \lim_{r \rightarrow 0} \tfrac{1}{|B_r(z) \cap \Omega'| }  \int_{\partial (B_r(z) \cap \Omega') } [(\theta(z),x-z) - \tfrac{1}{2} \tfrac{Q(z)}{(A(z) \nu(z), \nu(z))} |(\nu(z),x-z)| ] \nu_i(x) \; \mathrm{d}\mathcal{H}^{n-1}(x) \\
    & = \lim_{r \rightarrow 0} \tfrac{1}{|B_r(z) \cap \Omega'| } \int_{B_r(z) \cap \Omega' } \partial_{x_i}  [\theta(z),x-z) - \tfrac{1}{2} \tfrac{Q(z)}{(A(z) \nu(z), \nu(z))} |(\nu(z),x-z)| ] \; \mathrm{d}x
    \\ & = \lim_{r \rightarrow 0} \tfrac{1}{|B_r(z) \cap \Omega'| } \int_{B_r(z) \cap \Omega' } [\theta_i(z)  - \tfrac{1}{2} \tfrac{Q(z)}{(A(z) \nu(z), \nu(z))} \mathrm{sgn}(\nu(z), x-z) \nu_i(z) ] \; \mathrm{d}x 
    \\ & =  \theta_i(z) - \tfrac{1}{2}\tfrac{Q(z)}{(A \nu(z) , \nu(z))} \nu_i(z) \lim_{r \rightarrow 0} \frac{|B_r(z) \cap \Omega' \cap (z+H^+)|- |B_r(z) \cap \Omega' \cap (z+H^-)|}{|B_r(z) \cap \Omega'| }, 
\end{align}
where $H^{\pm} =\{y \in \mathbb{R}^n : \pm (y, \nu(z)) \geq 0 \}$. By 
\cite[Theorem 5.13]{EvGar} we have
\begin{equation}
  \lim_{r\rightarrow 0} \frac{1}{\alpha_n r^n} |B_r(z) \cap \Omega' \cap (z+H^{\pm})| =  \lim_{r \rightarrow 0} \frac{1}{\alpha_n} |B_1(0) \cap H^{\pm} \cap \tfrac{\Omega'- z}{r} | =  \begin{cases}
     0 & \textrm{Case '$+$'}, \\ \frac{1}{2} & \textrm{Case '$-$'},
   \end{cases}
\end{equation}
as well as
\begin{equation}
   \lim_{r\rightarrow 0} \frac{1}{\alpha_n r^n} |B_r(z) \cap \Omega'|  = \frac{1}{2}.
\end{equation}
Dividing and multiplying the quotient in the previous computation by $\alpha_n r^n$ we infer from \eqref{eq:quot}
\begin{equation}
    \mathrm{tr}_{\Omega'}(\partial_i u)(z) = \theta_i(z)  + \tfrac{1}{2}\tfrac{Q(z)}{(A \nu(z) , \nu(z))} \nu_i(z).
\end{equation}
This proves the claim. 
\end{proof}
\begin{remark}\label{rem:4.3}
The previous lemma reproduces a very classical result from potential theory, to be found in \cite[Eq. (14.15)]{Miranda}. For $\Gamma = \partial \Omega' \in C^{1,\alpha}$ and $Q \in C^{0,\alpha}$ one can define the \emph{single layer potential} $\mathcal{P}: \Omega \rightarrow \mathbb{R}$ given by
\begin{equation}
    \mathcal{P}(x) := \int_\Gamma F(x-y) Q(y) \; \mathrm{d}\mathcal{H}^{n-1}(y) \quad \textrm{with $F$ being the fundamental solution of $-\Delta$.} 
\end{equation}
 The result is that for all $x_0 \in \Gamma$ there exist $\partial^{\pm}_{\nu} \mathcal{P}(x_0) := \lim_{t \rightarrow 0 \pm } \frac{\mathcal{P}(x_0 + t \nu(x_0) )- \mathcal{P}(x_0)}{t}$ and one has the following \emph{normal jump formula}
%and $\partial^-_{\nu} \mathcal{P}(x_0) := \lim_{x \rightarrow x_0, x_0 \in \Omega'} ( \nabla \mathcal{P}(x), \nu(x_0))$ and 
\begin{equation}\label{eq:normalju}
    \partial_{\nu}^+\mathcal{P}(x_0) -   \partial_{\nu}^-\mathcal{P}(x_0)  = - Q(x_0). 
\end{equation}
A similar result can be also obtained if one replaces $-\Delta$ with an elliptic operator $-\mathrm{div}(A\nabla(\cdot))$, provided that $A$ is smooth enough.
The observation of \eqref{eq:normalju} is related to our findings in Lemma \ref{lem:traces}.  
If $A = \mathrm{Id}_{n\times n}$ the solution $u$ of \eqref{eq:1.1} is (cf. Proposition \ref{prop:Green}) given by 
\begin{equation}
    u(x)  = \int_\Gamma G_\Omega(x,y) Q(y) \; \mathrm{d}\mathcal{H}^{n-1}(y),
\end{equation}
where $G_\Omega$ denotes Green's function for $-\Delta$. Since $G_\Omega(x,y)$ and $F(x-y)$ coincide up to a smooth function in $\Omega$ we obtain that $u = \mathcal{P} + g$ for some $g \in C^\infty(\Omega)$. In particular 
\begin{equation}
    \partial_\nu^+ u(x_0) - \partial_\nu^- u(x_0) = \partial_\nu^+ \mathcal{P}(x_0) - \partial_\nu^- \mathcal{P}(x_0) = - Q(x_0).
\end{equation}
If one now replaces  $\partial_{\nu}^+ u$ by $(\mathrm{tr}_{\Omega''}(\nabla u), \nu)$ and $\partial_{\nu}^- u$ by $(\mathrm{tr}_{\Omega'}(\nabla u), \nu)$ one can obtain the same normal jump with the formulas from Lemma \ref{lem:traces}. 
\end{remark}
The previous remark says that the behavior of \emph{normal derivatives} of solutions was previously understood by means of potential theory. The upshot in our new approach is that we also get an explicit understanding of the \emph{tangential derivatives}, by means of the function $\theta$ obtained in Lemma \ref{lem:Taylor}. This helps us also characterize the regularity of the tangential derivatives. %Indeed, for suitably smooth input data one can find 
\begin{cor}
Suppose that $\Gamma = \partial \Omega' \in C^3$, $Q \in W^{2,s}(\Omega)$ and $A \in W^{2,s}(\Omega;\mathbb{R}^n)$ for some $s >n$. Then $\theta$ possesses an extension in $W^{1,s}(\mathbb{R}^n)$. 
\end{cor}
\begin{proof}
By Corollary \ref{cor:33} one has that $u \in W^{2,s}(\Omega') \cap W^{2,s}(\Omega'')$ which implies that $\nabla u \in W^{1,s}(\Omega') \cap W^{1,s}(\Omega'')$. In particular $\nabla u \vert_{\overline{\Omega'}}$ possesses an extension in $W^{1,s}(\mathbb{R}^n)$ (as $\Omega'$ is smooth enough for the existence of an extension operator). 
Now observe that on $ \Gamma = \partial \Omega'$ there holds by Lemma \ref{lem:traces}
\begin{equation}
   \theta =  \nabla u \vert_{\overline{\Omega'}}  - \tfrac{1}{2} \tfrac{Q}{(A \nu, \nu)} \nu. 
\end{equation}
Since $\nu = \nabla d_\Gamma \vert_{\Gamma}$ possesses a $C^2$-extension to $\mathbb{R}^n$ and $Q$ possesses a $W^{2,s}$-extension the claim follows.
%and $\nabla u \vert_{\Omega'}$ possesses (as discussed above) an extension in $W^{1,s}(\mathbb{R}^n)$. The claim follows. 
\end{proof}

\subsection{A priori bounds}

In the following section we use the results of the previous blow-up analysis to bound $||\nabla u||_{L^\infty(\Omega)}$ in terms of our data $Q,A, \Gamma$, which are still assumed appropriately smooth in this section. We will however check carefully that the constants in our a priori estimates on $||\nabla u||_{L^\infty(\Omega)}$ depend only on $[\Gamma]_{1,\alpha}$. 
%This regularity is namely to low to apply elliptic regularity results in $\Omega'$ and $\Omega''$. 
Notice in particular that this means that standard elliptic regularity estimates should not be used as the constants would then depend on $[\Gamma]_{2,0}$. However,
 for equations with (left -and right hand side) in divergence form, some regularity estimates hold true with constants only depending on $[\Gamma]_{1,0}$, cf.  \cite[Eq. (3)]{Auscher}. 

A second key ingredient for our estimates is the \emph{elliptic maximum principle}, which we however want to apply to the \emph{differentiated equation} to obtain estimates for $\nabla u$. The problem here is that one needs to investigate in what sense $\nabla u$ actually solves the differentiated equation: Since $\nabla u$ is a priori only $BV$-regular, one has to investigate carefully in what way differentiation of the equation makes sense. 
%, the differentiated equation may not hold true in the sense of weak solutions but only in a distributional sense, cf. Definition \ref{def:BVsol}. 
%It is a priori unclear whether this notion of solution is good enough to obtain a maximum principle at all. 
Further, there will appear very weak concepts of solutions, so called \emph{BV-solutions}, cf. Definition \ref{def:BVsol}. For these types of solutions it needs to be carefully checked, whether an elliptic maximum principle is at all available. We have dedicated Appendix \ref{app:maxpr} to this question. 

%Next we show how to obtain an a priori bound for $||\nabla u||_{L^\infty(\Omega)}$. 

%so that $\nabla u \in L^\infty(\Omega)$

\begin{lemma}\label{lem:apriobound}
Suppose that $\Gamma = \partial \Omega' \in C^\infty$, $Q \in W^{1,q}(\Omega)$ and  $A \in W^{1,q}(\Omega; \mathbb{R}^{n\times n})$  for some $q >n$. Let  $u \in L^2(\Omega)$ be a very weak solution of \eqref{eq:1.1} with $\nabla u \in BV(\Omega) \cap L^\infty(\Omega)$. Then for $r := \frac{(n+q)q}{q-n}$ there holds
\begin{equation}
    ||\nabla u ||_{L^\infty(\Omega)} \leq C(n,q,\Omega, [\Gamma]_{1,0}) ||DA||_{L^q(\Omega)} ||\nabla u||_{L^{r}(\Omega)} + ||\mathrm{tr}_{\partial\Omega'}(\nabla u)||_{L^\infty(\Gamma)} + ||\mathrm{tr}_{\partial \Omega''} (\nabla u )||_{L^\infty(\Gamma)}. 
\end{equation}
\end{lemma}
\begin{proof}
Define $w_j := \partial_j u\vert_{\Omega'} \in BV(\Omega') \cap L^\infty(\Omega')$. Naturally, $\mathrm{tr}_{\partial \Omega'}(w_j) = \mathrm{tr}_{\partial \Omega'}(\partial_j u)$.
We intend to obtain an equation for $w_j$ on $\Omega'$ and $\Omega''$ to which the maximum principle can be applied.
%we claim that on $\Omega'$ there holds (weakly) 
%\begin{equation}
%    \begin{cases}
 %    - \mathrm{div}(A(x) \nabla w_j) = \mathrm{div}((\partial_j A(x)) \nabla u )  & \textrm{in} \; \Omega' \\ w_j = \mathrm{tr}_{\partial \Omega'} (\partial_j u) & \textrm{on} \; \partial \Omega'  
 %   \end{cases}
%\end{equation}
To this end we compute for all $\phi \in C_0^\infty(\Omega')$
\begin{align}
  -  \int_{\Omega'}  w_j \mathrm{div}(A(x) \nabla \phi) \; \mathrm{d}x  & = \int_{\Omega'}  \partial_j u  \;  \mathrm{div}(A(x) \nabla \phi) \; \mathrm{d}x \\ &  =  \int_{\Omega'} u \;  \mathrm{div} (\partial_j A(x) \nabla \phi + A(x) \nabla (\partial_j \phi) ) \; \mathrm{d}x
  \\ & = - \int_{\Omega'} (\partial_j A(x) \nabla u , \nabla \phi)  \; \mathrm{d}x + \int_{\Omega'}  (A(x)  \nabla u , \nabla ( \partial_j \phi) )  \; \mathrm{d}x 
  \\ & = - \int_{\Omega'} (\partial_j A(x) \nabla u , \nabla \phi) \; \mathrm{d}x,  \label{eq:deriveq}
\end{align}
where we have used in the last step that $u$ solves \eqref{eq:1.1}. 
Notice that by Hölder's inequality
\begin{equation}\label{eq:hoeldi}
    || (\partial_j A) \nabla u ||_{L^{\frac{n+q}{2}}(\Omega')} \leq ||\partial_j A||_{L^q(\Omega)} ||\nabla u||_{L^r(\Omega)}.
\end{equation}
Hence by \cite[Theorem 1.1]{Auscher} there exists a unique weak solution $\tilde{w}_j \in W_0^{1,\frac{n+q}{2}}(\Omega')$ 
\begin{equation}
    \begin{cases}
     - \mathrm{div}(A(x) \nabla \tilde{w}_j) = \mathrm{div}((\partial_j A(x)) \nabla u )  & \textrm{in} \; \Omega', \\ \tilde{w}_j = 0 & \textrm{on} \; \partial \Omega',  
    \end{cases}
\end{equation}
in the sense that $\tilde{w}_j \in W_0^{1,\frac{n+q}{2}}(\Omega')$  and 
\begin{equation} \label{eq:tilwj}
    \int_{\Omega'} (A(x) \nabla \tilde{w}_j , \nabla \eta ) \; \mathrm{d}x =- \int_{\Omega'} (\partial_j A(x) \nabla u , \nabla \eta) \; \mathrm{d}x  \quad \forall \eta \in C_0^\infty(\Omega'). 
\end{equation}
By \cite[Eq. (3) and Theorem 1.1]{Auscher} and \eqref{eq:hoeldi} we infer also that
\begin{equation}
    ||\tilde{w}_j||_{W_0^{1,\frac{n+q}{2}}(\Omega')} \leq C([\partial \Omega']_{1,0})||\partial_j A \nabla u||_{L^{\frac{n+q}{2}}} \leq C([\Gamma]_{1,0})||DA||_{L^q} \; || \nabla u||_{L^{r}}.
\end{equation}
%As this is vital for our argument we stress another time that this result only uses $C^1$-regularity of $\Omega'$ and the constant in the above estimate really only depends on $[\partial \Omega']_{C^1}$, cf. \cite[Eq. (3)]{Auscher}.
Now defining $\psi := w_j - \tilde{w}_j$ we infer that $\psi \in BV(\Omega')$ and $\mathrm{tr}_{\partial \Omega'}(\psi) = \mathrm{tr}_{\partial \Omega'}(\partial_j u)$. Moreover \eqref{eq:deriveq} and \eqref{eq:tilwj} imply that for all $\eta \in C_0^\infty(\Omega')$ there holds
\begin{equation}
    \int_\Omega \psi \; \mathrm{div}(A(x) \nabla \eta) \; \mathrm{d}x = \int_\Omega (w_j - \tilde{w}_j) \mathrm{div}(A(x) \nabla \eta) \; \mathrm{d}x = 0.
\end{equation}
%and $\psi$ solves 
%\begin{equation}
%    \begin{cases}
%     -\mathrm{div} (A(x) \nabla \psi)  = 0 & \textrm{in} \; \Omega' \\ \psi = \mathrm{tr}_{\partial \Omega'} (\partial_j u) & \textrm{on} \;  \partial \Omega'. 
 %   \end{cases}
%\end{equation}
Hence $\psi$ is (in the sense of Definition \ref{def:BVsol}) a $BV$-solution of
\begin{equation}
    \begin{cases}
      -\mathrm{div}(A \nabla \psi) = 0 & \textrm{in }\Omega', \\ \qquad \quad \quad \; \;   \psi = \mathrm{tr}_{\partial \Omega'}(\partial_j u) & \textrm{on }\partial \Omega'.
    \end{cases}
\end{equation}
Therefore $\psi$ satisfies all prerequisites for the announced elliptic maximum principle (Lemma \ref{lem:maxBV}), from which we conclude that $||\psi||_{L^\infty(\Omega')} \leq  || \mathrm{tr}_{\partial \Omega'} (\partial_j u) ||_{L^\infty(\partial \Omega')}$. 
Using this and the  Sobolev embedding $W^{1, \frac{n+q}{2}} \hookrightarrow L^\infty$ (with embedding constant $D = D([\partial \Omega']_{1,0})$) we find
\begin{align}
    ||w_j||_{L^\infty(\Omega')}  & \leq ||\tilde{w}_j||_{L^\infty(\Omega')} + ||\psi||_{L^\infty(\Omega')}
    \leq D([\partial \Omega']_{1,0}) ||\tilde{w}_j||_{W^{1,\frac{n+q}{2}}(\Omega')}+  ||\psi||_{L^\infty(\Omega')}
    \\ & \leq D([\Gamma]_{1,0})C([\Gamma]_{1,0}) ||DA||_{L^q(\Omega)} ||\nabla u||_{L^{r}(\Omega)} + ||\mathrm{tr}_{\partial \Omega'} (\partial_j u) ||_{L^\infty}. 
\end{align}
For $\Omega''$ one can derive the same bound and thereupon infers the claim.
\end{proof}

\begin{cor}\label{cor:4.4}
Suppose that $\Gamma = \partial \Omega' \in C^\infty$, $Q \in W^{1,q}(\Omega)$ and $A \in W^{1,q}(\Omega;\mathbb{R}^{n\times n})$ for some $q> n$. Let $u \in L^2(\Omega)$ be a very  weak solution of \eqref{eq:1.1} with $\nabla u \in BV(\Omega) \cap L^\infty(\Omega)$. Let $\theta : \Gamma \rightarrow \mathbb{R}$ be as in Lemma \ref{lem:Taylor}. Then 
\begin{equation}\label{eq:gradestfastcomplete}
    ||\nabla u||_{L^\infty(\Omega)} \leq C(n,q, \Omega, [\Gamma]_{1,0})||DA||_{L^q(\Omega)} ||Q||_{L^\infty(\Gamma)} + 2||\theta||_{L^\infty} +  \frac{1}{\lambda  }||Q||_{L^\infty(\Gamma)}.
\end{equation}
\end{cor}
\begin{proof}
By Lemma \ref{lem:traces} and Lemma \ref{lem:apriobound} we have that for $r := \frac{q(n+q)}{q-n}$
\begin{equation}
    ||\nabla u||_{L^\infty(\Omega)} \leq C(n,q,\Omega,[\Gamma]_{1,0}) ||DA||_{L^q} ||u||_{W^{1,r}} + || \theta - \tfrac{1}{2} \tfrac{Q}{(A\nu,\nu)} ||_{L^\infty(\Gamma)}+|| \theta + \tfrac{1}{2} \tfrac{Q}{(A\nu,\nu)} ||_{L^\infty(\Gamma)}.
\end{equation}
Now notice that by Lemma \ref{lem:W1preg} we have
\begin{equation}
    ||u||_{W^{1,r}} \leq C(r,[\Gamma]_{0,1}) ||Q||_{L^\infty(\Gamma)} .
\end{equation}
Using this, the triangle inequality and $\frac{1}{(A\nu, \nu)} \leq \frac{1}{\lambda}$ the claimed estimate follows. 
\end{proof}

The estimate \eqref{eq:gradestfastcomplete} has (except for the second summand) constants depending only on $n, q, \Omega, [\Gamma]_{1,\alpha}, ||A||_{W^{1,q}}, \lambda$ and $||Q||_{C^{0,\alpha}}$. 
We also will achieve such control in the second summand by estimating $||\theta||_{L^\infty(\Gamma)}$ in terms of the same quantities. We will obtain a control of $||\theta||_{L^\infty(\Gamma)}$ using the Taylor expansion in Lemma \ref{lem:Taylor} and a special test function in \eqref{eq:1.1}. 
\begin{lemma}
Suppose that $\Gamma = \partial \Omega' \in C^\infty$, $Q \in C^{0,\alpha}(\Omega)$ and $A \in W^{1,q}(\Omega;\mathbb{R}^{n\times n })$ for some $q> n$. Let $u \in L^2(\Omega)$ be a very weak solution of \eqref{eq:1.1}  such that $\nabla u \in L^\infty(\Omega) \cap BV(\Omega)$ and let $\theta$ be as in Lemma \ref{lem:Taylor}. Then we have
\begin{equation}\label{eq:erg:qunatities}
    ||\theta||_{L^\infty(\Gamma)} \leq C(n , q, \alpha, \lambda ,  ||Q||_{C^{0,\alpha}}, ||A||_{W^{1,q}}, [\Gamma]_{1,\alpha} ,\mathrm{dist}(\Gamma, \partial \Omega) ).
\end{equation}
\end{lemma}
\begin{proof} Without loss of generality we may assume that $ \alpha < 1- \frac{n}{q} $, so that $W^{1,q} \hookrightarrow C^{0,\alpha}$. (If this is not the case then regard $Q$ as an element of $C^{0,\alpha'}$ for some $\alpha' < \alpha$ and work with $\alpha'$ instead). 
Let $x_0 \in \Gamma$ be an arbitrary Lebesgue point of $\theta$ such that the formula in Lemma \ref{lem:Taylor} holds and $|\theta(x_0)| \geq \frac{1}{2}||\theta||_{L^\infty(\Gamma)}.$
For this proof we set $A_0 := A(x_0)$. 
Fix $r \in (0, \tfrac{1}{2}\mathrm{dist}(\Gamma, \partial \Omega))$ and define for $\epsilon> 0$ small $\eta_\epsilon : [0,\infty) \rightarrow \mathbb{R}$ via 
\begin{equation}
    \eta_\epsilon(s) := \begin{cases}
        1 & 0 \leq s \leq r, \\ 1 - \frac{s-r}{\epsilon} & r < s \leq r+ \epsilon, \\ 0 & r \geq r+ \epsilon. 
    \end{cases}
\end{equation}
Let additionally $\eta_\epsilon^\delta := \eta_\epsilon * \psi_\delta$, where $(\psi_\delta)_{\delta > 0}$ is a standard mollifier on $\mathbb{R}$. Define
\begin{equation}
    \phi(x) := (\theta(x_0) ,x-x_0) \eta_\epsilon^\delta( |A_0^{-1/2}(x-x_0)| ), 
\end{equation}
where $A_0^{-1/2}$ is a symmetric positive definite matrix such that $(A_0^{-1/2})^2 = A_0^{-1}$. Using $\nabla |A_0^{-1/2}(x-x_0)| = \frac{A_0^{-1}(x-x_0)}{|A_0^{-1/2}(x-x_0)|}$ we compute 
\begin{equation}
  \nabla \phi(x)  = \theta(x_0) \eta_\epsilon^\delta(|A_0^{-1/2}(x-x_0)|) + (\theta(x_0),x-x_0) ( \eta_\epsilon^\delta)'(|A_0^{-1/2}(x-x_0)|) \frac{A_0^{-1}(x-x_0)}{|A_0^{-1/2}(x-x_0)|}.
\end{equation}
%We compute 
%\begin{align}
  %   & \mathrm{div}(A_0 \nabla \phi) 
    %\\ & = \mathrm{div} \left( A_0 \left[ \theta(x_0) \eta_\epsilon^\delta(|A_0^{-1/2}(x-x_0)|) + (\theta(x_0),x-x_0) ( \eta_\epsilon^\delta)'(|A_0^{-1/2}(x-x_0)|) \frac{A_0^{-1}(x-x_0)}{|A_0^{-1/2}(x-x_0)|} \right] \right) 
    %\\ & =  \mathrm{div}\left( \eta_\epsilon^\delta(|A_0^{-1/2}(x-x_0)|) A_0 \theta(x_0) + (\theta(x_0), x-x_0) (\eta_\epsilon^\delta)'(|A_0^{-1/2}(x-x_0)|) \frac{x-x_0}{|A_0^{-1/2}(x-x_0)|} \right) 
    %= (\eta_\epsilon^\delta)'(|A_0^{-1/2}(x-x_0)|) \left( \frac{A_0^{-1} (x-x_0)}{|A_0^{-1/2}(x-x_0)|} , A_0 \theta(x_0)  \right) \\ & + (\eta_\epsilon^\delta)'(|A_0^{-1/2}(x-x_0)|) \frac{(\theta(x_0), x- x_0) }{|A_0^{-1/2}(x-x_0)|} 
    %\\ & + (\theta(x_0),x-x_0) (\eta_\epsilon^\delta)''(|A_0^{-1/2}(x-x_0)|) \frac{(A_0^{-1}(x-x_0), x-x_0)}{|A_0^{-1/2}(x-x_0)|^2}
    %\\ & + (\theta(x_0),x-x_0) (\eta_\epsilon^\delta)'(|A_0^{-1/2}(x-x_0)|) \mathrm{div} \left( \frac{x-x_0}{|A_0^{-1/2}(x-x_0)|} \right) 
%\end{align}
By symmetry of $A_0$ one has 
\begin{equation}
    \left( \frac{A_0^{-1} (x-x_0)}{|A_0^{-1/2}(x-x_0)|} , A_0 \theta(x_0)  \right) = \frac{(\theta(x_0),x-x_0)}{|A_0^{-1/2}(x-x_0)|} 
\end{equation}
and 
\begin{equation}
    \mathrm{div} \left( \frac{x-x_0}{|A_0^{-1/2}(x-x_0)|} \right) = \frac{n}{|A_0^{-1/2}(x-x_0)|} - \frac{(A_0^{-1}(x-x_0), x-x_0)}{|A_0^{-1/2}(x-x_0)|^3} = \frac{n-1}{|A_0^{-1/2}(x-x_0)|}, 
\end{equation}
whereupon one readily checks 
\begin{equation}
    \mathrm{div}(A_0\nabla \phi) = (\theta(x_0),x-x_0)  \left[(n+1)  \frac{(\eta_\epsilon^\delta)'(|A_0^{-1/2}(x-x_0)|)}{|A_0^{-1/2}(x-x_0)|} +  (\eta_\epsilon^\delta)''(|A_0^{-1/2}(x-x_0)|) \right].  
\end{equation}
Hence there holds
\begin{align}
   &\int_{\Gamma} Q \phi \; \mathrm{d}\mathcal{H}^{n-1} =  \int_\Omega (A \nabla u, \nabla \phi) \; \mathrm{d}x  = \int_\Omega ( (A-A_0) \nabla u , \nabla \phi) \; \mathrm{d}x  - \int_\Omega u  \; \mathrm{div}(A_0 \nabla \phi) \; \mathrm{d}x
    \\ & =- \int_\Omega u  (\theta(x_0),x-x_0) \Big[(n+1)  \frac{(\eta_\epsilon^\delta)'(|A_0^{-1/2}(x-x_0)|)}{|A_0^{-1/2}(x-x_0)|}      % \\ &  \quad \quad \quad \quad \quad \quad \quad \quad \quad \quad \quad \quad \quad \quad  +  
    + (\eta_\epsilon^\delta)''(|A_0^{-1/2}(x-x_0)|) \Big] \; \mathrm{d}x  \\ &  \quad + S_\epsilon^\delta(r),
\end{align}
where 
\begin{align}\label{eq:Sphi}
    S_\epsilon^\delta(r)  &  := \int_\Omega ((A-A_0) \nabla u , \nabla \phi) \; \mathrm{d}x
    \\ &   = \int_\Omega ((A-A_0) \nabla u , \eta_\epsilon^\delta(|A_0^{-1/2}(x-x_0)|) \theta(x_0)) \; \mathrm{d}x \\ &  + \left((A-A_0) \nabla u ,  (\theta(x_0) , x- x_0) (\eta_\epsilon^\delta)'(|A_0^{-1/2}(x-x_0)|) \frac{A_0^{-1}(x-x_0)}{|A_0^{-1/2}(x-x_0)|} \right) \; \mathrm{d}x.
\end{align}
We remark that $S_\epsilon^\delta(r)$ depends on $r$ via the dependence of $\eta_\epsilon^\delta$ of $r$.
Notice that the supports of $(\eta_\epsilon^\delta)'(|A_0^{-1/2}(\cdot - x_0)|)$ and $(\eta_\epsilon^\delta)''(|A_0^{-1/2}(\cdot - x_0)|)$ lie in the closure of $x_0 + A_0^{1/2}(B_{r+\epsilon}(0) \setminus B_r(0)) \subset  \subset \Omega$ (if $\epsilon$ is appropriately small). Hence we can substitute $y = A_0^{-1/2}(x - x_0)$ and obtain
\begin{align}\label{eq:monoformu}
     & \int_\Gamma Q \phi \; \mathrm{d}\mathcal{H}^{n-1} - S_\epsilon^\delta(r)  \\ &= - \mathrm{det} (A_0^{1/2})\int_{\mathbb{R}^n}  u(x_0 + A_0^{1/2}y)(A_0^{1/2} \theta(x_0), y) \Big[(n+1) \frac{(\eta_\epsilon^\delta)'(|y|)}{|y|}  +  (\eta_\epsilon^\delta)''(|y|)) \Big] \; \mathrm{d}y. 
\end{align}
Looking at the second summand of integral we compute 
\begin{align}
    & \int_{\mathbb{R}^n} u(x_0+ A_0^{1/2}y) \big( (A_0^{1/2} \theta(x_0), y) (\eta_\epsilon^\delta)''(|y|) \big) \; \mathrm{d}y 
    \\ & = \int_0^\infty  (\eta_\epsilon^\delta)''(s) \int_{\partial B_s(0)} u(x_0+ A_0^{1/2}y)  (A_0^{1/2} \theta(x_0), y) \; \mathrm{d}\mathcal{H}^{n-1}(y) \; \mathrm{d}s 
    \\ & = - \int_0^\infty  (\eta_\epsilon^\delta)'(s) \frac{d}{ds} \left( \int_{\partial B_s(0)} u(x_0+ A_0^{1/2}y)  (A_0^{1/2} \theta(x_0), y) \; \mathrm{d}\mathcal{H}^{n-1}(y) \right) \; \mathrm{d}s, 
\end{align}
where (a.e.-)existence (and boundedness) of the derivative of the expression before is ensured by Lemma \ref{lem:Lipschitzmeanvalue}.
We define now  
\begin{equation}\label{eq:defha}
    h(s) :=  \int_{\partial B_s(0)} u(x_0+ A_0^{1/2}y)  (A_0^{1/2} \theta(x_0), y) \; \mathrm{d}\mathcal{H}^{n-1}(y)
\end{equation}
and observe that by \eqref{eq:monoformu}
\begin{equation}
   \frac{1}{\mathrm{det}(A_0^{1/2})} \left( \int_\Gamma Q \phi \; \mathrm{d}\mathcal{H}^{n-1}- S_\epsilon^\delta(r) \right) = \int_0^\infty (\eta_\epsilon^\delta)'(s) h'(s) \; \mathrm{d}s - \int_0^\infty (\eta_\epsilon^\delta)'(s) \frac{n+1}{s} h(s) \; \mathrm{d}s. 
\end{equation}
Our goal is now to let $\delta \rightarrow 0$. We remark that $(\eta_\epsilon^\delta)'$ converges in $L^1((\frac{r}{2},2r))$ to $\frac{1}{\epsilon}\chi_{[r,r+\epsilon]}$. 
Since $h'$ and $s \mapsto \frac{n+1}{s}h(s)$ are uniformly bounded in $(\frac{r}{2},2r)$ and for $\epsilon, \delta$ suitably small $(\eta_\epsilon^\delta)'$ is only supported in $(\tfrac{r}{2},2r)$ one can pass to the limit and obtains
\begin{align}
    &\frac{1}{\mathrm{det}(A_0^{1/2}) } \left(  \int_\Gamma Q(x)  (\theta(x_0) ,x-x_0) \eta_\epsilon(|A_0^{-1/2}(x-x_0)|) \; \mathrm{d}\mathcal{H}^{n-1} - S_\epsilon(r) \right)  \\ &  = \frac{1}{\epsilon} \int_{r}^{r+\epsilon} \left( h'(s)  - \frac{n+1}{s} h(s) \right) \; \mathrm{d}s,\label{eq:LimesSepsilongegen0}
\end{align}
where 
\begin{equation}
    S_\epsilon(r) := \lim_{\delta \rightarrow 0} S_\epsilon^\delta(r).
\end{equation}
The existence of this limit we will show later.
Letting also $\epsilon \rightarrow 0$ we obtain for a.e. $r \in (0,\frac{1}{2}\mathrm{dist}(\Gamma, \partial \Omega))$
\begin{equation}
    h'(r) - \frac{n+1}{r} h(r) = \frac{1}{\mathrm{det}(A_0^{1/2}) } \left( \int_{\Gamma\cap \mathcal{E}_r(x_0)} Q(x) (\theta(x_0) ,x-x_0) \; \mathrm{d}\mathcal{H}^{n-1} + S(r) \right),
\end{equation}
where $S(r) := \lim_{\epsilon \rightarrow 0 } S_\epsilon(r)$ (which due to \eqref{eq:LimesSepsilongegen0} must exist whenever  $\lim_{\epsilon \rightarrow 0} \int_{r}^{r+\epsilon} \left( h'(s)  - \frac{n+1}{s} h(s) \right) \; \mathrm{d}s$  exists, i.e. at each Lebesgue point of $h'$) and 
\begin{equation}
    \mathcal{E}_r(x_0) := \{ x \in \mathbb{R}^n : |A_0^{-1/2}(x-x_0)| < r \} = A_0^{1/2} B_r(A_0^{-1/2}x_0) . 
\end{equation}
We infer 
\begin{equation}\label{eq:der}
    \frac{d}{dr} \left( \frac{h(r)}{r^{n+1}}  \right) = \frac{1}{\mathrm{det}(A_0^{1/2})} \frac{1}{r^{n+1}} \int_{\Gamma \cap \mathcal{E}_r(x_0)} Q(x)  (\theta(x_0), x-x_0) \; \mathrm{d}\mathcal{H}^{n-1}(x) + \frac{1}{\mathrm{det}(A_0^{1/2})} \frac{1}{r^{n+1}} S(r).
\end{equation}
We claim next that (with $\omega_n := \mathcal{H}^{n-1}(\partial B_1(0))$)
\begin{equation}\label{eq:limithfunc}
    \lim_{r\rightarrow 0}  \left( \frac{h(r)}{r^{n+1}}  \right)= \frac{\omega_n}{n} |A_0^{1/2}\theta(x_0)|^2.
\end{equation}
To this end notice that by the Taylor expansion in Lemma  \ref{lem:Taylor} one has
\begin{align}
    \lim_{r \rightarrow 0} & \left( \frac{h(r)}{r^{n+1}} \right)   = \lim_{r \rightarrow 0 } \frac{1}{r^{n+1}} \int_{\partial B_r(0)}  u(x_0 + A_0^{1/2} y ) (A_0^{1/2} \theta(x_0) ,y) \; \mathrm{d}\mathcal{H}^{n-1}(y) 
    \\ & = \lim_{r \rightarrow 0 } \frac{1}{r^{n+1}} \int_{\partial B_r(0)}  \left[ u(x_0) - \tfrac{1}{2} \tfrac{Q(x_0)}{(A_0 \nu(x_0), \nu(x_0)} |(A_0^{1/2}y, \nu(x_0))| + (\theta(x_0), A_0^{1/2} y) + o(r) \right] \\  & \quad \quad \quad \quad \quad \quad \quad \quad \quad \quad \quad \quad \quad \quad \quad \quad \quad \quad \quad \quad \quad \quad \quad \quad \quad \cdot (A_0^{1/2} \theta(x_0) ,y) \; \mathrm{d}\mathcal{H}^{n-1}(y).
\end{align}
Since $y \mapsto \left[ u(x_0) - \tfrac{1}{2} \tfrac{Q(x_0)}{(A_0 \nu(x_0), \nu(x_0)} |(A_0^{1/2}y, \nu(x_0))|  \right] $ is even with respect to the transformation $y \mapsto -y$ and $y \mapsto (A^{-1/2} \theta(x_0) ,y)$ is odd w.r.t this transformation we infer
%Now noticing that $|Q(x_0 + A^{1/2}y) - Q(x_0)| \leq C r^\gamma$ for all 
\begin{equation}
    \frac{1}{r^{n+1}} \int_{\partial B_r(0)} \left[ u(x_0) - \tfrac{1}{2} \tfrac{Q(x_0)}{(A_0 \nu(x_0), \nu(x_0)} |(A_0^{1/2}y, \nu(x_0))|  \right]  (A^{-1/2} \theta(x_0) ,y) \; \mathrm{d}\mathcal{H}^{n-1}(y) =  0. 
\end{equation}
Moreover, 
\begin{equation}
    \left\vert \frac{o(r)}{r^{n+1}} \int_{\partial B_r(0)} (A_0^{1/2}\theta(x_0), y) \; \mathrm{d}\mathcal{H}^{n-1}(y) \right\vert \leq \frac{o(1)}{r^n} \omega_n r^{n-1} |A_0^{1/2} \theta(x_0)| r  = o(1),  
\end{equation}
as $r \rightarrow 0$. Hence 
\begin{align}
    \lim_{r \rightarrow 0} \left(\frac{h(r)}{r^{n+1}} \right) & = \lim_{r \rightarrow 0 } \frac{1}{r^{n+1}} \int_{\partial B_r(0)}  (A_0^{1/2} \theta(x_0) , y)^2 \; \mathrm{d}\mathcal{H}^{n-1}(y) \\ & = |A_0^{1/2} \theta(x_0)|^2  \lim_{r \rightarrow 0 } \frac{1}{r^{n+1}} \int_{\partial B_r(0)}  (\tfrac{A_0^{1/2} \theta(x_0)}{|A_0^{1/2} \theta(x_0)|} , y)^2 \; \mathrm{d}\mathcal{H}^{n-1}(y).   
\end{align}
We may now apply an orthogonal transformation that maps  $\tfrac{A_0^{1/2} \theta(x_0)}{|A_0^{1/2} \theta(x_0)|}$ to $e_1$ and use the symmetry of $\partial B_r(0)$ to find 
\begin{align}
    \lim_{r \rightarrow 0} \left( \frac{h(r)}{r^{n+1}} \right) & = |A_0^{1/2} \theta(x_0)|^2 \lim_{r \rightarrow 0} \frac{1}{r^{n+1}}\int_{\partial B_r(0)} z_1^2 \; \mathrm{d}\mathcal{H}^{n-1}(z) 
    \\ & =  
    \frac{1}{n}|A_0^{1/2} \theta(x_0)|^2 \lim_{r \rightarrow 0} \frac{1}{r^{n+1}}\int_{\partial B_r(0)} (z_1^2+ ...+ z_n^2) \; \mathrm{d}\mathcal{H}^{n-1}(z)
    \\ & = \frac{1}{n}|A_0^{1/2} \theta(x_0)|^2 \lim_{r \rightarrow 0} \frac{1}{r^{n+1}} \int_{\partial B_r(0)}r^2 \; \mathrm{d}\mathcal{H}^{n-1}(z)
     \\ & = \frac{1}{n}|A_0^{1/2} \theta(x_0)|^2 \lim_{r \rightarrow 0} \frac{1}{r^{n+1}} \omega_n r^{n+1} = \frac{\omega_n}{n}|A_0^{1/2} \theta(x_0)|^2. 
\end{align}
This shows \eqref{eq:limithfunc}. Using this and \eqref{eq:der} we obtain for any $r_0 \in ( 0, \tfrac{1}{2} \mathrm{dist}(\Gamma,\partial\Omega))$ 
\begin{align}
    \frac{\omega_n}{n} |A_0^{1/2}\theta(x_0)|^2  & = \lim_{r \rightarrow 0} \frac{h(r)}{r^{n+1}}  \leq \frac{h(r_0)}{r_0^{n+1}} + \int_0^{r_0} \left\vert \frac{d}{dr} \frac{h(r)}{r^{n+1}} \right\vert \; \mathrm{d}r
    \\ & = \frac{h(r_0)}{r_0^{n+1}} + \int_0^{r_0} \left\vert \frac{1}{\mathrm{det}(A_0^{1/2})} \frac{1}{r^{n+1}} \int_{\Gamma \cap \mathcal{E}_r(x_0)} Q(x) (\theta(x_0), x-x_0) \; \mathrm{d}\mathcal{H}^{n-1}(x)  \right\vert \; \mathrm{d}r \\ &  + \int_0^{r_0} \frac{1}{\mathrm{det}(A_0^{1/2})} \frac{1}{r^{n+1}} |S(r)| = \textrm{(I) + (II) + (III)}. \label{eq:Term1+2+3}
\end{align}
We will estimate all the three summands appearing in this sum.\\
\textbf{Term (I).}
Estimating the integrand in \eqref{eq:defha} by its $L^\infty$-norm and using Lemma \ref{lem:W1preg} (with $W^{1,p} \hookrightarrow L^\infty$ for $p > n$)  we find
\begin{equation}
    \left\vert  \frac{h(r_0)}{r_0^{n+1}}  \right\vert \leq \frac{1}{r_0^{n+1}} ||u||_{L^\infty(\Omega)} |A^{1/2} \theta(x_0)|  \omega_n r_0^n \leq \frac{1}{r_0}C(\Omega,[\Gamma]_{0,1})  |A_0^{1/2} \theta(x_0)| \; ||Q||_{L^\infty}.
\end{equation}
\textbf{Term (II).} Notice that
\begin{align}\label{eq:absorbQ}
    &\frac{1}{r^{n+1}}\int_{\Gamma \cap \mathcal{E}_r(x_0)} Q(x)  (\theta(x_0), x-x_0) \; \mathrm{d}\mathcal{H}^{n-1}(x)   =  \frac{Q(x_0)}{r^{n+1}} \int_{\Gamma \cap \mathcal{E}_r(x_0)}(\theta(x_0), x-x_0) \; \mathrm{d}\mathcal{H}^{n-1}(x)  \\ & \quad + \frac{1}{r^{n+1}} \int_{\Gamma \cap \mathcal{E}_r(x_0)} (Q(x)- Q(x_0)) (\theta(x_0), x-x_0) \; \mathrm{d}\mathcal{H}^{n-1}(x).
\end{align}
The last summand is bounded by 
\begin{align}\label{eq:boundabsorbedQ}
    \frac{1}{r^{n+1}}||Q||_{C^{0,\alpha}} & r^\alpha |\theta(x_0)| r \mathcal{H}^{n-1}(\Gamma \cap \mathcal{E}_r(x_0))  \\ &\leq \frac{||Q||_{C^{0,\alpha}}|\theta(x_0)|}{r^{n-\alpha}}  \mathcal{H}^{n-1}(\Gamma \cap B_{\frac{r}{\sqrt{\lambda}}}(x_0)) \leq C(n, \lambda,||Q||_{C^{0,\alpha}}, [\Gamma]_{0,1}) \frac{1}{r^{1-\alpha}}|\theta(x_0)|,
\end{align}
which is integrable on $(0,r_0)$. 
For the first summand we obtain
\begin{align}
    \frac{Q(x_0)}{r^{n+1}}\int_{\Gamma \cap \mathcal{E}_r(x_0)} & (\theta(x_0), x-x_0) \; \mathrm{d}\mathcal{H}^{n-1}(x)   =  \frac{Q(x_0)}{r^{n+1}} \int_{\Gamma \cap A_0^{1/2} B_r (A_0^{-1/2}x_0)}  (\theta(x_0), x-x_0) \; \mathrm{d}\mathcal{H}^{n-1}(x)
    \\ & =  \frac{Q(x_0)}{r^{n+1}} \int_{A_0^{-1/2}\Gamma \cap  B_r (A_0^{-1/2}x_0)}  (\theta(x_0), A_0^{1/2}x-x_0) \; \mathrm{J}_{A_0^{1/2}}(x)\; \mathrm{d}\mathcal{H}^{n-1}(x),
\end{align}
where $\mathrm{J}_{A_0^{1/2}}(x) := \mathrm{det}(A_0^{1/2}\tau_i(x), A_0^{1/2}\tau_j(x))_{i,j=1,...,n-1}$ and $\{ \tau_1(x),...,\tau_{n-1}(x) \}$ denotes an orthonormal basis of $T_x \Gamma$.  With Lemma \ref{lem:ballgraphsystem} it is easily checked that $\mathrm{J}_{A_0^{1/2}}$ lies in $C^{0,\alpha}(\Gamma)$ with $C^{0,\alpha}$-norm bounded by a constant $D(||A||_{L^\infty(\Omega)},[\Gamma]_{1,\alpha})$. We can therefore proceed as above \eqref{eq:absorbQ} and \eqref{eq:boundabsorbedQ} and write 
\begin{align}
    \frac{1}{r^{n+1}} & \int_{\Gamma \cap \mathcal{E}_r(x_0)} Q(x) (\theta(x_0), x-x_0) \; \mathrm{d}\mathcal{H}^{n-1}(x)  \\ & = \frac{Q(x_0)\mathrm{J}_{A_0^{1/2}}(x_0)}{r^{n+1}} \int_{A_0^{-1/2}\Gamma \cap  B_r (A_0^{-1/2}x_0)}  (\theta(x_0), A_0^{1/2}x-x_0)\; \mathrm{d}\mathcal{H}^{n-1}(x) + R(r)
\end{align}
where $R(r)$ is integrable and $\int_0^{r_0} R(r) \; \mathrm{d}r \leq C(\alpha, ||Q||_{C^{0,\alpha}}, ||A||_{L^\infty}, [\Gamma]_{1,\alpha}) r_0^\alpha$, which can be bounded by the desired quantities as $r_0 \leq \mathrm{dist}(\Gamma, \partial \Omega)$.  Now using Lemma \ref{lem:Lipschitzmeanvalue} with $\tilde{\Gamma} = A_0^{-1/2} \Gamma$ we can estimate
\begin{align}
     \frac{1}{r^{n+1}} & \int_{\Gamma \cap \mathcal{E}_r(x_0)} Q(x) (\theta(x_0), x-x_0) \; \mathrm{d}\mathcal{H}^{n-1}(x) \\ & = \frac{Q(x_0)\mathrm{J}_{A_0^{1/2}}(x_0)}{r^{n+1}} \left( A_0^{1/2}\theta(x_0) , \int_{A_0^{-1/2}\Gamma \cap  B_r (A_0^{-1/2}x_0)} (x- A_0^{-1/2} x_0) \; \mathrm{d}\mathcal{H}^{n-1}(x) \right) +R(r)
     \\ & \leq \frac{1}{r^{n+1}} ||Q||_{L^\infty} |A_0^{1/2}| \;  |A_0^{1/2} \theta(x_0) | C( [A_0^{-1/2} \Gamma]_{1,\alpha}) r^{n+\alpha} + R(r) \\ & =  ||Q||_{L^\infty} |A_0^{1/2}| \;  |A_0^{1/2} \theta(x_0) | \;  C( [A_0^{-1/2} \Gamma]_{1,\alpha}) \frac{1}{r^{1-\alpha}} + R(r).
\end{align}
 Noticing that $\frac{1}{r^{1-\alpha}}$ is integrable, $[A_0^{-1/2}\Gamma]_{1,\alpha} \leq C(\lambda, [\Gamma]_{1,\alpha})$, $r_0 \leq \mathrm{dist}(\Gamma, \partial \Omega)$ and $|A_0^{1/2}| \leq ||A||_{L^\infty}^{1/2}$ we infer that 
\begin{align}
    \int_0^{r_0} & \frac{1}{r^{n+1}} \left\vert \int_{\Gamma \cap \mathcal{E}_r(x_0)} Q(x) (\theta(x_0) ,x-x_0) \; \mathrm{d}\mathcal{H}^{n-1} \right\vert \\ & \leq C(\alpha, \lambda, ||A||_{C^{0,\alpha}}, ||Q||_{C^{0,\alpha}}, [\Gamma]_{1,\alpha} ) \mathrm{dist}(\Gamma, \partial \Omega)^\alpha |A_0^{1/2} \theta(x_0)| .
\end{align}
\textbf{Term (III).}
We next need to estimate 
\begin{equation}
    \int_0^{r_0} \frac{1}{r^{n+1}}|S(r)| \; \mathrm{d}r.
\end{equation}
To that end we notice that (by the dominated convergence theorem and the fact that $\nabla u \in L^\infty(\Omega)$)
\begin{align}
     & S(r)  = \lim_{\epsilon \rightarrow 0 } \lim_{\delta \rightarrow 0} (S_\epsilon^\delta)(r) \\ & = \lim_{\epsilon \rightarrow 0}  \left( \int \big((A-A_0) \nabla u , \eta_\epsilon ( |A_0^{-1/2}(x-x_0)|)  \theta(x_0) \big)  \right. \\ & + \left. \int \Big( (A- A_0) \nabla u , (\theta(x_0) , x-x_0) \frac{1}{\epsilon} \chi_{[r,r+\epsilon]}(|A_0^{-1/2}(x-x_0)| ) \tfrac{A_0^{-1}(x-x_0)}{|A_0^{-1/2}(x-x_0)|} \Big)  \right). \label{eq:secsumS}
\end{align}
For the first summand we obtain 
\begin{align}
    \lim_{\epsilon \rightarrow 0}  & \left\vert  \int_\Omega \big((A-A_0) \nabla u , \eta_\epsilon ( |A_0^{-1/2}(x-x_0)|) \theta(x_0) \big) \; \mathrm{d}x \right\vert   =  \left\vert \int_{\Omega \cap  \mathcal{E}_r(x_0)} \big((A- A_0) \nabla u, \theta(x_0) \big) \; \mathrm{d}x   \right\vert  \\ & \leq ||A-A_0||_{L^\infty(\mathcal{E}_r(x_0))} ||\nabla u||_{L^1(\Omega)} |\theta(x_0)| \;  |\mathcal{E}_r(x_0)|  .
\end{align}
Since $\mathcal{E}_r(x_0) \subset B_{\frac{r}{\sqrt{\lambda}}}(x_0)$ and (by Lemma \ref{lem:W1preg}) $||\nabla u||_{L^1} \leq C([\Gamma]_{0,1})||Q||_{L^\infty(\Gamma)}$ we obtain 
\begin{align}
    \lim_{\epsilon \rightarrow 0} & \left\vert  \int_\Omega \big((A-A_0) \nabla u , \eta_\epsilon ( |A_0^{-1/2}(x-x_0)|) \theta(x_0) \big) \; \mathrm{d}x \right\vert \leq ||A||_{C^{0,\alpha}} \frac{r^\alpha}{\sqrt{\lambda}^\alpha} C([\Gamma]_{0,1}) |\theta(x_0)| | \mathcal{E}_r(x_0)|   \\ & \leq C(\lambda, [\Gamma]_{0,1} ) ||A||_{C^{0,\alpha}} |\theta(x_0)| r^{n+\alpha}. \label{eq:est1}
\end{align}
For the second summand in \eqref{eq:secsumS} we find after substitution $y = A_0^{-1/2}(x-x_0)$
\begin{align}
    & \lim_{\epsilon \rightarrow 0 } \left\vert \int \Big( (A- A_0) \nabla u , (\theta(x_0) , x-x_0) \frac{1}{\epsilon} \chi_{[r,r+\epsilon]}(|A_0^{-1/2}(x-x_0)| ) \tfrac{A_0^{-1}(x-x_0)}{|A_0^{-1/2}(x-x_0)|} \Big)  \right\vert 
    \\ &  = \lim_{\epsilon \rightarrow 0 } \left\vert \mathrm{det}(A_0^{1/2}) \int \Big( (A(x_0+ A_0^{1/2}y)- A_0) \nabla u(x_0+ A_0^{1/2} y) , (A_0^{1/2}\theta(x_0) , y) \frac{1}{\epsilon} \chi_{[r,r+\epsilon]}(|y| ) \tfrac{A_0^{-1/2}y}{|y|}  \Big) \right\vert
     \\   & = \lim_{\epsilon \rightarrow 0 } \left\vert \mathrm{det}(A_0^{1/2})  \frac{1}{\epsilon}\int_{r}^{r+\epsilon} \int_{\partial B_s(0)} \Big( (A(x_0+ A_0^{1/2}y)- A_0) \nabla u(x_0+ A_0^{1/2} y) , (A_0^{1/2}\theta(x_0) , y)  \tfrac{A_0^{-1/2}y}{|y|}  \Big) \right\vert.
\end{align}
In the last step we used the radial integration formula for an integrand that lies in $L^\infty(\mathbb{R}^n) \cap BV(\mathbb{R}^n)$. Radial integration in this case is somewhat delicate since the integrand is technically not defined on sets of Lebesgue measure zero. Nevertheless it holds true if one  takes the \emph{precise representative} $(\nabla u)^*$ (cf. Lemma \ref{lem:layercake}). With this precise representative it is also allowed to pass to the limit as $\epsilon \rightarrow 0$ (cf. Lemma  \ref{lem:contimeanval}) and obtain 
\begin{align}
    = & \left\vert \mathrm{det}(A_0^{1/2})   \int_{ \partial B_r(0)} ( (A(x_0+ A_0^{1/2}y)- A_0) (\nabla u)^*(x_0+ A_0^{1/2} y) , (A_0^{1/2}\theta(x_0) , y)  \frac{A_0^{-1/2}y}{|y|} \; \mathrm{d}\mathcal{H}^{n-1}(y) \right\vert
    \\ & \leq  \mathrm{det}(A_0^{1/2})  \sup_{y \in \partial B_r(0)} |A(x_0+ A_0^{1/2}y)- A_0| \;  ||\nabla u||_{L^\infty(\Omega)}  |A_0^{1/2} \theta(x_0)| \;  r  \;  |A_0^{-1/2}|   \omega_n r^{n-1}.
\end{align}
Using that $\sup_{y \in \partial B_r(0)} |A(x_0+ A_0^{1/2}y)- A_0|  \leq ||A||_{C^{0,\alpha}} \sup_{y \in \partial B_r(0)}   |A_0^{1/2} y|^\alpha \leq ||A||_{C^{0,\alpha}} |A_0^{1/2}| r^\alpha$ we find 
\begin{align}
    & \lim_{\epsilon \rightarrow 0 } \left\vert \int ( (A- A_0) \nabla u , (\theta(x_0) , x-x_0) \frac{1}{\epsilon} \chi_{[r,r+\epsilon]}(|A_0^{-1/2}(x-x_0)| ) \frac{A_0^{-1}(x-x_0)}{|A_0^{-1/2}(x-x_0)|} \right\vert  \\ &  \leq \omega_n \mathrm{det}(A_0^{1/2}) |A_0^{1/2} \theta(x_0)| \;   ||\nabla u||_{L^\infty} r^{n+\alpha}.
    \end{align}
Together with \eqref{eq:est1} we obtain 
\begin{equation}
    |S(r)| \leq C(||A||_{C^{0,\alpha}}, \lambda,[\Gamma]_{0,1} )  (1+ ||\nabla u||_{L^\infty}) \; |A_0^{1/2} \theta(x_0)|  \;  r^{n+\alpha}.
\end{equation}
Hence 
\begin{equation}
    \int_0^{r_0} \frac{1}{r^{n+1}} |S(r)| \; \mathrm{d}r \leq C(||A||_{C^{0,\alpha}}, \lambda, [\Gamma]_{0,1} ) ( 1+ ||\nabla u||_{L^\infty}) \; |A_0^{\frac{1}{2}} \theta(x_0)| \;  \frac{1}{\alpha} r_0^\alpha.
 \end{equation}
Term (I), (II) and (III) are now estimated. Going back to \eqref{eq:Term1+2+3} we find 
\begin{align}
    & \frac{\omega_n}{n} |A_0^{1/2} \theta(x_0)|^2 \\    & \leq \frac{1}{r_0} C_1 |A_0^{1/2}\theta(x_0)|  +C_2 \frac{1}{\alpha} \mathrm{dist}(\Gamma, \partial \Omega)^\alpha |\theta(x_0)|    + C_3 ( 1+ ||\nabla u||_{L^\infty(\Omega)}) |A_0^{1/2} \theta(x_0)| \frac{1}{\alpha} r_0^\alpha,
\end{align}
with constants $C_1,C_2,C_3$ that depend only on the desired quantities on the right hand side of \eqref{eq:erg:qunatities}. 
%From now on all the constants with a number in the index depend only on desired quantities.
Define $C_4 := \max\{C_1,C_2,C_3\}$. 
Using that $ |A_0^{1/2} \theta(x_0)|^2 \geq \sqrt{\lambda }|\theta(x_0)|^2 \geq \frac{1}{4}\sqrt{\lambda}||\theta||^2_{L^\infty}$ we obtain
\begin{align}
    & \frac{\omega_n}{n} \frac{\sqrt{\lambda}}{4} ||\theta||_{L^\infty}^2 \\ &  \leq C_4 \left( \frac{1}{r_0}||A||_{L^\infty}^{1/2}||\theta||_{L^\infty} + \frac{1}{\alpha} \mathrm{dist}(\Gamma, \partial \Omega)^\alpha  ||\theta||_{L^\infty} + \frac{1}{\alpha}  ||A||_{L^\infty}^{1/2} (1+ ||\nabla u||_{L^\infty}) ||\theta||_{L^\infty} r_0^\alpha \right)  .  
\end{align}
Now we use that by Corollary \ref{cor:4.4} one has
\begin{align}
    ||\nabla u||_{L^\infty(\Omega)}  & \leq C([\Gamma]_{1,\alpha}) ||DA||_{L^q} ||Q||_{L^\infty} + \frac{1}{\lambda } ||Q||_{L^\infty} + 2 ||\theta||_{L^\infty} \\ &  = 2 ||\theta||_{L^\infty} + C([\Gamma]_{1,\alpha}, ||A||_{W^{1,q}} , ||Q||_{L^\infty}) .
\end{align}
Therefore we have
\begin{align}
    & \frac{\omega_n}{n} \frac{\sqrt{\lambda}}{4} ||\theta||_{L^\infty}^2   \leq \tilde{C}_4  \left( \frac{1}{r_0}||\theta||_{L^\infty} + ||\theta||_{L^\infty} + C_5 r_0^\alpha ||\theta||_{L^\infty}^2  + ||\theta||_{L^\infty} C_6  r_0^\alpha \right) .   
\end{align}
Now choosing $r_0$ such that $ \tilde{C}_4 C_5 r_0^\alpha = \min \{ \frac{1}{4}\tilde{C}_4 C_5 \mathrm{dist}(\Gamma,\partial \Omega)^\alpha, \frac{\omega_n}{n} \frac{\sqrt{\lambda}}{8} \} 
$ (which automatically implies $r_0 < \mathrm{dist}(\Gamma, \partial \Omega)$)
we infer that 
\begin{align}
  &  \frac{\omega_n}{n} \frac{\sqrt{\lambda}}{4} ||\theta||_{L^\infty}^2  \leq \frac{\omega_n}{n} \frac{\sqrt{\lambda}}{8} ||\theta||_{L^\infty}^2   + \tilde{C}_4 \left( \frac{1}{r_0} + 1 +  \tilde{C}_6 r_0^\alpha \right) ||\theta||_{\infty}. 
\end{align}
Using that the chosen $r_0$ now also only depends on desired quantities we obtain the asserted bound for $||\theta||_{L^\infty}$. 
\end{proof}

 \begin{cor}\label{cor:prep}
 Let $\Gamma = \partial \Omega' \in C^\infty$, $Q \in W^{2,s}(\Omega)$ and $A \in W^{2,s}(\Omega;\mathbb{R}^{n\times n})$ for some $s >n$. Then for each $q> n$ and $\alpha > 0$ the solution $u \in L^2(\Omega)$ of \eqref{eq:1.1} satisfies $u \in W^{1,\infty}(\Omega)$ and 
 \begin{equation}
     || u||_{W^{1,\infty}} \leq C(n ,q , \alpha, \lambda,  ||Q||_{C^{0,\alpha}}, ||A||_{W^{1,q}}, [\Gamma]_{1,\alpha} ,\mathrm{dist}(\Gamma, \partial \Omega) ).
 \end{equation}
 \end{cor}
 \begin{proof}
 That $u \in W^{1,\infty}(\Omega)$ follows from Lemma \ref{lem:31}. Corollary \ref{cor:4.4} now yields that 
 \begin{equation}
     ||\nabla u||_{L^\infty} \leq C(n,q, \Omega , [\Gamma]_{0,1})||DA||_{L^q(\Omega)} ||Q||_{L^\infty} + 2 ||\theta||_{L^\infty} + \frac{1}{\lambda} ||Q||_{L^\infty}.
 \end{equation}
 Together with the fact that by the previous lemma 
 \begin{equation}
     ||\theta||_\infty \leq C(n , q, \alpha, \lambda,  ||Q||_{C^{0,\alpha}}, ||A||_{W^{1,q}}, [\Gamma]_{1,\alpha} ,\mathrm{dist}(\Gamma, \partial \Omega) )
 \end{equation}
 we infer that 
 \begin{equation}
     ||\nabla u||_{L^\infty} \leq C(n ,q , \alpha, \lambda,  ||Q||_{C^{0,\alpha}}, ||A||_{W^{1,q}}, [\Gamma]_{1,\alpha} ,\mathrm{dist}(\Gamma, \partial \Omega) ).
 \end{equation}
 Since $u \vert_{\partial \Omega} \equiv 0$ we may estimate 
 \begin{equation}
     ||u||_{W^{1,\infty}(\Omega)} \leq (1+ \mathrm{diam}(\Omega))||\nabla u||_{L^\infty(\Omega)}.
 \end{equation}
 The claim follows.
 \end{proof}

\subsection{Completion of the Proof}
\begin{proof}[Proof of Theorem \ref{thm:main}]
Let $Q \in C^{0,\alpha}(\Gamma)$, $A \in W^{1,q}(\Omega;\mathbb{R}^{n\times n})$ and $\Gamma = \partial \Omega' \in C^{1,\alpha}$ be as in the statement.  Choose an approximating sequence $(\Gamma_j = \partial \Omega_j')_{j \in \mathbb{N}}$ of smooth boundaries as in Lemma \ref{lem:approx}, in particular for $j$ large enough, point (1) in Lemma \ref{lem:approx} yields that $\Omega_j' \subset \subset \Omega$ for all $j \in \mathbb{N}$ and point (3) in Lemma \ref{lem:approx} yield that 
$([\Gamma_j]_{1,\alpha})_{j \in \mathbb{N}}$ is uniformly bounded in $j$. Moreover choose $(A_j)_{j \in \mathbb{N}} \subset  C^\infty(\overline{\Omega}; \mathbb{R}^{n\times n})$ such that $A_j \rightarrow A$ in $W^{1,q}(\Omega; \mathbb{R}^{n\times n}).$ By Lemma  \ref{lem:Hoelextension} we may assume that $Q \in C^{0,\alpha}(\overline{\Omega})$ without changing $||Q||_{C^{0,\alpha}}$. By Lemma \ref{lem:Hoelapprox} it is also possible to approximate $Q$ uniformly on $K := \overline{B_\epsilon(\Gamma)}$,$ \epsilon > 0$ suitably chosen,  with some sequence $(Q_j)_{j \in \mathbb{N}} \subset C^\infty(\Omega)$ such that $||Q||_{C^{0,\alpha}(K)} = \lim_{j \rightarrow \infty} ||Q_j||_{C^{0,\alpha}(K)}$.
For $j \in \mathbb{N}$ let $u_j \in L^2(\Omega)$ be the (unique) very weak solution to 
\begin{equation}\label{eq:auxprob}
    \begin{cases}
      -\mathrm{div}(A_j(x) \nabla u_j ) = Q_j \; \mathcal{H}^{n-1} \mres \Gamma_j & \textrm{in }\Omega \\ u = 0 & \textrm{on }\partial \Omega.
    \end{cases}
\end{equation}
By Corollary \ref{cor:prep} we infer that $u_j \in W^{1,\infty}(\Omega)$ and that  there exists some $C > 0$ such that 
\begin{equation}\label{eq:u_jbound}
    ||u_j||_{W^{1,\infty}} \leq C(||Q_j||_{C^{0,\alpha}}, \lambda_j, ||A_j||_{W^{1,q}} , [\Gamma_j]_{1,\alpha}, \mathrm{dist}(\Gamma_j, \partial \Omega)) \leq C.
\end{equation}
Notice that here we used that by Point (1) of Lemma $\ref{lem:approx}$ one has $\mathrm{dist}(\Gamma_j, \partial\Omega) \rightarrow \mathrm{dist}(\Gamma, \Omega)$, i.e. $\mathrm{dist}(\Gamma_j, \partial \Omega)$ is uniformly bounded above and below. One infers by Lemma \ref{lem:B6} that up to a subsequence (which we do not relabel) there exists $u \in W^{1,\infty}(\Omega)$ such that $u_j \rightarrow u$ uniformly on $\overline{\Omega}.$ We claim that $u$ solves \eqref{eq:1.1}. To this end notice that for all $\phi \in C^2(\overline{\Omega})$ such that $\phi \vert_{\partial \Omega} = 0$ there holds 
\begin{equation}
    \int_\Omega u \; \mathrm{div}(A \nabla \phi) \; \mathrm{d}x = \lim_{j \rightarrow \infty}  \int_\Omega u_j \mathrm{div}(A_j \nabla \phi)  \; \mathrm{d}x
\end{equation}
since $u_j \rightarrow u$ in $L^\infty(\Omega)$ and (by the product rule) $\mathrm{div}(A_j \nabla \phi) \rightarrow \mathrm{div}(A \nabla \phi)$ in $L^q(\Omega)$. Since $u_j$ solves \eqref{eq:auxprob} we infer 
\begin{equation}\label{eq:udiffeq}
    \int_\Omega u \; \mathrm{div}(A \nabla \phi) \; \mathrm{d}x = \lim_{j \rightarrow \infty}  \int_\Omega u_j \mathrm{div}(A_j \nabla \phi)  \; \mathrm{d}x = \lim_{j \rightarrow \infty} \int_{\Gamma_j} Q_j \phi \; \mathrm{d}\mathcal{H}^{n-1}
\end{equation}
Now observe that 
\begin{equation}
    \int_{\Gamma_j} Q_j \phi \; \mathrm{d}\mathcal{H}^{n-1}  =  \int_{\Gamma_j} (Q_j-Q) \phi \; \mathrm{d}\mathcal{H}^{n-1}  + \int_{\Gamma_j} Q \phi \; \mathrm{d}\mathcal{H}^{n-1}.  
\end{equation}
By Lemma \ref{lem:approx} the last term converges to $\int_\Gamma Q \phi \; \mathrm{d}\mathcal{H}^{n-1}$ and the first term can be estimated for large enough $j$ by  
\begin{equation}
    \left\vert \int_{\Gamma_j} (Q_j-Q) \phi \; \mathrm{d}\mathcal{H}^{n-1} \right\vert \leq ||Q_j- Q||_{L^\infty(K)} \mathcal{H}^{n-1}(\Gamma_j) \rightarrow 0 \quad (j \rightarrow \infty) 
\end{equation}
where we used in the last step that $\mathcal{H}^{n-1}(\Gamma_j)$ is bounded uniformly in $j$ (cf. Lemma \ref{lem:C5}). We infer that 
\begin{equation}
    \int_\Omega u \; \mathrm{div}(A\nabla \phi) \; \mathrm{d}x = \int_\Gamma Q \phi \; \mathrm{d}x \quad \forall \phi \in C^2(\overline{\Omega})  :  \phi \vert_{\partial \Omega} = 0 . 
\end{equation}
and hence $u$ is the unique very weak solution of \eqref{eq:1.1}. As we have already discussed we find that $u \in W^{1,\infty}(\Omega)$ is (the unique) solution of \eqref{eq:1.1}. The claim follows.
\end{proof}

\begin{remark}
Examining the proof once again we can infer from \eqref{eq:u_jbound} and the fact that by Lemma \ref{lem:B6} the Lipschitz norm is lower semicontinuous with respect to uniform convergence that for a solution $u$ of $\eqref{eq:1.1}$ with data $A, Q , \Gamma$ as in Main Theorem \ref{thm:main} one has 
\begin{equation}
    ||u||_{W^{1,\infty}} \leq \tilde{C}(||Q||_{C^{0,\alpha}}, \lambda, ||A||_{W^{1,q}} , [\Gamma]_{1,\alpha}, \mathrm{dist}(\Gamma, \partial \Omega)).
\end{equation}
Here we have used that for $(\Gamma_j)_{j \in \mathbb{N}}$  chosen as  in the previous proof one has $\mathrm{dist}(\Gamma_j, \partial\Omega) \rightarrow \mathrm{dist}(\Gamma, \partial \Omega)$ and $[\Gamma_j]_{1,\alpha} \leq C([\Gamma]_{1,\alpha})$. Notice that the latter estimate leads to the complication that $\tilde{C}$ might be larger than the constant in \eqref{eq:u_jbound}. 
\end{remark}

\section{Optimality discussion}\label{sec:optimalitydisc}

\subsection{Optimality of the condition on $\Gamma$}
The regularity requirement that $\Gamma = \partial \Omega' \in C^{1,\alpha}$ can not be substantially weakened. Indeed, we show that for some $\Gamma = \partial \Omega' \in C^{0,1}$ the conclusion of Main Theorem \ref{thm:main} is false.  

For our argument we will use an explicit representation for solutions in terms of the \emph{Green's function} for $-\Delta$.  

\begin{prop}[{\cite[p.44]{Ponce}}]\label{prop:Green}
Let $\Gamma = \partial \Omega' \in C^{0,1}, Q \in L^\infty(\Gamma)$ and let $G_\Omega : \Omega \times \Omega \rightarrow \mathbb{R}\cup \{\infty\}$ be Green's function for $-\Delta$ in $\Omega$. Then the unique solution of
\begin{equation}
    \begin{cases}
        -\Delta u = Q \; \mathcal{H}^{n-1}\mres \Gamma,  & \textrm{in } \Omega \\ \quad    \; \; u = 0 & \textrm{on } \partial \Omega
    \end{cases}
\end{equation}
is given by 
\begin{equation}
    u(x) = \int_\Gamma G_\Omega(x,y) Q(y) \; \mathrm{d}\mathcal{H}^{n-1}(y) \quad \textrm{for } \mathrm{a.e.} \;  x \in \Omega. 
\end{equation}
\end{prop}
Using this we find the following nonexistence result for Lipschitz solutions. 
\begin{lemma}\label{lem:51}
There exists a Lipschitz hypersurface $\Gamma = \partial \Omega' \in C^{0,1}$ such that in $\Omega = B_1(0) \subset \mathbb{R}^2$ the problem 
\begin{equation}\label{eq:probnosol}
    \begin{cases}
        -\Delta u = \mathcal{H}^{n-1}\mres \Gamma & \textrm{in } B_1(0), \\ \quad \; \;  u = 0  & \textrm{on } \partial B_1(0)
    \end{cases}
\end{equation}
admits no solution in $W^{1,\infty}(\Omega)$.
\end{lemma}
\begin{proof}
We look at the following rectangular equilateral triangle
\begin{equation}
    \Gamma := ([0,\tfrac{1}{2}] \times \{0\}) \cup (\{ 0 \} \times [0,\tfrac{1}{2}] ) \cup \{ (t,\tfrac{1}{2}-t) : t \in [0,\tfrac{1}{2}] \} =: \Gamma_1 \cup \Gamma_2 \cup \Gamma_3.
\end{equation}
%be the rectangular triangle compactly contained in $B_1(0)$.
Then clearly $\Gamma = \partial \Omega'$ for a domain $\Omega'\subset \subset B_1(0)$ with Lipschitz boundary. Let now $G : B_1(0) \times B_1(0) \rightarrow \mathbb{R}$ be Green's function for $-\Delta$ in $B_1(0)$. Proposition \ref{prop:Green} yields that then 
\begin{equation}\label{eq:greenrep}
    u(x) = \int_\Gamma G(x,y) \; \mathrm{d}\mathcal{H}^{n-1}(y) \quad \forall x \in B_1(0).
\end{equation}
Now we recall that 
\begin{equation}
    G(x,y) = -\frac{1}{2\pi} \left(  \log |x-y| + h_x(y) \right) \quad \forall x,y \in B_1(0), 
\end{equation}
where $h_x \in C^\infty(\overline{B_1(0)})$ is the unique harmonic function such that $h_x(z) = - \log|x-z|$ for all $z \in \partial B_1(0)$. Let $\epsilon \in ( 0, \tfrac{1}{4}) $ be such that $B_{2\epsilon}(0) \cap \Gamma_3 = \emptyset$. Since $x \mapsto h_x$ depends smoothly on $x$ in $B_\epsilon(0)$ we have by  \eqref{eq:greenrep} that 
\begin{equation}
   -2\pi u(x) = \int_{\Gamma_1} \log|x-y| \; \mathrm{d}\mathcal{H}^{n-1}(y) +  \int_{\Gamma_2} \log|x-y| \; \mathrm{d}\mathcal{H}^{n-1}(y)  + R(x) \quad \forall x \in B_\epsilon(0)
\end{equation}
for some function $R$ that is smooth on $\overline{B_\epsilon(0)}$. This implies that for $x \in B_\epsilon(0) \setminus (\Gamma_1 \cup \Gamma_2)$ there holds 
\begin{equation}
   -2\pi \partial_1 u (x) = \int_{\Gamma_1} \frac{x_1-y_1}{|x-y|^2} \; \mathrm{d}\mathcal{H}^{n-1}(y) + \int_{\Gamma_2} \frac{x_1-y_1}{|x-y|^2} \; \mathrm{d}\mathcal{H}^{n-1}(y) + \partial_1 R(x).
\end{equation}
 Now for $x = (x_1,x_2) \in B_\epsilon(0) \setminus \Gamma$ one has
 \begin{equation}
     \int_{\Gamma_1} \frac{x_1-y_1}{|x-y|^2} \; \mathrm{d}\mathcal{H}^{n-1}(y) = \int_0^{\frac{1}{2}} \frac{x_1 - t}{(x_1-t)^2+ x_2^2} \; \mathrm{d}t = -\frac{1}{2} \log ((x_1-\tfrac{1}{2})^2 + x_2^2) + \frac{1}{2} \log (x_1^2 + x_2^2)
 \end{equation}
 and 
 \begin{equation}
     \int_{\Gamma_2} \frac{x_1-y_1}{|x-y|^2} \; \mathrm{d}\mathcal{H}^{n-1}(y) = \int_0^{\frac{1}{2}} \frac{x_1 }{x_1^2+ (x_2-t)^2} \; \mathrm{d}t = -\mathrm{arctan}\left( \frac{x_2-\tfrac{1}{2}}{x_1} \right) +  \mathrm{arctan}\left( \frac{x_2}{x_1} \right) . 
 \end{equation}
 Notice that the second integral lies in $L^\infty(B_\epsilon(0))$ and the first integral takes the form $\log|x| + \tilde{R}$ for some $\tilde{R} \in L^\infty(B_\epsilon(0))$ (recalling that $\epsilon < \frac{1}{4}< \frac{1}{2}$). We infer that on $B_\epsilon(0) \setminus \Gamma$ there holds 
 \begin{equation}
     \partial_1 u(x) = \log|x| + \bar{R}(x),
 \end{equation}
 where $\bar{R}$ lies in $L^\infty(B_\epsilon(0))$. It follows that $\partial_1 u \not \in L^\infty(B_1(0))$. 
\end{proof}

\begin{remark}
We have discussed in the introduction that for $\alpha \in [0,1)$ one has that the \emph{growth property} $\mu(B_r(x)) \leq Cr^{n-2+\alpha}$ (for all $x \in \mathbb{R}^n$ and $r> 0$) implies that the solution of the measure-valued Dirichlet problem \eqref{eq:measeqgen} lies in $C^{0,\alpha}(\overline{\Omega})$. For Lipschitz manifolds, the growth property of the measure $\mu = \mathcal{H}^{n-1} \mres \Gamma$ is satisfied with $\alpha = 1$, cf. Lemma \ref{lem:measterm}. However as we have seen, the solution does not necessarily lie in $C^{0,1}(\overline{\Omega})$. In \cite[Discussion after Theorem 2.9]{Kilpelainen} it is said that the limit case of the regularity statement for $\alpha = 1$ is not yet well-understood. Our counterexample shows that the growth property is not sufficient to deduce regularity. %However if $\mu$ has a certain regularity, then Lipschitz-regularity can still be expected, as our Main Theorem \ref{thm:main} suggests. %Here we have demonstrated that is fails for general measure $\mu$, so the growth property . 
%If we look at the above example and take into account the growth property derived in Lemma \ref{lem:measurgrowth}
\end{remark}

\begin{remark}
This example can also be seen as an explicit example for failure of the \emph{duality method}  in the case of
$p = \infty$, which we shall explain now. It is a classical result that for all $F \in L^p(\Omega)$, $1< p < \infty$ the equation
\begin{equation}
   \begin{cases}
     -\Delta u = \mathrm{div}(F) & \textrm{in }\Omega,  \\ u = 0 &  \textrm{on }\partial \Omega
   \end{cases}
\end{equation}
has a unique weak solution $u \in W_0^{1,p}(\Omega)$, cf. \cite{Auscher} and references therein. For $F \in L^\infty(\Omega)$, it is in general not true that $u \in W_0^{1,\infty}(\Omega)$. This is shown by a very implicit contradiction argument in \cite[Remark 3.7]{PoissonMarius}. Using the insight above we can now give an explicit counterexample: Let $\Gamma \subset B_1(0)$ be the triangle in the proof of Lemma \ref{lem:51}. By Lemma \ref{lem:measterm} there must exist some $F \in L^\infty(B_1(0))$  such that in the sense of distributions $\mathcal{H}^{n-1} \mres \Gamma = \mathrm{div}(F)$. For such $F$ we have seen in Lemma \ref{lem:51} that 
\begin{equation}
    \begin{cases}
      - \Delta u = \mathrm{div}(F) [= \mathcal{H}^{n-1} \mres \Gamma] & \textrm{in } B_1(0), \\ u = 0 & \textrm{on } \partial B_1(0)
    \end{cases}
\end{equation}
has a solution which does not lie in $W^{1,\infty}(\Omega).$
\end{remark}
\subsection{Optimality of the condition on $A$}
We show in this section that it is not possible to weaken the assumption of Main Theorem \ref{thm:main} to $A \in W^{1,q}(\Omega;\mathbb{R}^{n\times n})$ for any $q<n$. The case $A \in W^{1,n}(\Omega;\mathbb{R}^{n\times n})$ is an interesting limit case and not yet fully understood, but the details would go beyond the scope of this article.  Our example is just a slight modification of \cite[Section 5]{Meyers}, where regularity of (the gradient of) $A$-harmonic functions for irregular coefficients $A$ is studied. The measure term actually has no major impact on the irregularity phenomenon exposed in \cite[Section 5]{Meyers}. 

\begin{lemma} Suppose that $n = 2$. 
Then for each $q <n $ there exists some $A \in W_0^{1,q}(\Omega;\mathbb{R}^{n\times n})$, some $\Gamma = \partial \Omega' \in C^\infty$ and some $Q \in C^\infty(\Gamma)$ such that the solution to \eqref{eq:1.1} does not lie in $W^{1,\infty}(\Omega)$.
\end{lemma}
\begin{proof}
We consider for $\mu \in (0,1)$ the set $\Omega = B_2(0) \subset \mathbb{R}^2$ and $\Gamma = \partial B_1(0)$. We define $A: \Omega \rightarrow \mathbb{R}^{2\times 2}$ via
\begin{equation}
    A(x,y) = 
        \begin{pmatrix} 1- (1-\mu^2) \frac{y^2}{x^2+y^2}  & (1-\mu^2) \frac{xy}{x^2+y^2} \\ (1-\mu^2) \frac{xy}{x^2+y^2} & 1 - (1-\mu^2) \frac{x^2}{x^2+y^2} \end{pmatrix}. 
\end{equation}
One readily checks that $A$ is symmetric and uniformly elliptic and $A \in W^{1,q}(B_2(0),\mathbb{R}^{2\times 2})$ for all $q < 2$.
Next we define 
\begin{equation}
    u_1(x,y) := \frac{x}{(x^2+y^2)^\frac{1-\mu}{2}} \quad (x,y) \in B_2(0).
\end{equation}
Moreover let $u_2 \in C^\infty(\overline{B_2(0)} \setminus B_1(0))$ be the unique solution to the $A$-harmonic Dirichlet problem 
\begin{equation}
    \begin{cases}
        \mathrm{div}(A\nabla u_2) = 0 & \textrm{in } B_2(0) \setminus \overline{B_1(0)},  \\ \qquad \; \;    u_2(x,y) = 0 & \textrm{on }\partial B_1(0),  \\ \qquad \; \;  u_2(x,y) =- \frac{1}{2^{1-\mu}} x & \textrm{on }\partial B_2(0).
    \end{cases}
\end{equation}
Such smooth solution exists as the restriction of $A$ to $\overline{B_2(0)} \setminus B_1(0)$ satisfies $A \in C^\infty(\overline{B_2(0)} \setminus B_1(0))$ and the prescribed boundary data are smooth. We now extend $u_2$ to the whole of $B_2(0)$ by defining $u_2(x,y) = 0$ for all $(x,y) \in B_1(0)$. We will not rename this extension.  Observe that this extension is Lipschitz continuous on $\overline{B_2(0)}$. 
One readily checks that there holds (pointwise) $\mathrm{div}(A \nabla u_1) = 0$ on $B_2(0) \setminus \{0 \}$ and $\mathrm{div}(A \nabla u_2) = 0$ on $B_2(0) \setminus \partial B_1(0)$.
Next we define $u := u_1 + u_2$, which clearly has zero boundary trace on $\partial B_2(0)$. As shown in \cite[Section 5]{Meyers} one has that $\nabla u_1 \in L^p(B_2(0))$ if and only if $p < \frac{2}{1-\mu}$. In particular $\nabla u = \nabla u_1 + \nabla u_2 \not \in L^\infty(B_2(0))$. It remains to show that $u$ solves a problem of the form \eqref{eq:1.1}. 
As $2 < \frac{2}{1-\mu}$ we obtain that $u \in W_0^{1,2}(B_2(0))$ and check that for $\phi \in C_0^\infty(B_2(0))$ there holds
\begin{align}
    \int_\Omega (A \nabla u, \nabla \phi) \; \mathrm{d}x  & = \int_{B_2(0)} (A \nabla u_1 , \nabla \phi) + \int_{B_2(0) \setminus B_1(0)} (A \nabla u_2 , \nabla \phi) 
     \\  & = \lim_{\epsilon \rightarrow 0} \int_{B_2(0) \setminus B_\epsilon(0)} (A \nabla u_1, \nabla \phi) + \int_{\partial (B_2(0) \setminus B_1(0))} (A \nabla u_2 , \nu) \phi \; \mathrm{d}\mathcal{H}^{n-1} 
     \\ & = -  \lim_{\epsilon \rightarrow 0} \int_{\partial B_\epsilon(0)} (A \nabla u_1 , \nu) \phi \; \mathrm{d} \mathcal{H}^{n-1} - \int_{\partial B_1(0)} (A \nabla u_2 , (x,y)^T) \phi \; \mathrm{d}\mathcal{H}^{n-1}.
\end{align}
Noticing next that $A \in L^\infty(B_2(0);\mathbb{R}^{2\times 2})$ we infer that 
\begin{align}
    \left\vert  \int_{\partial B_\epsilon(0)} (A \nabla u_1 , \nu) \phi \; \mathrm{d} \mathcal{H}^{n-1} \right\vert  & \leq ||A||_{L^\infty} ||\phi||_{L^\infty} \int_{\partial B_\epsilon(0)} |\nabla u_1| \; \mathrm{d}\mathcal{H}^{n-1}  \\ & \leq \frac{C||A||_\infty||\phi||_\infty}{\epsilon^{1-\mu}} (2\pi \epsilon) \rightarrow 0  \quad (\epsilon \rightarrow 0). 
\end{align}
In particular we obtain that for $Q := -(A \nabla u_2 , (x,y)^T) \vert_{\partial B_1(0)} \in C^\infty( \partial B_1(0))$ one has 
\begin{equation}
    \int_\Omega (A \nabla u , \nabla \phi) \; \mathrm{d}x = \int_{\partial B_1(0)} Q \phi \; \mathrm{d}\mathcal{H}^{n-1} \quad \forall \phi \in C_0^\infty(\Omega) .
\end{equation}
By density and the fact that $A \in L^\infty$ this also holds true for $\phi \in W_0^{1,2}(\Omega)$.
In particular  for arbitrary $\phi \in W^{2,2}(\Omega) \cap W_0^{1,2}(\Omega)$ we may integrate by parts and obtain (as $u$ has zero boundary trace)
\begin{equation}
    \int_\Omega u  \; \mathrm{div}(A \nabla \phi)  \; \mathrm{d}x = \int_{\partial B_1(0)} Q \phi \;\mathrm{d}\mathcal{H}^{n-1}  \quad \forall \phi \in W^{2,2}(\Omega) \cap W_0^{1,2}(\Omega),
\end{equation}
showing that $u$ is a weak solution of an equation of type \eqref{eq:1.1}. The claim follows.
\end{proof}

\begin{remark}
The limit case $\mu \rightarrow 0$ in the previous example shows actually that for general (symmetric and uniformly elliptic) $A \in L^\infty$ no more regularity than $W_0^{1,2}(\Omega)$ can be obtained.
%Indeed, the regularity of the solution constructed above is $W_0^{1,p}$ for all $p \in [ 2, \tfrac{2}{1-\mu})$.
In particular we observe that the best possible regularity depends vitally on the $W^{1,q}$-regularity of the coefficients in $A$. For $q < n $ the optimal regularity is $W_0^{1,2}( \Omega)$ and for $q > n$ the optimal regularity is $W_0^{1,\infty}(\Omega)$. This jump shows again that the case $q=n$ is very interesting.   
\end{remark}

\section{Applications to Free boundary problems}\label{sec:applic}
%\subsection{The thin obstacle problem}
%
Equations like \eqref{eq:1.1} appear in many free boundary problems, for example in the \emph{thin obstacle problem}, cf \cite[Equation 3.8]{Fernandez}. Here we illustrate two further examples specifically, both of which consider \emph{adhesive free boundary problems}.
\subsection{The Bernoulli problem}
A solution of Bernoulli's problem (in the sense of \cite{Henrot}) is a pair consisting of a suitably smooth function $u: \Omega \rightarrow \mathbb{R}$  and a suitably regular subset $S \subset \subset \Omega$ that satisfies 
\begin{equation}\label{eq:Bern}
    \begin{cases}
      \Delta u = 0 & \textrm{in }\Omega \setminus \overline{S}, \\ \quad u = 0 & \textrm{on } \partial \Omega, \\  \quad u = 1 & \textrm{on }\partial S, \\
     \; \;   \frac{\partial u}{\partial \nu} = Q & \textrm{on }S. 
    \end{cases}
\end{equation}
In \cite{Henrot} it is observed that one can find a one-parameter family of solutions around a given (suitably regular, hyperbolic) \emph{prototype solution}. This is done by means of a carefully chosen \emph{evolution equation}. For the derivation of the evolution law it is possible to profit from \cite[Proposition 4.3]{Henrot}, saying that $u$ is a solution  of \eqref{eq:Bern} if and only if 
\begin{equation}\label{eq:Alti}
    \begin{cases}
     - \Delta u = Q \; \mathcal{H}^{n-1} \mres \partial S  & \textrm{in } \Omega, \\
\quad \; \;       u = 0 & \textrm{on } \partial \Omega, \\  \quad \; \;  u \equiv 1 & \textrm{on }S.
    \end{cases}
\end{equation}
This result requires some regularity of $S$ and uses the \emph{normal jump of the single layer potential}, cf. Remark \ref{rem:4.3}. Regularity of \eqref{eq:Alti} can be studied with the methods developed in this article, with optimal regularity requirements on $S$. Possibly, this can also help for the study of the derived evolution equation under less regularity assumptions on the needed prototype solution. Moreover, the Bernoulli problem can be considered for other operators than $-\Delta$. 
%Our result paves the way for the equivalence of these under less regul 
%In particular 

%With our regularity result one can show
%that all solutions of \eqref{eq:Alti} lie in $W^{1,\infty}(\Omega)$
%this equivalence holds true under less regularity assumptions on $S$. also pave the way for further study of elliptic operators, as e.g. done variationally already in previous work, cf. \cite{Ferrari}. 

\subsection{The biharmonic Alt-Caffarelli Problem}
Recently, the \emph{biharmonic Alt-Caffarelli problem} has raised a lot of interest, cf. \cite{Serena1}, \cite{Serena2}, \cite{AltCaffarelliMarius}, \cite{AltCaffarelliMariusGrunau}. This is the minimization of 
\begin{equation}
    \mathcal{E}(u) := \int_\Omega (\Delta u)^2 \; \mathrm{d}x + |\{ x \in \Omega : u(x) > 0 \}|
\end{equation}
among all $u \in W^{2,2}(\Omega)$ such that $u \vert_{\partial \Omega} = u_0$ for some $u_0 \in C^\infty(\overline{\Omega}), u_0 > 0$. Here $| \cdot |$ denotes the Lebesgue measure. Minimizers of this problem have to find an optimal balance between \emph{bending} (measured by the first summand) and \emph{positivity} (measured by the second summand). Since $u_0 > 0$ it will however require some bending to leave the positivity region, which is why the interests of both terms are conflicting. 
The set $\Gamma := \{ u = 0 \}$ is where the functional loses its regularity and is hence considered the \emph{free boundary} of the problem. In \cite[Theorem 1.4]{AltCaffarelliMarius} it is shown that in dimension two (i.e. if  $\Omega \subset \mathbb{R}^2$) minimizers lie in $C^2(\overline{\Omega})$ and satisfy $\nabla u \neq 0$ on $\Gamma = \{ u = 0\}$. Moreover $\Gamma = \partial \Omega'$ for some $C^2$-smooth domain $\Omega' \subset \subset \Omega$. Additionally, each minimizer $u$ satisfies the equation 
\begin{equation}\label{eq:Eulieq}
    \int_\Omega \Delta u \Delta \phi \; \mathrm{d}x = \int_\Gamma \frac{1}{|\nabla u|} \phi \; \mathrm{d}x \quad \forall \phi \in W^{2,2}(\Omega) \cap W_0^{1,2}(\Omega). 
\end{equation}
Our results can therefore be used to study further regularity of minimizers. Indeed, if we define  $w := \Delta u \in L^2(\Omega)$  then \eqref{eq:Eulieq} yields that $w$ is a very weak solution of  
\begin{equation}
    \begin{cases}
-\Delta w = Q \; \mathcal{H}^{n-1}\mres \Gamma & \textrm{in }\Omega, \\ w = 0 & \textrm{on }\partial \Omega, 
    \end{cases}
\end{equation}
with $Q := \frac{1}{|\nabla u|} \in C^1(\Gamma)$. We infer from our Main Theorem \ref{thm:main} that $w = \Delta u \in W^{1,\infty}(\Omega)$. By elliptic regularity this also implies $u \in W^{3,p}(\Omega)$ for all $p \in (1,\infty)$. It will be subject of future research whether also $u \in W^{3,\infty}(\Omega)$. This would make $\Gamma = \{ u = 0 \}$ a $C^{2,1}$-manifold as $\nabla u$ does not vanish on the level set $\Gamma$. We remark also that Theorem \ref{thm:main} paves the way for the study of the \emph{$A-$biharmonic Alt-Caffarelli problem}, i.e. the minimization of 
\begin{equation}
    \mathcal{E}(u) := \int_\Omega |\mathrm{div}(A(x) \nabla u(x)) |^2 \; \mathrm{d}x + |\{x \in \Omega:  u(x) > 0 \}| 
\end{equation}
under the same boundary conditions as above. 
%Regularity statements   that 
\appendix

\section{Very weak solutions}\label{app:weakform}

In this appendix we show that (under mild conditions) very weak solutions are also weak solutions, as already discussed in the introduction. In the case of $A= Id_{n\times n}$ this is already studied in \cite[Section 3.1]{Ponce}. The mild condition we need additionally is that $A \in W^{1,q}(\Omega;\mathbb{R}^{n\times n})$ for some $q> n$. 

\begin{lemma}\label{lem:stronsolution}
Let $\Omega \subset \mathbb{R}^n$ be a smooth domain, $n \geq 2$. Suppose that $A \in W^{1,q}(\Omega;\mathbb{R}^{n\times n})$ for some $q >  n$ and let $p \in [2,n]$. Then for each $g \in L^p(\Omega)$ there exists a unique $\psi \in W^{2,p}(\Omega) \cap W_0^{1,p}(\Omega)$ such that 
\begin{equation}
   \mathrm{div}(A \nabla \psi ) = g \quad \textrm{pointwise a.e..}
\end{equation}
\end{lemma}

\begin{proof}
%Choose $s \in (1,\infty)$ such that $s \geq 2$ and
% \begin{equation}
% \begin{cases}
%     \frac{pq}{q-p} < s < \frac{pn}{n-p} & p < n, \\
%     \frac{pq}{q-p} < s < \infty & p = n.
%     \end{cases}
% \end{equation}
% Notice that achieving $s \geq 2$ is always possible as $\frac{pn}{n-p}>  p \geq 2$.
 Let $(A^{(k)})_{k \in \mathbb{N}} \subset C^\infty(\overline{\Omega};\mathbb{R}^{n\times n})$ be such that $A^{(k)} \rightarrow A$ in $W^{1,q}(\Omega;\mathbb{R}^{n\times n})$. We denote the coefficients of $A^{(k)}$ by $a_{ij}^{(k)}$ for $i,j = 1,...,n$. By elliptic regularity (cf. \cite[Theorem 9.15]{GilTru}) there exist $\psi_k \in W^{2,p}(\Omega) \cap W_0^{1,p}(\Omega)$ such that (pointwise a.e.) there holds
 \begin{equation}\label{eq:DGL}
     \sum_{i,j = 1}^n  a_{ij}^{(k)} \partial^2_{ij} \psi_k  + (\mathrm{div}(A^{(k)}),\nabla \psi_k) = g,  
 \end{equation}
 i.e. pointwise a.e. there holds
\begin{equation}\label{eq:pointwisedivi}
    \mathrm{div}(A^{(k)} \nabla \psi_k) = g.
\end{equation} 
 We define next the differential operator
 \begin{equation}
     L_k u := \sum_{i,j = 1}^n  a_{ij}^{(k)} \partial^2_{ij} u.
 \end{equation}
 For all $k \in \mathbb{N}$, $L_k$ satisfies the prerequisites of \cite[Lemma 9.17]{GilTru} and therefore there exist constants $C_k$ (cf. \cite[Lemma 9.15]{GilTru}) depending only on $\Omega$, $p$, the ellipticity constant of $A^{(k)}$ and the $L^\infty$-norm of $A^{(k)}$ and the moduli of continuity of $A^{(k)}$ such that 
 \begin{equation}
     ||u||_{W^{2,p}} \leq C_k ||L_k u||_{L^p} \quad \forall u \in W^{2,p}(\Omega) \cap W_0^{1,p}(\Omega).
 \end{equation}
 As $A^{(k)}$ converges to $A$ in $C^{0,\gamma}$ for  $\gamma = 1 - \frac{n}{q} > 0$ the moduli of continuity and the ellipticity constants can all be uniformly controlled in $k$ and hence one can find a uniform $\bar{C} > 0$ such that for all $k \in \mathbb{N}$ there holds
 \begin{equation}\label{eq:eLLIPTICestimate}
     ||u||_{W^{2,p}} \leq \bar{C} ||L_k u||_{L^p} \quad \forall u \in W^{2,p}(\Omega) \cap W_0^{1,p}(\Omega).
 \end{equation}
 The PDE \eqref{eq:DGL} can now also be viewed as follows 
 \begin{equation}
     L_k \psi_k =   g - (\mathrm{div}(A^{(k)}),\nabla \psi_k) . 
 \end{equation}
  We infer from \eqref{eq:eLLIPTICestimate} and Hölder's inequality
  \begin{align}
      ||\psi_k||_{W^{2,p}} & \leq \bar{C} ||L_k \psi_k||_{L^p} \leq  \bar{C} || g- (\mathrm{div}(A^{k}),\nabla \psi_k) ||_{L^p}
      \\ & \leq  \bar{C} ( ||g||_{L^p} + ||(\mathrm{div}(A^{(k)}), \nabla \psi_k )||_{L^p})
      \\ & \leq  \bar{C} ( ||g||_{L^p} + ||\mathrm{div}(A^{(k)})||_{L^q} || \nabla \psi_k ||_{L^{\frac{qp}{q-p}}}) \\ &  =  \bar{C} ( ||g||_{L^p} + ||\mathrm{div}(A^{(k)})||_{L^q} ||  \psi_k ||_{W_0^{1,\frac{qp}{q-p}}}), \label{eq:W2peesti}
  \end{align}
  We next intend to bound the right hand side uniformly in $k$. 
  Notice that $q >n$ implies that $\frac{qp}{q-p} < \frac{np}{n-p}$ and hence $W^{2,p}(\Omega) \cap W_0^{1,p}(\Omega)$ embeds compactly into
  $W_0^{1, \frac{qp}{q-p}}(\Omega)$.  Since also $\frac{qp}{q-p}\geq p \geq 2$ we observe that 
  \begin{equation}
      W^{2,p}(\Omega) \cap W_0^{1,p}(\Omega) \overset{c}{\hookrightarrow} W_0^{1, \frac{qp}{q-p}}(\Omega) \hookrightarrow W_0^{1,2}(\Omega) 
  \end{equation}
  forms a \emph{Banach triple} (i.e. the first embedding is compact and the second is continuous). It follows that for each $\epsilon> 0$ there exists $D(\epsilon) > 0$ independent of $k$ such that 
  \begin{equation}
    ||  \psi_k ||_{W_0^{1,\frac{qp}{q-p}}} \leq \epsilon ||\psi_k||_{W^{2,p}} + D(\epsilon) ||\psi_k||_{W_0^{1,2}}. 
  \end{equation}
  Using this in \eqref{eq:W2peesti} with $\epsilon := \frac{1}{2 \bar{C} \sup_{k \in \mathbb{N}} ||\mathrm{div}(A^{(k)})||_{L^q} + 1} $ (and calling $D := D(\epsilon) = D(\bar{C}, ||A||_{W^{1,q}})$) we infer from \eqref{eq:W2peesti} that 
  \begin{equation}
      ||\psi_k||_{W^{2,p}} \leq \frac{1}{2}||\psi_k||_{W^{2,p}} +  \bar{C} ||g||_{L^p} + \bar{C}D ||\mathrm{div}(A^{(k)})||_{L^q} ||\psi_k||_{W_0^{1,2}}.
  \end{equation}
  In particular
  \begin{equation}\label{eq:weezweipee}
      ||\psi_k||_{W^{2,p}} \leq 2 \bar{C}( 1+ D) ( ||g||_{L^p} + ||\mathrm{div}(A^{(k)})||_{L^q} ||\psi_k||_{W_0^{1,2}}). 
  \end{equation}
  It remains to bound $||\psi_k||_{W_0^{1,2}}$ uniformly in $k$. To that end we use \eqref{eq:pointwisedivi}, 
  %Let $(\eta_{j})_{j \in \mathbb{N}} \subset C_0^\infty(\Omega)$ such that $\eta_j \rightarrow \psi_k$ in $W_0^{1,2}(\Omega)$ as $j \rightarrow \infty$. Then \eqref{eq:pointwisedivi}
  which yields (if $\lambda_k$ is the ellipticity constant of $A^{(k)}$)
  \begin{align}
      \int_\Omega |\nabla \psi_k|^2 \; \mathrm{d}x  & \leq \frac{1}{\lambda_k} \int_\Omega (A^{(k)}\nabla \psi_k, \nabla \psi_k) \; \mathrm{d}x  
       =   -\frac{1}{\lambda_k} \int_\Omega \mathrm{div}(A^{(k)} \nabla \psi_k) \psi_k \; \mathrm{d}x  = -\frac{1}{\lambda_k} \int_\Omega g \psi_k \; \mathrm{d}x.
  \end{align}
   We infer 
  \begin{equation}
      ||\psi_k||_{W_0^{1,2}}^2 \leq \frac{C_1(\Omega)}{\lambda_k} ||g||_{L^p}||\psi_k||_{L^2} \leq \frac{C_2(\Omega)}{\lambda_k} ||g||_{L^p} ||\psi_k||_{W_0^{1,2}},
  \end{equation}
  where $C_i(\Omega),i=1,2$ depend on the volume and the optimal Poincare constant of $\Omega$. Since $A^{(k)} \rightarrow A$ in $C^{0,\gamma}$ one can also bound the reciprocals of the ellipticity constants $\frac{1}{\lambda_k}$ uniformly in $k$ by $\frac{1}{\tilde{\lambda}}$. We infer 
  \begin{equation}
      ||\psi_k||_{W_0^{1,2}} \leq \frac{C_2(\Omega)}{\tilde{\lambda}} ||g||_{L^p}.
  \end{equation}
  Using this uniform bound in \eqref{eq:weezweipee} and the fact that $||\mathrm{div}(A^{(k)})||_{L^q}$ is uniformly bounded in $k$ we infer that $(\psi_k)_{k \in \mathbb{N}}$ is uniformly bounded in $W^{2,p}(\Omega)\cap W_0^{1,p}(\Omega)$. Hence we may extract a weakly convergent subsequence (which we do not relabel) and a limit $\psi \in W^{2,p}(\Omega) \cap W_0^{1,p}(\Omega)$ such that $\psi_k \rightharpoonup \psi$ weakly in $W^{2,p}$. We show now that $\mathrm{div}(A\nabla \psi) = g$ pointwise a.e.. To this end fix $\eta \in C_0^\infty(\Omega)$. Then using that $A^{(k)} \rightarrow A$ in $C^{0,\gamma}$ and $\psi_k \rightarrow \psi$ in $W^{1,p}$ we find 
  \begin{align}
      \int_\Omega \mathrm{div}(A \nabla \psi) \eta  & =- \int_\Omega (A \nabla \psi , \nabla \eta)  \; \mathrm{d}x  = -\lim_{k \rightarrow \infty} \int_\Omega (A^{(k)} \nabla \psi_k , \nabla \eta) \; \mathrm{d}x \\ &= \lim_{k \rightarrow \infty} \int_\Omega \mathrm{div}(A^{(k)} \nabla \psi_k)  \eta  \; \mathrm{d}x = \lim_{k \rightarrow \infty} \int_\Omega g  \eta \; \mathrm{d}x =  \int_\Omega g  \eta \; \mathrm{d}x.
  \end{align}
  Since $\eta \in C_0^\infty(\Omega)$ was arbitrary we conclude $\mathrm{div}(A \nabla \psi ) = g$ pointwise a.e.. \\
  \textbf{Step 2.} Uniqueness. By linearity it suffices to show that for each $\psi \in W^{2,p}(\Omega) \cap W_0^{1,p}(\Omega)$ such that $\mathrm{div}(A \nabla \psi) = 0$ a.e. in $\Omega$ there must hold $\psi = 0$ a.e. in $\Omega$. To this end we choose $\eta \in C_0^\infty(\Omega)$ arbitrary and infer 
  \begin{equation}
      0 = \int_{\Omega} \mathrm{div}(A \nabla \psi) \eta \; \mathrm{d}x = \int_\Omega (A \nabla \psi, \nabla \eta) \; \mathrm{d}x. 
  \end{equation}
  Approximating $\psi$ in $W_0^{1,2}$ with $C_0^\infty(\Omega)$- functions we infer 
  \begin{equation}
      0 = \int_\Omega (A \nabla \psi , \nabla \psi) \; \mathrm{d}x \geq \lambda \int_\Omega |\nabla \psi|^2 \; \mathrm{d}x, 
  \end{equation}
  whereupon we conclude that  $\psi \equiv \mathrm{const}.$ The fact that $\psi \in W_0^{1,p}(\Omega)$ yields that $\psi \equiv 0$. 
\end{proof}

\begin{lemma}\label{lem:weakveryweak}
Let $\Omega \subset \mathbb{R}^n, n \geq 2$ be a smooth domain and $\Gamma \subset \subset \Omega$ be a Lipschitz hypersurface. Further assume $Q \in L^\infty(\Gamma)$ and $A \in W^{1,q}(\Omega;\mathbb{R}^{n\times n})$. Let $u \in L^2(\Omega)$. Then the following are equivalent 

    \begin{equation}
    \mathrm{(1)} \quad -\int_\Omega u \;  \mathrm{div}(A \nabla \phi) \; \mathrm{d}x = \int_\Gamma Q \phi \; \mathrm{d}\mathcal{H}^{n-1} \quad \forall \phi \in W^{2,2}(\Omega)\cap W_0^{1,2}(\Omega), \quad \quad \quad 
    \end{equation}
    \begin{equation}
      \mathrm{(1')} \quad \; \; -\int_\Omega u \;  \mathrm{div}(A \nabla \phi) \; \mathrm{d}x = \int_\Gamma Q \phi \; \mathrm{d}\mathcal{H}^{n-1} \quad \forall \phi \in C^2(\overline{\Omega}) : \phi\vert_{\partial \Omega} = 0, \quad \quad \quad \; 
    \end{equation}
    \begin{equation}\label{eq:weaksmooth}
    \mathrm{(2)} \quad \textrm{ $u \in W_0^{1,2}(\Omega)$ and } \quad    \int_\Omega (A \nabla u , \nabla \phi) \; \mathrm{d}x = \int_\Gamma Q \phi \; \mathrm{d}x \quad \forall \phi \in C_0^\infty(\Omega). \; 
    \end{equation}
\end{lemma}
\begin{proof}
Equivalence of $(1)$ and $(1')$ follows immediately by density, cf. Lemma \ref{lem:closure}. 
%We prove the claim first under the additional assumption that $A \in C^\infty(\overline{\Omega};\mathbb{R}^{n\times n})$. For the general case we will argue by approximation. 
We show first $(1) \Rightarrow (2)$. For fixed $f \in C_0^\infty(\Omega; \mathbb{R}^n)$ observe that by Lemma \ref{lem:stronsolution} there exists some $\phi \in W^{2,2}(\Omega) \cap W_0^{1,2}(\Omega)$, such that there holds (pointwise a.e.) 
\begin{equation}
    \mathrm{div}(A \nabla \phi) = \mathrm{div}(f).
\end{equation}
Multiplying by some $\psi \in C_0^\infty(\Omega)$ and integrating we find
\begin{equation}
    \int_\Omega (A(x) \nabla \phi, \nabla \psi) \; \mathrm{d}x = -\int_\Omega \mathrm{div}(f) \psi \; \mathrm{d}x \quad \forall \psi \in C_0^\infty(\Omega). 
\end{equation}
Taking the closure, we find that this equation holds also true for all $\psi \in W_0^{1,2}(\Omega)$. 
%Notice that by (SCHECHTER)  there holds $||\phi||_{W^{1,2}(\Omega)} \leq C || f ||_{L^2(\Omega; \mathbb{R}^n)}$.
We observe that 
\begin{equation}
    \lambda \int_\Omega |\nabla \phi|^2 \; \mathrm{d}x \leq \int_\Omega (A(x) \nabla \phi , \nabla \phi) \; \mathrm{d}x = - \int_\Omega \mathrm{div}(f) \phi \; \mathrm{d}x = \int_\Omega (f ,\nabla \phi) \; \mathrm{d}x. 
\end{equation}
The Cauchy-Schwarz inequality yields then  $||\phi||_{W_0^{1,2}}  = ||\nabla \phi||_{L^2} \leq \frac{1}{\lambda} ||f||_{L^2}$. 
For the computation to come we recall that by Lemma \ref{lem:measterm}
\begin{equation}
    T(\psi) := \int_\Gamma Q \psi \; \mathrm{d}\mathcal{H}^{n-1}  \quad ( \psi \in C_0^\infty(\Omega))
\end{equation}
extends to a linear continuous functional in $(W_0^{1,2}(\Omega))^*$. Using this we find \begin{align}
   - \int_\Omega u \; \mathrm{div}(f) \; \mathrm{d}x &=- \int_\Omega u \;  \mathrm{div}(A \nabla \phi) \; \mathrm{d}x = \int_\Gamma Q \phi \; \mathrm{d}\mathcal{H}^{n-1} \leq C_1||\phi||_{W_0^{1,2}(\Omega)} \leq C_2 ||f||_{L^2(\Omega; \mathbb{R}^n)}.
\end{align}
In particular we infer that 
\begin{equation}
    S: C_0^\infty(\Omega;\mathbb{R}^n) \rightarrow \mathbb{R}, \quad S(f) := - \int_\Omega u \;  \mathrm{div}(f) \; \mathrm{d}x 
\end{equation}
extends to a functional in $L^2(\Omega;\mathbb{R}^n)^*$, whereupon one finds $g \in L^2(\Omega;\mathbb{R}^n)$ such that 
\begin{equation}
    -\int_\Omega u \; \mathrm{div}(f) \; \mathrm{d}x = \int_\Omega (g,f) \; \mathrm{d}x \quad \forall f \in C_0^\infty(\Omega; \mathbb{R}^n).
\end{equation}
Setting $f= v e_j$ for $v \in C_0^\infty(\Omega)$ and $j = 1,...,n$ one readily checks that $u$ is weakly differentiable and $\nabla u = g \in L^2(\Omega)$. We have obtained that $u \in W^{1,2}(\Omega)$. We next show \eqref{eq:weaksmooth}. %and thereupon $\mathrm{tr}_{\partial \Omega} (u) = 0$, where $\mathrm{tr}_{\partial \Omega}$ denotes the trace operator on $W^{1,2}(\Omega)$. 
First notice that by smoothness of $\Omega$ there exist $(u_n)_{n \in \mathbb{N}} \subset C^\infty(\overline{\Omega})$ such that $u_n \rightarrow u$ in $W^{1,2}(\Omega)$. By the classical Gauss divergence theorem we then infer for each $\phi \in C^2(\overline{\Omega})$ s.t. $\phi \big\vert_{\partial \Omega} = 0$ that 
\begin{align}
   \int_{\Gamma} Q \phi \; \mathrm{d}\mathcal{H}^{n-1} & =  -\int_\Omega u \mathrm{div}(A \nabla \phi) \; \mathrm{d}x  = -\lim_{n \rightarrow \infty} \int_\Omega u_n \mathrm{div}(A \nabla \phi) \; \mathrm{d}x \\ &  =   \lim_{n \rightarrow \infty } \left( - \int_{\partial \Omega} u_n (A(x) \nabla \phi , \nu_\Omega) \; \mathrm{d}\mathcal{H}^{n-1} + \int_\Omega (A(x) \nabla u_n, \nabla \phi) \; \mathrm{d}x \right)
    \\ & =     - \int_{\partial \Omega} \mathrm{tr}_{\partial \Omega} (u)  (A(x) \nabla \phi , \nu_\Omega) \; \mathrm{d}\mathcal{H}^{n-1} + \int_\Omega (A(x) \nabla u, \nabla \phi) \; \mathrm{d}x,
\end{align}
where we have used continuity of the Sobolev trace operator in the last step. To simplify notation we write only $u$ instead of $\mathrm{tr}_{\partial \Omega} (u)$ in the sequel. 
Plugging in any $\phi \in C_0^\infty(\Omega)$ we infer that the integral formula in \eqref{eq:weaksmooth} holds true. It remains to show that $u \in W_0^{1,2}(\Omega)$, i.e. $\mathrm{tr}_{\partial \Omega}(u) = 0$. To that end look at $\phi = \psi d_\Omega$, where $d_\Omega: B_\epsilon(\partial \Omega) \rightarrow \mathbb{R}$, is the signed distance function for $\partial \Omega$, $\epsilon< \mathrm{dist}(\Gamma,\partial \Omega)$ is chosen small enough such that $d_\Omega$ is smooth, and $\psi \in C_0^\infty(B_\epsilon(\partial \Omega))$ is arbitrary. Then we have 
\begin{equation}\label{eq:zerotraceint}
    0 = \int_\Gamma Q ( \psi d_\Omega) \; \mathrm{d}\mathcal{H}^{n-1} = - \int_{\partial \Omega} u (A(x) \nabla (\psi d_\Omega), \nu_\Omega) \; \mathrm{d}\mathcal{H}^{n-1} + \int_\Omega (A(x) \nabla u , \nabla (\psi d_\Omega) ) \; \mathrm{d}x.
\end{equation}
Since $\nabla d_\Omega = \nu_\Omega$ on $\partial \Omega$ we infer 
\begin{equation}
     \int_{\partial \Omega} u (A(x) \nabla (\psi d_\Omega), \nu_\Omega) \; \mathrm{d}\mathcal{H}^{n-1} = \int_{\partial \Omega} u (A(x) \nu_\Omega, \nu_\Omega) \psi \; \mathrm{d}\mathcal{H}^{n-1}. 
\end{equation}
Moreover, observe that $\psi d_\Omega \big\vert_{\Omega} \in W_0^{1,2}(\Omega) \cap C(\overline{\Omega})$ and hence there exists $(\eta_m)_{m\in \mathbb{N}} \subset C_0^\infty(\Omega)$ such that $\eta_m \rightarrow \psi d_\Omega$ in $W^{1,2}(\Omega)$ and also uniformly, see Lemma \ref{lem:C0andW12}. 
Therefore we can rearrange the last summand in \eqref{eq:zerotraceint} (by using the already derived \eqref{eq:weaksmooth})
\begin{align}
    \int_\Omega (A(x) \nabla u , \nabla ( \psi d_\Omega) ) \; \mathrm{d}x & = \lim_{m \rightarrow \infty} \int_\Omega (A(x) \nabla u , \nabla \eta_m ) \; \mathrm{d}x  \\ & = \lim_{m \rightarrow \infty} \int_\Gamma Q \eta_m \; \mathrm{d}\mathcal{H}^{n-1} =  \int_\Gamma Q (\psi d_\Omega) \; \mathrm{d}\mathcal{H}^{n-1}= 0.
\end{align}
 Hence \eqref{eq:zerotraceint} yields \begin{equation}
     0 = \int_{\partial \Omega} u ( A(x) \nu_\Omega , \nu_\Omega) \psi \; \mathrm{d}\mathcal{H}^{n-1} \quad \forall \psi \in C_0^\infty(B_\epsilon(\partial \Omega)) .
 \end{equation}
 This implies $u (A(x) \nu_\Omega, \nu_\Omega) = 0$ a.e. on $\partial \Omega$ and thus by ellipticity of $A$ one infers $u= 0$ on $\partial \Omega$. In particular $u \in W_0^{1,2}(\Omega)$ and the claim is shown. Next we turn to $(2) \Rightarrow (1)$. If $u \in W_0^{1,2}(\Omega)$ then there exists $(u_k)_{k\in \mathbb{N}} \subset C_0^\infty(\Omega)$ such that $u_k \rightarrow u$ in $W^{1,2}(\Omega)$. For each $\phi \in C^2(\overline{\Omega})$ s.t. $\phi\big\vert_{\partial \Omega} = 0$ we then infer 
 \begin{align}
    - \int_\Omega u \;  \mathrm{div}(A(x) \nabla \phi) \; \mathrm{d}x & = -\lim_{k \rightarrow \infty} \int_\Omega u_k \;  \mathrm{div}(A(x) \nabla \phi) \; \mathrm{d}x = \lim_{k \rightarrow \infty} \int_\Omega (A(x) \nabla u_k, \nabla \phi) \; \mathrm{d}x 
    \\ & = \int_\Omega (A(x) \nabla u , \nabla \phi ) \; \mathrm{d}x.
 \end{align}
 Now we observe that $\phi \in W_0^{1,2}(\Omega) \cap C(\overline{\Omega})$ and we choose (due to Lemma \ref{lem:C0andW12}) $(\phi_n)_{n \in \mathbb{N}} \subset C_0^\infty(\Omega)$ such that $\phi_n \rightarrow \phi$ in $W^{1,2}(\Omega)$ and uniformly. Hence 
 \begin{equation}
    \int_\Omega (A(x) \nabla u , \nabla \phi ) \; \mathrm{d}x = \lim_{n \rightarrow \infty} \int_\Omega (A(x) \nabla u , \nabla \phi_n) \; \mathrm{d}x = \int_\Gamma Q \phi_n \; \mathrm{d}\mathcal{H}^{n-1} = \int_\Gamma Q \phi \; \mathrm{d}\mathcal{H}^{n-1}. 
 \end{equation}
 The previous two equations yield 
 \begin{equation}
     - \int_\Omega u \;  \mathrm{div}(A(x) \nabla \phi) \; \mathrm{d}x = \int_\Gamma Q \phi \; \mathrm{d}\mathcal{H}^{n-1} \quad \forall \phi \in  C^2(\overline{\Omega}) : \phi \big\vert_{\partial \Omega} = 0. 
 \end{equation}
\end{proof}

\section{Some approximation and regularity results}

Here we collect some technical lemmas that have been useful for the course of our argument.

\begin{lemma}\label{lem:C0andW12}
Let $\Omega \subset \mathbb{R}^n$ be $C^1$-smooth and $\eta \in W_0^{1,q}(\Omega) \cap C(\overline{\Omega})$ for some $q \in (1,\infty)$. Then there exists a sequence $(\eta_j)_{j \in \mathbb{N}} \subset C_0^\infty(\Omega)$ such that $\eta_j \rightarrow \eta$ in $W_0^{1,q}(\Omega)$ and $\eta_j \rightarrow \eta$ in $C(\overline{\Omega})$. 
\end{lemma}
\begin{proof}
For each $j \in \mathbb{N}$ the function 
\begin{equation}
    \tilde{\eta}_j = \left( \eta - \frac{1}{j} \right)^+ - \left( \eta + \frac{1}{j} \right)^-
 \end{equation}
 has compact support in $\Omega$. Clearly, $\tilde{\eta}_j \rightarrow \eta$ in $C(\overline{\Omega})$. Moreover, by the dominated convergence theorem and Stampacchia's Lemma (cf. \cite[Lemma A.4,Chapter 2]{KiStam}) one infers 
 \begin{equation}
     \int_\Omega | \nabla \tilde{\eta}_j - \nabla \eta|^q \; \mathrm{d}x = \int_{ \{|\eta| < \frac{1}{j} \}} |\nabla \eta|^q \; \mathrm{d}x \rightarrow \int_{\{ \eta = 0 \}} |\nabla \eta|^q \; \mathrm{d}x = 0. \quad (j \rightarrow \infty).
  \end{equation}
This also implies that $\tilde{\eta}_j \rightarrow \eta$ in $W_0^{1,q}(\Omega)$. Let now $(\psi_\epsilon)_{\epsilon> 0}$ be the standard mollifier. For each $j \in \mathbb{N}$ one can choose $\epsilon_j > 0$ such that $ \eta_j := \tilde{\eta}_j * \psi_{\epsilon_j} \in C_0^\infty(\Omega)$ and
\begin{equation}
||\tilde{\eta}_j * \psi_{\epsilon_j} - \tilde{\eta}_j||_{W^{1,q}(\Omega)} + ||\tilde{\eta}_j * \psi_{\epsilon_j} - \tilde{\eta}_j||_{C^0(\overline{\Omega})} < \frac{1}{j}. 
\end{equation}
A straightforward application of the triangle inequality and the previous observations shows that then 
\begin{equation}
   \eta_j := \tilde{\eta}_j * \psi_{\epsilon_j}  \rightarrow \eta \quad \textrm{in } W^{1,q}(\Omega) \; \textrm{and } C(\overline{\Omega}). 
\end{equation}
\end{proof}

\begin{lemma}\label{lem:closure}
Let $\Omega \subset \mathbb{R}^n$ be a smooth domain. Then the closure of $C^2(\overline{\Omega}) \cap W_0^{1,2}(\Omega)$ with respect to the $W^{2,2}$-norm is $W^{2,2}(\Omega) \cap W_0^{1,2}(\Omega)$. 
\end{lemma}
\begin{proof}
Let $v \in W^{2,2}(\Omega) \cap W_0^{1,2}(\Omega)$. Then $v$ lies in the domain of the \emph{Dirichlet Laplacian} $\Delta_0$, a densely defined and sectorial operator. Hence $\Delta_0$ generates an analytic semigroup $(e^{t\Delta_0})_{t \geq 0}$ in $L^2(\Omega)$. Next we set $v_j := e^{\frac{1}{j}\Delta_0} v$ for  all $j \in \mathbb{N}$. We claim that $v_j \in C^2(\overline{\Omega})\cap W_0^{1,2}(\Omega)$ and $v_j \rightarrow v$ in $W^{2,2}(\Omega)$. The first assertion follows easily from the fact that by \cite[Proposition 2.1.1]{Lunardi} $v_j$ lies in the domain of $\Delta_0^k$ for all $k \in \mathbb{N}$ and hence actually in $C^\infty(\overline{\Omega})\cap W_0^{1,2}(\Omega)$. Moreover since $v$ lies in the domain of $\Delta_0$, \cite[Proposition 2.1.4]{Lunardi} yields 
\begin{equation}
    ||\Delta_0 v - \Delta_0 v_j ||_{L^2} = ||\Delta_0 ( e^{\frac{1}{j}\Delta_0} v- v) ||_{L^2} \rightarrow 0 \quad (j \rightarrow \infty).  
\end{equation}
Since $\Delta_0 : W^{2,2}(\Omega) \cap W_0^{1,2}(\Omega) \rightarrow L^2(\Omega)$ is an isomorphism we infer that  also $|| v- v_j||_{W^{2,2}} \rightarrow 0$ as $j \rightarrow \infty$. 
\end{proof}
\begin{lemma}\label{lem:Lipschitzmeanvalue}
Let $f : \mathbb{R}^n \setminus \{0 \} \rightarrow \mathbb{R}$ be locally Lipschitz continuous. Then the map $I : (0,\infty) \rightarrow \mathbb{R}$
\begin{equation}
    s \mapsto \int_{\partial B_s(0)} f(y) \; \mathrm{d}\mathcal{H}^{n-1}(y) 
\end{equation}
is locally Lipschitz continuous on $(0,\infty)$.
\end{lemma}
\begin{proof}
\begin{align}
   & I(r)- I(s)  = \int_{\partial B_r(0)} f(y) \; \mathrm{d}\mathcal{H}^{n-1}(y) -  \int_{\partial B_s(0)} f(y) \; \mathrm{d}\mathcal{H}^{n-1}(y)\\ &  = \int_{\partial B_r(0)} f(y) \; \mathrm{d}\mathcal{H}^{n-1}(y) -   \left( \frac{s}{r}\right)^{n-1} \int_{\partial B_r(0)} f(\tfrac{s}{r}y) \; \mathrm{d}\mathcal{H}^{n-1}(y)
    \\ &  =  \left( 1 - \left(\frac{s}{r}\right)^{n-1} \right) \int_{\partial B_r(0)} f(y) \; \mathrm{d}\mathcal{H}^{n-1}(y) + \left(\frac{s}{r}\right)^{n-1} \int_{\partial B_r(0)}  ( f(y) - f(\tfrac{s}{r} y) ) \; \mathrm{d}\mathcal{H}^{n-1}(y). 
\end{align}
Now for each $y \in \partial B_r(0)$ one has
\begin{equation}
    |f(y) - f(\tfrac{s}{r}y)| \leq ||f||_{W^{1,\infty}(B_s(0) \setminus \overline{B_r(0)} )}  |y - \tfrac{s}{r}y| \leq ||f||_{W^{1,\infty}(B_s(0) \setminus \overline{B_r(0)} )}  |r-s|
\end{equation}
and 
\begin{align}
   \left\vert  \left( 1 - \left(\frac{s}{r}\right)^{n-1} \right) \int_{\partial B_r(0)} f(y) \; \mathrm{d}\mathcal{H}^{n-1}(y)  \right\vert  & \leq ||f||_{L^\infty(\partial B_r(y))} \omega_n |r^{n-1}- s^{n-1}|  \\  &= ||f||_{L^\infty(\partial B_r(y))} \omega_n  \left(\sum_{k = 0}^{n-2} r^k s^{n-2-k}   \right) |r-s|.
\end{align}
With the previous two equations we infer 
\begin{equation}
    |I(r)-I(s)| \leq L(r,s) |r-s|,
\end{equation}
where 
\begin{equation}
    L(r,s) := ||f||_{W^{1,\infty}(B_s(0) \setminus \overline{B_r(0)} )} + ||f||_{L^\infty(\partial B_r(y))} \omega_n  \left(\sum_{k = 0}^{n-2} r^k s^{n-2-k}   \right)
\end{equation}
is clearly locally bounded in $(0,\infty)^2$. The local Lipschitz property follows. 
\end{proof}

 \begin{lemma}\label{lem:B5}
 One has $BV(\Omega) \cap W^{1,1}_{loc}(\Omega) = W^{1,1}(\Omega)$.
 \end{lemma}
 \begin{proof}
 If $w \in BV(\Omega) \cap W^{1,1}_{loc}(\Omega)$ then for all open sets $U \subset \subset \Omega$ and $i= 1,...,n$ we have 
 \begin{align}
     \int_U |\partial_i w| \; \mathrm{d}x & = \sup_{\phi \in C_0^1(U),||\phi||_\infty \leq 1} \int_U \partial_i w \; \phi \; \mathrm{d}x = \sup_{\phi \in C_0^1(U),||\phi||_\infty \leq 1}  - \int_U  w \; \partial_i \phi \; \mathrm{d}x
     \\ & = \sup_{\phi \in C_0^1(U),||\phi||_\infty \leq 1}   \int_U  w \; \mathrm{div}(-\phi e_i) \; \mathrm{d}x \leq |Dw|(U) \leq |Dw|(\Omega). 
 \end{align}
 As $\Omega$ can be monotonically exhausted by countably many sets $(U_j)_{j\in \mathbb{N}}, U_j \subset U_{j+1} \subset \subset \Omega$ we infer by the monotonce convergence theorem that $|\partial_i w | \in L^1(\Omega)$. The claim follows.
 \end{proof}
 
 \begin{lemma}\label{lem:B6}
  Let $\Omega \subset \mathbb{R}^n$ be a bounded domain, $(v_k)_{k \in \mathbb{N}} \subset W^{1,\infty}(\Omega)$ be a sequence such that $(||v_k||_{W^{1,\infty}(\Omega)})_{k \in \mathbb{N}}$ is uniformly bounded in $k$. Then there exists a subsequence $(v_{l_k})_{k \in \mathbb{N}} \in \mathbb{N}$ and some $v \in W^{1,\infty}(\Omega)$ such that $(v_{l_k})_{k \in \mathbb{N}}$ converges to $v$ uniformly on $\overline{\Omega}$, weakly in $W^{1,p}(\Omega)$ for all $p \in (1,\infty)$. Further, one has
  \begin{equation}\label{eq:liplower}
      ||v||_{W^{1,\infty}(\Omega)} \leq \liminf_{k \rightarrow \infty} ||v_{l_k}||_{W^{1,\infty}(\Omega)}.
  \end{equation}
 \end{lemma}
\begin{proof}
One readily checks that $(v_k)_{k \in \mathbb{N}}$ satisfies all the prerequisites of the Arzela-Ascoli theorem, which explains the existence of a uniformly convergent subsequence and a limit $v \in C^0(\overline{\Omega})$. Choosing a further subsequence and using a standard diagonal sequence argument one can find a subsequence that converges weakly in $W^{1,r}(\Omega)$ for all $r \in \mathbb{N}$. One readily checks that the limit of this subsequence must also be $v$. Let now $(v_{l_k})_{k \in \mathbb{N}}$ be this subsequence.  If now $p \in (1,\infty)$ is arbitrary we can choose some number $r \in \mathbb{N}$ such that $r> p$.  Observing that $W^{1,r}(\Omega) \subset W^{1,p}(\Omega)$ (and hence $W^{1,p}(\Omega)^* \subset W^{1,r}(\Omega)^*$) we obtain that $v_{l_k} \rightharpoonup v$ in $W^{1,p}(\Omega)$. For the norm estimate \eqref{eq:liplower} observe that by uniform convergence we have 
\begin{equation}\label{eq:arzela-ascoli}
    ||v||_{L^\infty(\Omega)} = \lim_{k \rightarrow \infty} ||v_{l_k}||_{L^\infty(\Omega)}.
\end{equation}
Now for the gradients we observe
\begin{align}
    ||\nabla v||_{L^\infty(\Omega)} & = \lim_{p \rightarrow \infty} ||\nabla v||_{L^p(\Omega)} \leq \limsup_{p \rightarrow \infty}  \liminf_{k \rightarrow \infty} ||\nabla v_{l_k}||_{L^p(\Omega)} \\ &  \leq \limsup_{p \rightarrow \infty} \liminf_{k \rightarrow \infty} ||\nabla v_{l_k}||_{L^\infty(\Omega)} |\Omega|^\frac{1}{p}  = \liminf_{k \rightarrow \infty} ||\nabla v_{l_k}||_{L^\infty(\Omega)}.
\end{align}
This and \eqref{eq:arzela-ascoli} yields 
\begin{equation}
    ||v||_{W^{1,\infty}(\Omega)} \leq \liminf_{k \rightarrow \infty}  ||v_{l_k}||_{W^{1,\infty}(\Omega)}.
\end{equation}
\end{proof}

\begin{lemma}\label{lem:hoeldiffeq}
 Let $y \in C^1([a,b],\mathbb{R})$ be such that $y'(t)= f(t,y(t))$ for all $t \in [a,b]$ for some $f \in C^{0,\alpha}(\mathbb{R}^2)$. Then 
 \begin{equation}
     ||y||_{C^{1,\alpha}([a,b])} \leq ||y||_\infty + [2+ (b-a)^{1-\alpha}] (1+||f||_{C^{0,\alpha}(\mathbb{R}^2)})^2
 \end{equation}
\end{lemma}
\begin{proof}
First derive  the estimate 
\begin{equation}\label{eq:lipEst}
    |y(t) - y(s) | = \int_{t}^s |f(u,y(u))| \; \mathrm{d}u \leq ||f||_\infty |t-s| . 
\end{equation}
In particular 
\begin{equation}
   |y(t) - y(s)| \leq ||f||_\infty (b-a)^{1-\alpha} |t-s|^\alpha.
\end{equation}
Moreover we infer using \eqref{eq:lipEst}
\begin{align}
    |y'(t) - y'(s)|  & = |f(t,y(t))- f(s,y(s))| \leq ||f||_{C^{0,\alpha}} ( |t-s|^2 + |y(t)- y(s)|^2 )^{\frac{\alpha}{2}}  \\
     & = ||f||_{C^{0,\alpha}} ( 1 + ||f||_\infty^2 )^\frac{\alpha}{2} |t-s|^\alpha.
\end{align}
Now note that $( 1 + ||f||_\infty^2 )^\frac{\alpha}{2} \leq ( 1 + ||f||_\infty )^\alpha \leq 1 + ||f||_\infty.$ The previous computations yield 
\begin{align}
    ||y||_{C^{1,\alpha}}  & \leq ||y||_\infty + ||y'||_\infty + [y]_{C^{0,\alpha}} + [y']_{C^{0,\alpha}} 
    \\ & \leq ||y||_\infty + ||f||_\infty + (b-a)^{1-\alpha} ||f||_\infty + ||f||_{C^{0,\alpha}} ( 1 + ||f||_\infty ).
\end{align}
The claim follows estimating all $||f||_\infty$ by $||f||_{C^{0,\alpha}}$ and some elementary estimates. 
\end{proof}

\begin{lemma}\label{lem:hoelcoord}
 Let $B \subset \mathbb{R}^m$ be an open ball and $h \in C^{0,\alpha}(\overline{B};\mathbb{R})$. For $z \in B$ and $v \in \mathbb{R}^m, |v| = 1$ we define $t_+(z,v)$ to be the supremum of all $t \in \mathbb{R}$ such that $z + t v \in B$ and similarly $t_-(z,v)$ to be the infimum. Define
 \begin{equation}
     y_{z,v}(t) := h(z + tv) \quad t \in [t_-(z,v), t_+(z,v)].
 \end{equation}
 Then 
 \begin{equation}
     ||h||_{C^{0,\alpha}(\overline{B})} \leq 2 \sup_{z \in B, |v| = 1} ||y_{z,v} ||_{C^{0,\alpha}([t_
     -(z,l),t_+(z,l)])}.
 \end{equation}
 %Then $h \in C^{0,\alpha}(\overline{B})$ if and only if $y_{l,z} \in C^{0,\alpha}([t_-(z,l), t_+(z,l)])$ for all $z \in B$ and $l = 1,...,n$ and there holds 
\end{lemma}
\begin{proof}
It is straightforward to show that 
\begin{equation}
    ||h||_\infty = \sup_{z \in B, |v|= 1} ||y_{z,v}||_{L^\infty([t_-(z,v),t_+(z,v)])}.
\end{equation}
We next show the corresponding estimate for $[h]_{C^{0,\alpha}(\overline{B})}$. Let $x_1,x_2 \in B$, $x_1 \neq x_2$. 
Define $z := x_1$ and  $v := \frac{x_2- x_1}{|x_2-x_1|}$ now observe that by convexity of $B$ 
\begin{equation}
    z + tv = x_1 + \tfrac{t}{|x_2 - x_1|} (x_2 - x_1) \in B \quad \forall t \in [0, |x_2-x_1|]. 
\end{equation}
this implies that $t_-(z,v) \leq 0  \leq |x_2-x_1| \leq t_+(z,v)$. Now 
\begin{align}
    |h(x_2) - h(x_1)| & = |y_{z,v}(|x_1-x_2|) - y_{z,v}(0)|  \leq [y_{z,v}]_{C^{0,\alpha}([t_-(z,v), t_+(z,v)]} | \; |x_1 -x_2| - 0|^\alpha  \\ & \leq \sup_{z \in B, |v| = 1 }[y_{z,v}]_{C^{0,\alpha}([t_-(z,v), t_+(z,v)]}  |x_1-x_2|^\alpha. 
\end{align}
We infer that $[h]_{C^{0,\alpha}(B)} \leq \sup_{z \in B, |v| = 1 }[y_{z,v}]_{C^{0,\alpha}([t_-(z,v), t_+(z,v)]}$ and the claim follows. 
\end{proof}

\begin{lemma}\label{lem:Hoelextension}
 Let  for $A \subset \mathbb{R}$ closed and $\alpha > 0$ $f : A \rightarrow \mathbb{R}$ be $\alpha-$Hölder continuous in the sense that 
 \begin{equation}
     |f(x) - f(y) | \leq H|x-y|^\alpha \quad \forall x,y \in A 
 \end{equation}
for some $H> 0$. Then there exists some $\bar{f} : \mathbb{R}^n \rightarrow \mathbb{R}$ such that  $\bar{f}\vert_A = f$ and 
\begin{equation}
    |\bar{f}(x) - \bar{f}(y) | \leq H|x-y|^\alpha \quad \forall x,y \in \mathbb{R}^n. 
\end{equation}
In particular $[\bar{f}]_{C^{0,\alpha}} = [f]_{C^{0,\alpha}}$. 
\end{lemma}
\begin{proof}
This is readily checked choosing 
\begin{equation}
    \bar{f}(z) := \inf_{x \in A}  \left\lbrace f(x) + H|z-x|^\alpha \right\rbrace \quad (z \in \mathbb{R}^n). 
\end{equation}
We remark that the subadditivity of $z \mapsto z^\alpha$ plays an important role in this proof.
\end{proof}

\begin{lemma}\label{lem:Hoelapprox}
 Let $K \subset \Omega$ be compact and $Q \in C^{0,\alpha}(K)$. Then there exists some $(Q_j)_{j \in \mathbb{N}} \subset C^\infty(\Omega)$ such that $Q_j \rightarrow Q$ uniformly on $K$ and $ \lim_{j \rightarrow \infty} ||Q_j||_{C^{0,\alpha}(K)} =  ||Q||_{C^{0,\alpha}(K)}$. 
\end{lemma}
\begin{proof}
By Lemma \ref{lem:Hoelextension} we may actually assume that $Q \in C^{0,\alpha}(\mathbb{R}^n)$ (without making the Hölder seminorm larger). Next let $Q_\epsilon := Q* \phi_\epsilon$ where $(\phi_\epsilon)_{\epsilon > 0}$ is the standard mollifier.  A standard result yields that $Q_\epsilon \rightarrow Q$ uniformly on $K$. The lower semicontinuity of $[\cdot]_{C^{0,\alpha}}$ with respect to uniform convergence yields \begin{equation}\label{eq:lscQ}
    [Q]_{C^{0,\alpha}(K)} \leq \liminf_{\epsilon \rightarrow 0 }  [Q_\epsilon]_{C^{0,\alpha}(K) }. 
\end{equation}
Now observe also that for all $x_1,x_2 \in K$ there holds 
\begin{align}
    |Q_\epsilon(x_1) - Q_\epsilon(x_2) | & = \left\vert \int Q(y) \phi_\epsilon(x_1-y ) \; \mathrm{d}y - \int Q(y) \phi_\epsilon(x_2- y) \; \mathrm{d}y \right\vert 
    \\ & \leq \left\vert \int (Q(x_1 - z) - Q(x_2-z)) \phi_\epsilon(z ) \; \mathrm{d}z  \right\vert  \\ & \leq \int |Q(x_1 - z) - Q(x_2-z)| \phi_\epsilon(z ) \; \mathrm{d}z
    \\ &  \leq \int [Q]_{C^{0,\alpha}(\mathbb{R}^n)} |x_1-z-(x_2-z)|^\alpha  \phi_\epsilon(z ) \; \mathrm{d}z
      \leq [Q]_{C^{0,\alpha}(K)} |x_1 -x_2|^\alpha .
\end{align}
This yields $[Q_\epsilon]_{C^{0,\alpha}(K)} \leq [Q]_{C^{0,\alpha}(K)}$ and together with \eqref{eq:lscQ} we conclude that 
$[Q]_{C^{0,\alpha}(K)} = \lim_{\epsilon \rightarrow 0 } [Q_\epsilon]_{C^{0,\alpha}(K)}$. All in all we infer that 
\begin{equation}
    ||Q||_{C^{0,\alpha}(K)} = ||Q||_\infty + [Q]_{C^{0,\alpha}(K)} = \lim_{\epsilon \rightarrow 0} ||Q_\epsilon||_\infty + [Q_\epsilon]_{C^{0,\alpha}(K)} = \lim_{\epsilon \rightarrow 0} ||Q_\epsilon||_{C^{0,\alpha}(K)}.  
\end{equation}
\end{proof}

\section{$C^{1,\alpha}$-hypersurfaces}\label{app:C1alpha}

In this section we collect some facts about $C^{1,\alpha}$-boundaries that we use. In the following we will always look at  $\Gamma = \partial \Omega'$ for a compact $C^{1,\alpha}$-domain $\Omega' \subset \subset \mathbb{R}^n$. 

An important tool that we use is the \emph{regularized signed distance function}, $\rho : \mathbb{R}^n \rightarrow \mathbb{R}$ introduced in \cite{Lieberman}. This is a function $\rho \in C^\infty(\mathbb{R}^n \setminus \Gamma) \cap C^{0,1}(\mathbb{R}^n)$ such that $\rho < 0$ on $\Omega'$, $\rho > 0$ on $\mathbb{R}^n \setminus \Omega'$ and  
\begin{equation}
    \frac{1}{2} \leq \frac{|\rho(x)|}{\mathrm{dist}(x,\Gamma)} \leq 2, \quad \textrm{and} \quad
    \frac{1}{2} \leq  |\nabla \rho (x)| \leq 2  \quad \forall x \in \mathbb{R}^n \setminus \Gamma.
\end{equation}
%First we will need a compatibility between the \emph{Euclidean distance} and the shortest-path distance on $\Gamma$. %which already holds true for $C^1$- submanifolds. 
We see from \cite[Theorem 2.1]{Lieberman} (used with $\zeta(z) = z^\alpha$), that for $\Gamma = \partial \Omega' \in C^{1,\alpha}$ there exists a regularized signed distance function $\rho$, which fulfills in addition 
\begin{equation}\label{eq:gradsigndisti}
    |\nabla \rho(x) - \nabla \rho(y) | \leq 10 [\Gamma]_{1,\alpha} |x-y|^\alpha  \quad \forall x,y \in \mathbb{R}^n.
\end{equation}
(We remark that for obtaining the previous equation we need to demand that the graph representations $f_i$ of $\Gamma$ (cf. \eqref{eq:Gaama}) must lie in $C^{1,\alpha}(4 \overline{U}_i)$ instead of just $C^{1,\alpha}( \overline{U}_i)$. This is why the counterintuitive constant $4$ appears in our definition of a representation). 
In particular $\rho$ extends to a $C^{1,\alpha}$-function on $\mathbb{R}^n$. As $\rho \equiv 0$ on $\Gamma$ the derivative $\nabla \rho$ must point in normal direction, which means that  the outward unit normal $\nu: \Gamma \rightarrow \mathbb{R}^n$ satisfies
\begin{equation}
     \nu(x) = \frac{\nabla \rho(x)}{|\nabla \rho(x)|} \quad \forall x \in \Gamma. 
\end{equation}
This, \eqref{eq:gradsigndisti} and the fact that $\frac{1}{2} \leq |\nabla \rho| \leq 2$ on $\mathbb{R}^n$ implies
 %if $\nu : \Gamma \rightarrow \mathbb{R}$ denotes the outward pointing unit normal of $\Gamma$ then one has the estimates
\begin{equation}\label{eq:hoeldernormal}
    |\nu(x) - \nu(y) | \leq C([\Gamma]_{1,\alpha}) |x-y|^\alpha \quad \forall x,y \in \Gamma.
\end{equation}
As a consequence one finds that if $\Pi_x$ denotes the orthogonal projection on $T_x \Gamma$ then (with $|\cdot|$ denoting the operator norm) there holds 
\begin{equation}\label{eq:C2}
    |\Pi_x- \Pi_y| \leq C([\Gamma]_{1,\alpha}) |x-y|^\alpha \quad \forall x,y \in \Gamma.
\end{equation}
In the following we prove some further geometric estimates for $C^{1,\alpha}$-hypersurfaces and check that the constants appearing in them only depend on $[\Gamma]_{1,\alpha}$. In the end we will show that one can approximate each $C^{1,\alpha}$-hypersurface $\Gamma$ with $C^\infty$-hypersurfaces $(\Gamma_j)_{j \in \mathbb{N}}$ such that $[\Gamma_j]_{1,\alpha}$ is uniformly bounded in terms of $[\Gamma]_{1,\alpha}$.  
%Before we proceed need one more elementary estimate of this kind.

\begin{lemma}\label{lem:normale}
 Let $\Gamma$ be a $C^{1,\alpha}$-hypersurface. Then there exists $C([\Gamma]_{1,\alpha}) > 0$ such that 
 \begin{equation}
    (\nu(x), x-y) \leq C([\Gamma]_{1,\alpha}) |x-y|^{1+\alpha} \quad \forall x,y \in \Gamma. 
\end{equation}
\end{lemma}
\begin{proof}
Let $x,y \in \Gamma$ be arbitrary.  
The mean value theorem on $\mathbb{R}^n$ yields that there exists some $\xi \in \mathbb{R}^n$ on the line segement connecting $x$ and $y$ which satisfies 
\begin{equation}
    (\nabla \rho (\xi) , x-y) = \rho(x) - \rho(y) = 0.
\end{equation}
We infer
\begin{equation}
    (\nabla \rho(x) ,x-y) = (\nabla \rho(x) - \nabla \rho(\xi) , x-y) \leq |\nabla \rho(x) - \nabla \rho(\xi) | \;  |x-y|.
\end{equation}
Using \eqref{eq:gradsigndisti} we can bound the first factor by $10 [\Gamma]_{1,\alpha} |x-\xi|^\alpha$. Since $\xi$ lies on the line segment connecting $x$ and $y$ we also have $|x-\xi|^\alpha \leq |x-y|^\alpha$. Hence we obtain 
\begin{equation}
    (\nabla \rho(x) , x-y) \leq 10[\Gamma]_{1,\alpha} |x-y|^{1+\alpha}.
\end{equation}
The claim follows then from $\nu= \frac{\nabla \rho}{|\nabla \rho|}$ and $|\nabla \rho| \geq \frac{1}{2}$. 
\end{proof}

Next we show that the intrinsic distance and the Euclidean norm are comparable on a $C^{1,\alpha}$ hypersurface and the constant in the comparison estimate only depends on $[\Gamma]_{1,\alpha}$. 

\begin{lemma}\label{lem:C1}
 Let $\Gamma \subset \subset \mathbb{R}^n$ be an $n-1$ dimensional $C^{1,\alpha}$-hypersurface. Then there exists $C= C([\Gamma]_{1,\alpha}) >1$ such that for all $y, z \in \Gamma$ one has $|y-z| \leq \mathrm{dist}_\Gamma (y,z) \leq C([\Gamma]_{1,\alpha})|y-z|$, where
    \begin{equation}
        \mathrm{dist}_\Gamma(y,z) := \inf \left\lbrace  \int_0^1 |\gamma'(s)| \; \mathrm{d}s : \gamma \in C^{1,\alpha}([0,1];\Gamma): \gamma(0) = y, \gamma(1) = z \right\rbrace. 
    \end{equation}
\end{lemma}
\begin{proof}
We use \cite[Theorem 1.1]{Blatt}. For a compact  $C^{1}$-submanifold $S$ of $\mathbb{R}^n$ one can define 
\begin{equation}
    \gamma_0(S) := \max\{\gamma_1(S), \gamma_2(S) \}
\end{equation}
where 
\begin{equation}
    \gamma_1(S) := \sup_{x \in S, R>0} \inf_{v \in \mathbb{R}^n,|v|= 1}  I_{x,R}(v),
\end{equation}
\begin{equation}
   I_{x,R} (v) :=  \frac{1}{\mathcal{H}^{n-1}(S \cap B_R(x))} \int_{S \cap B_R(x)}   |\nu(y) - v| \; \mathrm{d}\mathcal{H}^{n-1}(y)  \quad (v \in \mathbb{R}^n, |v| = 1) 
\end{equation}
and 
\begin{equation}
    \gamma_2(S) := \inf_{x \in S, R >0 } \sup_{y \in S \cap B_R(x), v \in \mathrm{arg inf} I_{x,R} }  \frac{(v,x-y)}{R}, 
\end{equation}
where $\mathrm{arg inf}$ is meant among all $v \in \mathbb{R}^n$ s.t. $|v| = 1$. 
Now \cite[Theorem 1.1]{Blatt} states that there exists $\epsilon = \epsilon(n) > 0$ and $C = C(n) > 0$ such that $\gamma_0(S) \leq  \epsilon$ implies that 
\begin{equation}
    \left\vert \frac{\mathrm{dist}_\Gamma(x,y)}{|x-y|} - 1 \right\vert \leq C(n) \gamma_0(S) \log \tfrac{1}{\gamma_0(S)} \quad \forall x ,y \in \Gamma.
\end{equation}
%Hence we only need to show that $\gamma(\Gamma)$ can be bounded in terms of $[\Gamma]_{1,\alpha}$.
Next we look at the compact submanifold $S := \Gamma \cap \overline{B_s(z)} \subset \Gamma$ for some arbitrary $z \in \mathbb{R}^n$, $s> 0$. %that intersects $\Gamma$ transversally.
Fix $x \in S$. One has (with $C= C([\Gamma]_{1,\alpha})$ defined as in \eqref{eq:hoeldernormal})
\begin{align}
    \inf_{|v|= 1} I_{x,R}(v) \leq  I_{x,R}(\nu(x)) & = \fint_{S \cap B_R(x)} |\nu(y)- \nu(x)| \; \mathrm{d}\mathcal{H}^{n-1}(y)  \\ & \leq C \fint_{S \cap B_R(x)}|y-x|^\alpha \leq C(2s)^\alpha. \label{gamma1klein}
\end{align}
and thus $\gamma_1(\Gamma \cap \overline{B_s(z)}) \leq D([\Gamma]_{1,\alpha}) s^\alpha$. Moreover for each $v \in \mathrm{arg inf} I_{x,R}$ we have 
\begin{align}
    |v- \nu(x)| & = \fint_{S \cap B_r(x)} |v-\nu(x)| \; \mathrm{d}\mathcal{H}^{n-1}(y)  \\ & \leq  \fint_{S \cap B_r(x)} |\nu(y)-v| \; \mathrm{d}\mathcal{H}^{n-1}(y) + \fint_{S \cap B_r(x)} |\nu(y)-\nu(x)| \; \mathrm{d}\mathcal{H}^{n-1}(y)
    \\ & \leq 2 \fint_{S \cap B_r(x)} |\nu(y) - \nu(x)| \leq 2 D([\Gamma]_{1,\alpha}) s^{\alpha}. 
\end{align}
Therefore we have that for each $v \in \mathrm{arg inf} I_{x,R}$ and $y \in S \cap B_R(x)$ there holds 
\begin{align}
    \frac{(v,x-y)}{R}  & = \frac{(v-\nu(x),x-y)}{R} + \frac{(\nu(x),x-y)}{R} \\ & \leq 2D([\Gamma]_{1,\alpha})s^\alpha \frac{|x-y|}{R} + C([\Gamma]_{1,\alpha})  \frac{|x-y|^{1+\alpha}}{R}  \\ & \leq 2D([\Gamma]_{1,\alpha})s^\alpha \frac{|x-y|}{R} + C([\Gamma]_{1,\alpha}) (2s)^\alpha \frac{|x-y|}{R}  \\ & \leq 2D([\Gamma]_{1,\alpha})s^\alpha+ C([\Gamma]_{1,\alpha}) (2s)^\alpha.
\end{align}
This implies that 
\begin{equation}
    \gamma_2(\Gamma \cap \overline{B_s(z)}) = E([\Gamma]_{1,\alpha}) s^\alpha. 
\end{equation}
From this and the discussion after \eqref{gamma1klein} we infer that $s < s_0([\Gamma]_{1,\alpha})$ implies that $\gamma_0(\Gamma \cap \overline{B_s(z)}) \leq \epsilon(n)$ and hence for all $s < s_0([\Gamma]_{1,\alpha})$ one has 
    \begin{equation}\label{eq:distloc}
    \left\vert \frac{\mathrm{dist}_{\Gamma \cap \overline{B_s(z)}}(x,y)}{|x-y|} - 1 \right\vert \leq C(n) \sup_{t \in (0, \epsilon(n))} t \log \tfrac{1}{t} < \infty. 
\end{equation}
Now suppose that $x,y \in \Gamma$. If $|x-y| \geq \frac{1}{2}s_0([\Gamma]_{1,\alpha})$ then one has 
\begin{equation}
    \frac{\mathrm{dist}_\Gamma(x,y)}{|x-y|} \leq \frac{2\diam_\Gamma(\Gamma)}{s_0([\Gamma]_{1,\alpha})}, 
\end{equation}
where $\diam_\Gamma$ denotes the intrinsic diameter of $\Gamma$, which we bound later in terms of $[\Gamma]_{1,\alpha}$. If $|x-y| <\frac{1}{2}s_0([\Gamma]_{1,\alpha}) $ then one has by \eqref{eq:distloc}
\begin{equation}
    \frac{\mathrm{dist}_\Gamma(x,y)}{|x-y|} =  \frac{\mathrm{dist}_{\Gamma \cap B_s(x)}(x,y)}{|x-y|} \leq C(n) \sup_{t \in (0, \epsilon(n))} t \log \tfrac{1}{t}.
\end{equation}
Choosing 
\begin{equation}
    C_1:= \frac{2\diam_\Gamma(\Gamma)}{s_0([\Gamma]_{1,\alpha})}  + C(n) \sup_{t \in (0, \epsilon(n))} t \log \tfrac{1}{t}
\end{equation}
we infer from the previous computations that $\mathrm{dist}_\Gamma(x,y) \leq C_1 |x-y|$. To show that $C_1$ depends only on $[\Gamma]_{1,\alpha}$ it remains to prove that 
 $   \mathrm{diam}_\Gamma (\Gamma)$ is bounded in terms of $[\Gamma]_{1,\alpha}$. To this end pick a representation $(O_i,U_i,V_i, f_i$, $i=1,...,M)$ and let $\Gamma_i$ be as in \eqref{eq:Gaama}, i.e. $\Gamma_i \cap \Gamma_{i+1} \neq \emptyset$. Choose arbitrary $x_i \in \Gamma_i \cap \Gamma_{i+1}$ for all $i = 1,...,M$. Now let $y, z \in \Gamma$ be arbitrary. Without loss of generality we may assume $y \in \Gamma_i$ and $z \in \Gamma_j$ for some $i \leq j$. Now we estimate 
 \begin{equation}\label{eq:Harnackchain}
     \mathrm{dist}_\Gamma(y,z) \leq \mathrm{dist}_\Gamma(y,x_i) + \sum_{l = i}^{j-1} \mathrm{dist}_\Gamma(x_{l+1},x_l)  + \mathrm{dist}_\Gamma(z,x_j).
 \end{equation}
Observe that each summand is the distance of two points lying in the same subset $\Gamma_i$. For any two $y_1,y_2 \in \Gamma_i$, corresponding to the values $(z_1,f(z_1)),(z_2,f(z_2))$ for some $z_1,z_2 \in U_i$ one sees that 
\begin{equation}
    \gamma(t) := O_i ((z_1+t (z_2-z_1), f(z_1 + t(z_2-z_1)))  
\end{equation}
lies in $C^{1,\alpha}([0,1];\Gamma)$ and connects $y_1,y_2$. (Notice that we have used here that $U_i$ is a ball). Therefore 
\begin{align}
    \mathrm{dist}_\Gamma(y_1,y_2) & \leq \int_0^1 |\gamma'(t)| \; \mathrm{d}t   \leq |z_1-z_2| + |f(z_1) - f(z_2)|
    \leq (1+ ||Df_i||_\infty) |z_1- z_2| 
     \\ &\leq (1+ ||Df_i||_\infty) |y_1-y_2|  \leq 
    (1+ ||Df_i||_\infty)\mathrm{diam}(\Gamma) \leq C([\Gamma]_{1,\alpha}) . 
\end{align}
We infer from \eqref{eq:Harnackchain} (and the fact that $M$ is bounded by $[\Gamma]_{1,\alpha}$, cf. \eqref{eq:Gaama} f.)  that 
\begin{equation}
    \mathrm{dist}_\Gamma(y,z) \leq M C([\Gamma]_{1,\alpha}) \leq \tilde{C}([\Gamma]_{1,\alpha}).
\end{equation}
Since $y,z \in \Gamma$ were arbitrary the diameter bound follows.
%we infer the desired diameter bound from \eqref{eq:Harnackchain}. 
\end{proof}

\begin{lemma}\label{lem:ballgraphsystem}
Let $\Gamma \subset \mathbb{R}^n$ be an $(n-1)$-dimensional $C^{1,\alpha}$-hypersurface. Then there exists some $r_1 = r_1([\Gamma]_{1,\alpha}) > 0$ such that for all $r < r_1$ and for all  $x_0 \in \Gamma$ the hypersurface $\Gamma \cap B_r(x_0)$ can be covered by a single graph of a $C^{1,\alpha}$-function $u: W \rightarrow \mathbb{R}$ with vanishing gradient at the preimage of $x_0$ and Lipschitz-norm bounded by 1. The radius $r_1$ and $||u||_{C^{1,\alpha}}$ depend only on $[\Gamma]_{1,\alpha}$.
%, i.e. $r_1 = r_1([\Gamma]_{1,\alpha})$ and the $C^{1,\alpha}$-norm of the graph can also be bounded only in terms of $[\Gamma]_{1,\alpha}$.
\end{lemma}
\begin{proof}
Let $C_1= C_1([\Gamma]_{1,\alpha})$ be the constant from Lemma \ref{lem:C1} and $C_2= C_2([\Gamma]_{1,\alpha})$ be the constant from \eqref{eq:C2}. 
Next we let $r_1 = r_1([\Gamma]_{1,\alpha})> 0$ be chosen such that 
\begin{equation}
    r_1^\alpha   \min\{ 2^\alpha  (1+ C_1)^\alpha C_1 C_2 ,   C_2 \}  = \frac{1}{2}.
\end{equation}
Now fix $x \in \Gamma$ arbitrary. Let $\Pi_x^\perp = I- \Pi_x$ be the orthogonal projection on $T_x\Gamma^\perp$. Let $r \leq r_1$ and $y, z \in \Gamma \cap B_r(x)$ be arbitrary. Fix $\epsilon >0$ and choose some $\gamma \in C^{1,\alpha}([0,1]; \Gamma)$ such that $\gamma(0)=y$, $\gamma(1)= z$, $L(\gamma) := \int_0^1 |\gamma'(s)| \; \mathrm{d}s \leq (1+\epsilon) \mathrm{dist}_\Gamma(y,z)$. 
%In the coming computation we will just write $C_2$ instead of $C_2([\Gamma]_{1,\alpha})$.
Observe that for all $t \in [0,1]$ one has $\gamma'(t) \in T_{\gamma(t)} \Gamma$ and $\mathrm{dist}_{\Gamma}(\gamma(t),y) \leq L(\gamma\vert_{[0,t]}) \leq L(\gamma) \leq (1+\epsilon) \mathrm{dist}(y,z)$. We compute
\begin{align}
   & |\Pi_x^\perp y - \Pi_x^\perp z|   = \left\vert \Pi_x^\perp \int_0^1 \gamma'(s) \; \mathrm{d}s \right\vert  = \left\vert \int_0^1 \Pi_x^\perp \gamma'(s) \; \mathrm{d}s \right\vert = \left\vert \int_0^1 \Pi_x^\perp \Pi_{\gamma(s)} \gamma'(s) \; \mathrm{d}s \right\vert
    \\ & =   \left\vert \int_0^1 \Pi_x^\perp (\Pi_{\gamma(s)}- \Pi_x) \gamma'(s) \; \mathrm{d}s \right\vert
     \leq \int_0^1 |\Pi_x^\perp| |\Pi_{\gamma(s)}- \Pi_x| |\gamma'(s)| \; \mathrm{d}s 
    \underset{\eqref{eq:C2}}{\leq} C_2 \int_0^1 |\gamma(s) - x|^\alpha |\gamma'(s)| \; \mathrm{d}s 
    \\ & \leq  C_2 \int_0^1 (|\gamma(s) - y| + |y-x|)^\alpha |\gamma'(s)| \; \mathrm{d}s
    \leq C_2  \int_0^1 ( \mathrm{dist}_\Gamma(\gamma(s),y) + r)^\alpha  |\gamma'(s)| \; \mathrm{d}s 
     \\ & \leq C_2 ( (1+\epsilon)\mathrm{dist}_\Gamma(z,y) + r)^\alpha  \int_0^1 |\gamma'(s)| \; \mathrm{d}s
    \leq C_2 ( (1+\epsilon)\mathrm{dist}_\Gamma(z,y) + r)^\alpha  (1+\epsilon) \mathrm{dist}_\Gamma(y,z) 
    \\ &  \leq C_2 ( (1+\epsilon)(\mathrm{dist}_\Gamma(z,x) + \mathrm{dist}_\Gamma(x,y)) + r)^\alpha  (1+\epsilon) \mathrm{dist}_\Gamma(y,z) 
    \\ & \leq C_2 C_1 ( (1+\epsilon)C_1 (|x-z| +  |y-z|) + r)^\alpha  (1+\epsilon) |y-z| \\
     & \leq (1+\epsilon)C_1 C_2 (2(1+\epsilon) C_1 + 1)^\alpha r^\alpha |y-z|.
\end{align}
Letting $\epsilon \rightarrow 0$ and using $r \leq r_1$ we obtain 
\begin{equation}
    |\Pi_x^\perp y - \Pi_x^\perp z| \leq \frac{1}{2}|y-z| \leq \frac{1}{2}|\Pi_x y-\Pi_x z| + \frac{1}{2}|\Pi_x^\perp y -\Pi_x^\perp z |.
\end{equation}
Rearranging we obtain 
\begin{equation}\label{eq:inj}
    |\Pi_x^\perp y - \Pi_x^\perp z| \leq |\Pi_x y -\Pi_x z | \quad \forall y,z \in \Gamma \cap B_r(x),
\end{equation}
and one also concludes immediately 
\begin{equation}\label{eq:inv}
    |y-z| \leq \sqrt{2} |\Pi_x y -\Pi_x z | \quad \forall y,z \in \Gamma \cap B_r(x).
\end{equation}
We conclude from \eqref{eq:inj} that $\Pi_x$ is injective on $\Gamma \cap B_r(x)$, call $W$ its image, and infer from \eqref{eq:inv} that the inverse of $(\Pi_x \vert_{\Gamma \cap B_r(x)})^{-1}: W \rightarrow \Gamma \cap B_r(x)$ is Lipschitz continuous. We infer that $\Pi_x : \Gamma \cap B_r(x) \rightarrow W$ is a homeomorphism.
Hence $\Gamma\cap B_r(x)$ is a Lipschitz graph over $T_x \Gamma$. Since $\Gamma \cap B_r(x)$ is a $C^{1,\alpha}$-submanifold, one obtains that this graph representation must be a $C^{1,\alpha}$-graph.  
%It remains to show that this graph is $C^{1,\alpha}$-regular. It is tempting to say that this is clear, however we need to check that all operations we perform do not affect the choice of $r_1$ as a function of only $[\Gamma]_{1,\alpha}$, which is why we present a proof that does not require any shrinking of $r_1$. Up to rotation we may assume $T_x \Gamma = \mathrm{span}(e_1,...,e_{n-1})$ and 
This means 
\begin{equation}\label{eq:graphigraphi}
    \Gamma \cap B_r(x) = O \{(x',f(x')) : x' \in W \}
\end{equation}
for some $C^{1,\alpha}$ function $f: W \rightarrow \mathbb{R}$ and some orthogonal matrix $O$ such that $O \cdot  \mathrm{span}(e_1,...,e_{n-1})= T_x\Gamma$. 
%As $\Gamma$ is a $C^{1,\alpha}$-submanifold, one also obtains readily that $f \in C^{1,\alpha}(W)$.
By the choice of $O$ we also infer $\nabla f(x') = 0$ if $x'$ is the preimage of $x$. The Lipschitz continuity with constant $1$ of $f$ follows immediately from \eqref{eq:inj}. It only remains to show that the $C^{1,\alpha}$-norm of $f$ can be bounded in terms of $[\Gamma]_{1,\alpha}$. Notice first that the Lipschitz estimate implies $||\nabla f||_{L^\infty(W)} \leq 1$. Define $h(z) := \frac{z}{\sqrt{1+|z|^2}}$ and observe that for all $z_1,z_2 \in B_1(0)$ one has
\begin{equation}
    |z_1 - z_2|  = |h^{-1}(h(z_1)) - h^{-1}(h(z_2)) | \leq ||\nabla h^{-1}||_\infty |h(z_1)- h(z_2)| \leq \sqrt{2} |h(z_1)- h(z_2)|.
\end{equation}
Therefore we obtain for $x',y' \in W$
\begin{align}
    |\nabla f(x') - \nabla f(y') |  & \leq \sqrt{2} | h(\nabla f(x')) - h(\nabla f(y'))| \leq \sqrt{2} |\nu((x',f(x')))- \nu((y',f(y'))|  \\ & \leq \sqrt{2}C([\Gamma]_{1,\alpha}) |(x', f(x')) - (y',f(y'))|^{\alpha} \leq \tilde{C}([\Gamma]_{1,\alpha})\ |x' - y'|^{\alpha},
\end{align}
where we used in the last step again that $f$ is Lipschitz continuous with $[f]_{0,1} \leq 1$. 
\end{proof}

\begin{lemma}
Suppose that $\Gamma \subset \mathbb{R}^n$ is an $(n-1)$-dimensional $C^{1,\alpha}$-hypersurface. Let $r_1 = r_1([\Gamma]_{1,\alpha})$ be as in Lemma \ref{lem:ballgraphsystem}. Then there exists $C = C([\Gamma]_{1,\alpha})$ such that for all $x_0 \in \Gamma$ one has 
\begin{equation}
    \left\vert \int_{\Gamma \cap B_r(x_0)} (x-x_0) \; \mathrm{d}\mathcal{H}^{n-1}(x) \right\vert \leq  C r^{n+\alpha} \quad \forall r < r_1. 
\end{equation}
\end{lemma}
\begin{proof}
After performing a translation and rotation we may assume $x_0 = 0$ and that $\Gamma \cap B_{r_1}(0)$ can be written as 
\begin{equation}\label{eq:xingamma}
    \Gamma \cap B_{r_1}(0) = \{ (x,u(x)) : x \in W_1 \}. 
\end{equation}
for some $u \in C^{1,\alpha}(W_1)$ such that $u(0)= \nabla u(0) = 0$, $u$ is Lipschitz with Lipschitz constant 1 and $||u||_{C^{1,\alpha}(W)} \leq C([\Gamma]_{1,\alpha})$.
As a consequence, for all $r< r_1$ there holds
\begin{equation}
    \Gamma \cap B_r(0) = \{ (x,u(x)) : x \in W \} 
\end{equation}
with $W :=  (\mathrm{id} + u e_n)^{-1}(B_r(0))$. 
We now claim that for all $x' \in W$ one has $|u(x')| \leq \tilde{C}([\Gamma]_{1,\alpha})|x'|^{1+\alpha}.$ Indeed, for all $x' \in W$ there holds with Lemma \ref{lem:normale}
\begin{align}
   | u(x')- \nabla u(x') \cdot x'|  &= |(\nu((x',u(x')), (x',u(x')))| \sqrt{1+ |\nabla u(x')|^2} \\ & = |(\nu((x',u(x')), (x',u(x'))- (0,u(0)))| \sqrt{1+|\nabla u(x')|^2} \\ & \leq C([\Gamma]_{1,\alpha}) |(x',u(x'))|^{1+\alpha}. 
\end{align}
Using $u(0) = 0$ and that $u$ is Lipschitz with Lipschitz constant $1$ one finds that the right hand side is bounded by $\sqrt{2}^{1+\alpha}C([\Gamma]_{1,\alpha})|x'|^{1+\alpha}=: \bar{C}([\Gamma]_{1,\alpha})|x'|^{1+\alpha}$. Using that by Lemma \ref{lem:ballgraphsystem} the $C^{1,\alpha}$-norm of $u$ can be bounded in terms of $[\Gamma]_{1,\alpha}$ we find 
\begin{align}
    |u(x')| & \leq  |\nabla u(x') \cdot x' |+ \bar{C}([\Gamma]_{1,\alpha})|x'|^{1+\alpha}\leq |\nabla u (x')| \; |x'| +  \bar{C}([\Gamma]_{1,\alpha})|x|^{1+\alpha}
    \\ & \leq |\nabla u (x')- \nabla u (0) | \; |x'| +  \bar{C}([\Gamma]_{1,\alpha})|x'|^{1+\alpha}  \leq \tilde{C}([\Gamma]_{1,\alpha}) |x'|^{1+\alpha}.
\end{align}
%from which the desired estimate follows immediately.
%(Notice that the evaluation $\nabla u(tx')$ for some $t \in (0,1)$ would be ill-defined if $tx' \not \in W$.  To make this evaluation well-defined we take a $C^{0,\alpha}$-extension of $\nabla u$ on the whole of $\mathbb{R}^n$ with the same Hölder constant, cf. Lemma \ref{lem:Hoelextension}).
Now define $f(s) := s\sqrt{ 1 + \tilde{C}([\Gamma]_{1,\alpha})^2 s^{2\alpha} }$. We claim that then 
\begin{equation}
    B_{f^{-1}(r)}(0) \subset W \subset B_r(0).
\end{equation}
Indeed, for the second inclusion observe that if $x \in W$ then $(x, u(x)) \in B_r(0)$ and hence
    $r^2 > |x|^2 + u(x)^2$. In particular $r^2 > |x|^2$ and thus $|x|< r$. Now we turn to the first inclusion. If $f(|x|)< r$ then
\begin{equation}
    |x|^2 + u(x)^2 \leq |x|^2 + \tilde{C}([\Gamma]_{1,\alpha})^2 |x|^{2+2\alpha} = |x|^2 (1+  \tilde{C}([\Gamma]_{1,\alpha})^2 |x|^{2\alpha}) = f(|x|)^2 < r^2 
\end{equation}
and hence $(x,u(x)) \in B_r(0)$, implying $x \in W$ by the definition of $W$. 
%That $(x,u(x)) \in \Gamma$ follows from \eqref{eq:xingamma}. Thus $x \in (\mathrm{id} + u e_n)^{-1}( \Gamma \cap B_r(0)) = W$. 
Now we turn to the integral we want to estimate 
\begin{equation}
    \int_{\Gamma \cap B_r(0)} x \; \mathrm{d}\mathcal{H}^{n-1}(x) = \int_{W}  \begin{pmatrix} x' \\ u(x') \end{pmatrix} \sqrt{1+ |\nabla u(x')|^2}  \; \mathrm{d}x'.
\end{equation}
For the last component observe that  
\begin{align}
   \left\vert  \int_{W}  u(x')  \sqrt{1+ |\nabla u(x')|^2}  \; \mathrm{d}x'\right\vert  & \leq  \int_W |u(x')| \sqrt{1+ |\nabla u (x')|^2 } \; \mathrm{d}x'  \\ & \leq \sqrt{1+ ||\nabla u||_\infty^2}  \int_{W} |u(x')|  \; \mathrm{d}x'
   \\ & = \sqrt{2} \int_{B_r(0)} \tilde{C}([\Gamma]_{1,\alpha}) |x'|^{1+\alpha} \; \mathrm{d}x'  \\
    & \leq \sqrt{2}  \tilde{C}([\Gamma]_{1,\alpha})c(n) r^{n-1} r^{1+\alpha} \leq \tilde{C}([\Gamma]_{1,\alpha})  r^{n+\alpha}.
\end{align}
For the first $n-1$ components observe 
\begin{align}
     & \int_W x_i' \sqrt{1+ |\nabla u(x')|^2} \; \mathrm{d}x' \\ & = \sqrt{1+|\nabla u(0)|^2} \int_W x_i' \; \mathrm{d}x' + \int_W x_i' ( \sqrt{1+ |\nabla u(x')|^2}- \sqrt{1 + |\nabla u(0)|^2}) \; \mathrm{d}x'. 
\end{align}
Now observe that 
\begin{equation}
    ( \sqrt{1+ |\nabla u(x')|^2}- \sqrt{1 + |\nabla u(0)|^2}) \leq D([\Gamma]_{1,\alpha}) |x'|^\alpha 
\end{equation}
and hence the absolute value of last summand is bounded by 
\begin{align}
   \left\vert  \int_W x_i' ( \sqrt{1+ |\nabla u(x')|^2}- \sqrt{1 + |\nabla u(0)|^2}) \; \mathrm{d}x_i' \right\vert  & \leq \int_{B_r(0)} D|x'| |x'|^\alpha \; \mathrm{d}x' \leq Dr^{n+\alpha}.
\end{align}
For the first summand we conclude by symmetry of $B_{f^{-1}(r)}(0)$ with respect to $x \mapsto -x$ 
\begin{align}
    \left\vert \int_W x_i' \; \mathrm{d}x' \right\vert &  =  \left\vert \int_W x_i' \; \mathrm{d}x' -\int_{B_{f^{-1}(r)}(0)} x_i' \; \mathrm{d}x' \right\vert \leq \int_{W \setminus B_{f^{-1}(r)}(0)} |x_i'| \; \mathrm{d}x'
    \\ & \leq \int_{B_r(0) \setminus B_{f^{-1}(r)}(0)} |x'| \; \mathrm{d}x' \leq \alpha_n r ( r^{n-1} - f^{-1}(r)^{n-1}). \label{eq:r^n-finv}
\end{align}
Now we estimate 
\begin{align}
    r - f^{-1}(r) = f^{-1}(f(r)) - f^{-1}(r) = \int_r^{f(r)} (f^{-1})'(s) \; \mathrm{d}s =  \int_r^{f(r)} \frac{1}{f'(f^{-1}(s))} \; \mathrm{d}s.
\end{align}
Recalling that $f(s) = s \sqrt{1 + \tilde{C}([\Gamma]_{1,\alpha})^2 s^{2\alpha}}$ we find $f'(s) \geq 1$ for all $s > 0$.
Therefore the previous two equations yield
\begin{align}
     r - f^{-1}(r) & \leq f(r) - r = r \left( \sqrt{ 1 + \tilde{C}([\Gamma]_{1,\alpha})^2 r^{2\alpha}} - 1 \right)  \\ & = \frac{\tilde{C}([\Gamma]_{1,\alpha})^2 r^{1+2\alpha}}{ 1 + \sqrt{1+ \tilde{C}([\Gamma]_{1,\alpha})^2 r^{2\alpha}}}  \leq \tilde{C}([\Gamma]_{1,\alpha})^2 r^{1+2\alpha}.
\end{align}
We conclude (using $f^{-1}(r) \leq r$) that 
\begin{equation}
     |r^{n-1} - f^{-1}(r)^{n-1}| = (r- f^{-1}(r)) \left\vert \sum_{i= 0}^{n-2} r^i f^{-1}(r)^{n-2-i}   \right\vert \leq  (r- f^{-1}(r)) nr^{n-2} \leq \tilde{C}([\Gamma]_{1,\alpha})^2 r^{n-1+2\alpha}.
\end{equation}
Together with \eqref{eq:r^n-finv} the claim follows.
\end{proof}

\begin{lemma}\label{lem:approx}
 Let $\Gamma = \partial \Omega' \subset \subset \Omega$ be a $C^{1,\alpha}$-boundary. Then there exists a sequence $(\Gamma_j = \partial \Omega_j')_{j \in \mathbb{N}}$ of $C^\infty$-boundaries such that $\Omega_j' \supset \Omega'$ and
 \begin{enumerate}
     \item $\sup_{x \in \Gamma_j} \mathrm{dist}(x, \Gamma) \rightarrow 0 $ as $j \rightarrow \infty$.
     \item $\mathcal{H}^{n-1} \mres \Gamma_j \overset{*}{\rightharpoonup} \mathcal{H}^{n-1} \mres \Gamma$ in $C(\overline{\Omega})^*$ in the sense that for each $f \in C(\overline{\Omega})$ one has $\int_{\Gamma_j} f \; \mathrm{d}\mathcal{H}^{n-1} \rightarrow \int_\Gamma f \; \mathrm{d}\mathcal{H}^{n-1}$.
     \item There exists some $C = C([\Gamma]_{1,\alpha}) > 0$ independent of $j$ such that $[\Gamma_j]_{1,\alpha} \leq C$. 
 \end{enumerate}
\end{lemma}
\begin{proof}
Let $(O_i,U_i,V_i,f_i$, $i = 1,...,M)$  represent $\Gamma$ in the sense of \eqref{eq:Gaama} and be such that 
\begin{equation}\label{eq:unt}
    \sum_{i = 1}^M \sup_{x \in V_i}|x| + 1 + |U_i| + ||f_i||_{C^{1,\alpha}(4U_i)} \leq [\Gamma]_{1,\alpha} + 1.
\end{equation}
 Fix also a regularized signed distance function $\rho$. We will show that $\Omega_j' := \{ \rho < \frac{1}{j} \}$ and $\Gamma_j = \partial \Omega_j'$ satisfies the desired properties. Property (1) is clear since $|\rho| \geq \frac{1}{2} \mathrm{dist}(\cdot, \Gamma)$. Property (2) is straightforward to show with the definition of $\rho$ and the fact that the level sets of $\rho$ are regular.
%For property (2) let first $f \in C^1(\overline{\Omega})$. 
%and let $\psi \in C_0^\infty(\Omega)$ be a cutoff function such that $\psi \equiv 1$ on $B_\frac{\epsilon}{2}(\Gamma)$ and $\mathrm{spt}(\psi) \subset B_\epsilon(\Gamma)$.
%We conclude 
%for all $j$ 
%such that $\frac{1}{j} < \frac{\epsilon}{2}$
%using that $\nu_j = \frac{\nabla \rho}{|\nabla \rho|}$ on $\Gamma_j$ for all $j$
%\begin{align}
  % &  \left\vert \int_{\Gamma_j} f \; \mathrm{d}\mathcal{H}^{n-1} - \int_\Gamma f \mathrm{d}\mathcal{H}^{n-1} \right\vert 
%=  \left\vert \int_{\Gamma_j} f  \frac{\nabla \rho}{|\nabla \rho|} \nu_j \; \mathrm{d}\mathcal{H}^{n-1} - \int_\Gamma f   \frac{\nabla \rho}{|\nabla \rho|} \nu \mathrm{d}\mathcal{H}^{n-1} \right\vert
%\\ &  = \left\vert \int_{\Omega_j'} \mathrm{div} \left( f  \frac{\nabla \rho}{|\nabla \rho|} \right) \; \mathrm{d}x - \int_{\Omega'} \mathrm{div} \left( f  \frac{\nabla \rho}{|\nabla \rho|} \right) \; \mathrm{d}x \right\vert  \leq \int_{\Omega_j' \setminus \Omega' }\left\vert \mathrm{div}\left( f \frac{\nabla \rho}{|\nabla \rho|} \right)\right\vert \; \mathrm{dx} \\
%& \leq C ||f||_{C^1} |\Omega_j' \setminus \Omega'| . 
%\end{align}
%Now 
%\begin{equation}
%    \lim_{j\rightarrow \infty} |\Omega_j' \setminus \Omega'| = |\overline{\Omega'} \setminus \Omega'| = 0 
%\end{equation}
%and therefore Property (2) holds for each $f \in C^1(\overline{\Omega})$. As each $f \in C^0(\overline{\Omega})$ can be uniformly approximated by $C^1(\overline{\Omega})$-functions we infer property (2) for  all $f \in C(\overline{\Omega})$.
Next we proceed to property (3). To this end we first show the following \\
\textbf{Intermediate claim.} $(\nabla \rho(x), O_i e_n) \neq 0$ for all  $ x \in O_i(4U_i \times V_i)\cap \Gamma$ and $i =1,...,M$. To show the intermediate claim notice that for each $x' \in 4U_i$ one has $\rho(O_i(x', f_i(x'))) = 0$. We infer by taking the derivative with respect to $x_l'$  $(l = 1,...,n-1)$
\begin{equation}
    \partial_n (\rho \circ O_i) \partial_{x_l'} f_i(x') = - \partial_l (\rho \circ O_i) .
\end{equation}
Taking squares and summing over all $l$ we find
\begin{equation}
     [\partial_n (\rho \circ O_i)]^2 |\nabla f_i|^2 = |\nabla (\rho \circ O_i)|^2 - [\partial_n(\rho \circ O_i)]^2 .
\end{equation}
This and $|\nabla \rho|^2 \geq \frac{1}{4}$ yields 
\begin{equation}
    [\partial_n(\rho \circ O_i)]^2  \geq \frac{1}{4 ( 1 + ||\nabla f_i||_\infty^2)} \geq \frac{1}{4[\Gamma]_{1,\alpha}^2}. 
\end{equation}
since $\partial_n(\rho \circ O_i) = (\nabla \rho \circ O_i , O_i e_n)$, the intermediate claim is shown.

For $s \in (0,\infty)$ define now 
\begin{equation}
    \zeta(s) := \frac{1}{2\sqrt{1+ [\Gamma]_{1,\alpha}^2}}s - 10 [\Gamma]_{1,\alpha} s^{1+\alpha}.
\end{equation}
Note that there exists some $\delta_0=\delta_0([\Gamma]_{[1,\alpha]})$ such that for each $\delta < \delta_0$  the equation $\zeta(s)= \delta$ has two positive solutions $s_1,s_2\in (0,\infty)$. Let $s_\delta> 0$ be the unique smallest positive solution of $\zeta(s)= \delta$. One readily checks that $\lim_{\delta \rightarrow 0} s_\delta = 0$.

Now choose some $\delta  \in (0,\delta_0)$ small enough such that
\begin{itemize}
    \item On $B_{3\delta}(\Gamma) \cap O_i(4U_i \times V_i)$  there holds $|(\nabla \rho  , O_i e_n)| \geq \frac{1}{8 [\Gamma]_{1,\alpha}^2}$ for all $i \in \{ 1,..., M \}$,
    \item $ B_{3\delta}(\Gamma) \subset \bigcup_{i = 1}^M O_i(U_i \times V_i),$
    \item $2s_\delta < \mathrm{diam}(V_i\cap \{ y: y > f_i(x') \})$ for all $i \in \{ 1, ...., M\}$.
\end{itemize}
 Without loss of generality we assume in the following $(\nabla \rho , O_i e_n) > 0$ on $B_{3\delta}(\Gamma)$.  We next derive a uniform bound for $[\Gamma_j]_{1,\alpha}$ for all $j  \in \mathbb{N}$ such that $\frac{1}{j} <  \delta$. We do so by giving an explicit representation for $\Gamma_j$. Let $i \in \{ 1,..., M\}$. For $x' \in 4U_i$ we define 
\begin{equation}
    h(x') := \inf \{y \in V_i, y > f_i(x'): \rho(O_i(x',y)) = \delta \}.   
\end{equation}
We claim that the number $h(x')$ is well-defined and lies in $V_i$. Indeed, if we assume  $\rho(O_i(x',y))< \delta$  for all $y \in V_i, y> f_i(x')$ then one has (with $\xi$ being a certain point on the line segment between $O_i(x',y)$ and  $O_i(x', f_i(x'))$) 
\begin{align}
    \delta &  > \rho (O_i (x',y))= \rho(O_i(x',y))- \rho(O_i(x', f_i(x'))) 
   = \nabla \rho(\xi) \cdot O_i  \begin{pmatrix}
      0 \\ y - f_i(x')
  \end{pmatrix} \\ & = \nabla \rho(O_i(x',f_i(x'))) \cdot  O_i \begin{pmatrix}
      0 \\ y - f_i(x')
  \end{pmatrix} - (\nabla \rho(\xi) - \nabla \rho(O_i(x',f_i(x')))) \cdot  O_i \begin{pmatrix}
      0 \\ y - f_i(x')
  \end{pmatrix}  
  \\ & = |\nabla \rho|(O_i(x',f_i(x'))) \nu(O_i(x',f_i(x')) \cdot  O_i \begin{pmatrix}
      0 \\ y - f_i(x')
  \end{pmatrix}   \\ & \quad  - (\nabla \rho(\xi) - \nabla \rho(O_i(x',f_i(x')))) \cdot  O_i \begin{pmatrix}
      0 \\ y - f_i(x')
  \end{pmatrix}  
  \\ & =  |\nabla \rho|(O_i(x',f_i(x')))  \frac{1}{\sqrt{1+|\nabla f_i(x')|^2}} O_i \begin{pmatrix}
      -\nabla f_i(x') \\ 1 
  \end{pmatrix} \cdot  O_i \begin{pmatrix}
      0 \\ y - f_i(x')
  \end{pmatrix} \\ & \quad  - (\nabla \rho(\xi) - \nabla \rho(O_i(x',f_i(x')))) \cdot  O_i \begin{pmatrix}
      0 \\ y - f_i(x')\end{pmatrix}
\\ & =  |\nabla \rho|(O_i(x',f_i(x')))  \frac{1}{\sqrt{1+|\nabla f_i(x')|^2}}  \begin{pmatrix}
      -\nabla f_i(x') \\ 1 
  \end{pmatrix} \cdot   \begin{pmatrix}
      0 \\ y - f_i(x')
  \end{pmatrix} \\ & \quad  - (\nabla \rho(\xi) - \nabla \rho(O_i(x',f_i(x')))) \cdot  O_i \begin{pmatrix}
      0 \\ y - f_i(x')\end{pmatrix}  
       \\ & =  |\nabla \rho|(O_i(x',f_i(x')))  \frac{1}{\sqrt{1+|\nabla f_i(x')|^2}}  (y-f_i(x')) 
  \\ & \quad  - (\nabla \rho(\xi) - \nabla \rho(O_i(x',f_i(x')))) \cdot  O_i \begin{pmatrix}
      0 \\ y - f_i(x')\end{pmatrix}  
        \\ & \geq   |\nabla \rho|(O_i(x',f_i(x')))  \frac{1}{\sqrt{1+|\nabla f_i(x')|^2}}  (y-f_i(x')) 
  \\ & \quad  - ||\nabla \rho||_{C^{0,\alpha}} |y-f_i(x')|^{1+\alpha}   
  \\ &    \geq  \frac{1}{2\sqrt{1+ [\Gamma]_{1,\alpha}^2}}|y-f_i(x')| - 10 [\Gamma]_{1,\alpha} | y- f_i(x')|^{1+\alpha} = \zeta(|y- f_i(x')|). 
\end{align}
We have shown that for all $y \in V_i, y >f_i(x')$ one has $|y- f_i(x')| \in \zeta^{-1}((-\infty,\delta)) \cap (0,\infty)$. Notice that $\zeta^{-1}((-\infty,\delta)) \cap (0,\infty) = (0,s_\delta) \cup (s_\delta^{(2)}, \infty)$ for some $s_\delta^{(2)} > s_\delta$ (the other solution of $\zeta(s) = \delta$). Notice that $V_i \ni y \mapsto |y- f_i(x')|$ is continuous and takes arbitrarily small values for $y \in V_i$ close to $f_i(x')$. Since $V_i \cap \{ y : y > f_i(x') \}$ is an interval (i.e. connected) and $(0,s_\delta)\cup (s_\delta^{(2)}, \infty)$ is disconnected we  conclude 
\begin{equation}
    |y- f_i(x')| < s_\delta  \qquad \forall y \in V_i, \; y > f_i(x'). 
\end{equation}
In particular, for all $y_1,y_2 \in V_i \cap \{ y : y > f_i(x')\}$ one has 
\begin{equation}
    |y_1- y_2| \leq |y_1 - f_i(x')| + |y_2 - f_i(x')| < 2s_\delta,
\end{equation}
contradicting $\mathrm{diam}(V_i \cap \{ y : y < f_i(x')\})> 2s_\delta$. Hence $h(x') \in V_i$ is well-defined. 

We also note that $h(x')$ is the only value $y \in V_i$ such that $\rho(O_i(x',y)) = \delta$. Indeed, $\partial_y \rho(O_i(x',y)) = (\nabla \rho(O_i(x',y)), O_i e_n) > 0$ if $O_i(x',y) \in B_{3\delta}(\Gamma)$ and if $O_i(x',y) \in B_{3\delta}(\Gamma)^C$ one has $\rho (O_i(x',y)) \geq \frac{1}{2} \mathrm{dist}(O_i(x',y),\Gamma) \geq \frac{3}{2}\delta$. Hence $y \mapsto \rho(O_i(x',y))$ must strictly increase until it  reaches the level $\frac{3}{2}\delta$ and can never drop below $\frac{3}{2}\delta$ from there. This shows that the value $\delta$ is attained exactly once. 
We therefore obtain 
\begin{equation}\label{eq:bla}
    \{ \rho = \delta \} \cap O_i( 4U_i \times V_i ) = O_i \{ (x', h(x')) : x' \in 4U_i \} \quad \forall i = 1,..., M.
\end{equation}
In particular, $\rho(O_i(x',h(x')))= \delta$ for all $x' \in 4 \overline{U_i}$. 
%From the fact that $\rho(O_i(x',h(x')))= \delta$ for all $x' \in 4 U_i$ and the implicit function theorem we obtain 
This also implies that $h \in C^{1,\alpha}(4 \overline{U_i})$. Moreover for all $v \in \mathbb{R}^{n-1} : |v|= 1$ one has (if $\partial_v$ denotes the directional derivative with respect to $v$) 
\begin{equation}\label{eq:gradh}
    \partial_v h(x') = -\frac{\partial_v (\rho \circ O_i)(x',h(x'))}{\partial_n( \rho \circ O_i)(x',h(x'))} \quad (x' \in 4\overline{U_i}).
\end{equation}
We now intend to bound $||h||_{C^{1,\alpha}(4\overline{U_i})}$ independently of $\delta$ with a constant in terms of $[\Gamma]_{1,\alpha}$. For a fixed $z \in 4U_i$ and $v \in \mathbb{R}^{n-1} : |v|=1$ look at the function $y_{z,v}(t) := h(z+ tv)$, $t \in [-a^-_{z,v},a^+_{z,v}]$ 
where (if $z_i$ is the center and $r_i$ is the radius of $4U_i$) one defines
\begin{equation}\label{eq:aplusminus}
    a^{\pm}_{z,v} = -\frac{(z-z_i,v)}{2} \pm \sqrt{\frac{(z-z_i,v)^2}{4} + r_i^2 - |z-z_i|^2}
\end{equation}
which are exactly the two unique numbers $t_{\pm} \in \mathbb{R}$ such that $z + t_{\pm} v$ lie in $\partial (4U_i)$. 
 Then \eqref{eq:gradh} yields that 
\begin{equation}
    y_{z,v}'(t) =- \frac{\partial_v (\rho \circ O_i)(z+tv,y_{z,v}(t))}{\partial_n( \rho \circ O_i)(z+tv,y_{z,v}(t))} \quad (t \in [a_{z,v}^-,a_{z,v}^+])
\end{equation}
and therefore by Lemma \ref{lem:hoeldiffeq} one has  
\begin{equation}\label{eq:y_l}
    ||y_{z,v}'||_{C^{1,\alpha}} \leq ||y_{z,v}||_\infty + (2 + (a_{z,l}^+ - a_{z,l}^-)^{1-\alpha}) (1 + ||\tfrac{\partial_v(\rho \circ O_i)}{\partial_n(\rho \circ O_i) } ||_{C^{0,\alpha}( O_i^T \{\rho = \delta\}) } )^2,
\end{equation}
where we have used that $(x',h(x')) \in O_i^T \{ \rho = \delta \}$ for all $x' \in 4 \overline{U_i}$ by \eqref{eq:bla}. 
Noticing now that on $O_i^T \{ \rho = \delta \} \subset O_i^T B_{3\delta}(\Gamma)$ one has 
\begin{equation}
    \partial_n(\rho \circ O_i)  = (\nabla \rho \circ O_i , O_i e_n ) \geq \frac{1}{8[\Gamma]_{1,\alpha}^2} 
\end{equation}
and $\nabla \rho \in C^{0,\alpha}(\mathbb{R}^n)$ one readily checks that $||\tfrac{\partial_l(\rho \circ O_i)}{\partial_n(\rho \circ O_i) } ||_{C^{0,\alpha}( O_i^T \{\rho = \delta\}) } \leq C([\Gamma]_{1,\alpha})$.  Going back to \eqref{eq:y_l} and using that $y_{z,v} = h(z+ t v) \in V_i$ for all $t \in [a_{z,v}^-, a_{z,v}^+]$ as well as by \eqref{eq:aplusminus} $|a_{z,l}^{+} - a_{z_l}^-| \leq 2 r_i =  2 (\frac{1}{\omega_n} |4U_i|)^\frac{1}{n} = 8 (\frac{1}{\omega_n} |U_i|)^\frac{1}{n}$ one has
\begin{equation}
    ||y_{z,v}'||_{C^{1,\alpha}} \leq \sup_{x \in V_i} |x| + (2+8 \tfrac{1}{\omega_n^{1/n}}(|U_i|)^\frac{1-\alpha}{n}) C([\Gamma]_{1,\alpha}) \leq C([\Gamma]_{1,\alpha}).  
\end{equation}
This and Lemma  \ref{lem:hoelcoord} yield that also $||h||_{C^{1,\alpha}(4 \overline{U_i})} \leq C([\Gamma]_{1,\alpha})$. 
 From now on we will explicitly write the dependence of $h$ on $i$ and $\delta$, i.e. $h= h_i^\delta :4 U_i \rightarrow \mathbb{R}$ for any $i = 1,...,M$ and $\delta$ as above. Observe by the fact that $\{ \rho = \delta \} \subset B_{3\delta}(\Gamma) \subset \bigcup_{i = 1}^M O_i(U_i \times V_i)$ and  \eqref{eq:bla}  that 
 \begin{align}
     \{ \rho = \delta \}  &  = \bigcup_{i = 1}^M  \{ \rho = \delta \} \cap O_i (U_i \times V_i) 
     \\ & = \bigcup_{i = 1}^M  O_i [(U_i \times V_i) \cap \{ (x',h_i^\delta(x')) : x' \in U_i \} ].
 \end{align}
 Recalling that  $h_i^\delta \in C^{1,\alpha}$ and $||h_i^\delta||_{C^{1,\alpha}(4 \overline{U}_i)} \leq C([\Gamma]_{1,\alpha})$ we find  that $O_i, U_i, V_i, h_i^\delta$, $i = 1,...,M$ is a representation of $\{ \rho = \delta \}$. Therefore
 \begin{equation}
     [\{ \rho = \delta \}]_{ 1,\alpha} \leq \sum_{i = 1}^M \sup_{x \in V_i}|x| + 1 + |U_i| + ||h_i^\delta||_{C^{1,\alpha}(U_i)}  \leq C([\Gamma]_{1,\alpha}),  
 \end{equation}
 where we used \eqref{eq:unt} and the fact that $M$ is bounded by $[\Gamma]_{1,\alpha}$ in the last step. Observing that $\delta = \frac{1}{j_0}$ for $j_0$ large enough is a possible choice, Property (3) follows.  
%We next explicitly construct a set $\Gamma_\delta$ as a graph over $\Gamma$. Let $x \in \Gamma$. For some $i \in \{ 1,...,n\}$ one has 
\end{proof}

\begin{lemma}\label{lem:C5}
 Let $\Gamma= \partial \Omega' \in C^{1,\alpha}$. Then $\mathcal{H}^{n-1}(\Gamma) \leq (1+[\Gamma]_{1,\alpha})^2$.
\end{lemma}
\begin{proof}
Let $\epsilon >0$ and $O_i,U_i,V_i,f_i$, $i= 1,...,M$ be a representation of $\Gamma$ such that 
\begin{equation}
    \sum_{i = 1}^M \sup_{x \in V_i} |x| + |U_i| + 1 + ||f_i||_{C^{1,\alpha}(4U_i)} \leq  [\Gamma]_{1,\alpha} + 1. 
\end{equation}
Then one infers from \eqref{eq:Gaama} 
\begin{align}
    \mathcal{H}^{n-1}(\Gamma) & \leq \sum_{i = 1}^M \int_{U_i} \sqrt{1+ |\nabla f_i|^2} \; \mathrm{d}x' \leq \sum_{i = 1}^M \int_{U_i} (1+ |\nabla f_i|)  \; \mathrm{d}x'  \\ & \leq \sum_{i = 1}^M |U_i|+ |U_i| \; ||f_i||_{C^{1,\alpha}}   \leq \sum_{i = 1}^M |U_i|+ (1+ [\Gamma]_{1,\alpha})  ||f_i||_{C^{1,\alpha}(U_i)} \\ &  \leq  (1+ [\Gamma]_{1,\alpha}) \sum_{i= 1}^M (|U_i|+   ||f_i||_{C^{1,\alpha}(U_i)})  \leq  (1+ [\Gamma]_{1,\alpha}) ([\Gamma]_{1,\alpha} + 1).  
\end{align}
The claim follows. 
\end{proof}

\section{The space $L^\infty(\mathbb{R}^n) \cap BV(\mathbb{R}^n)$}

In this section we prove some useful integral formulas that hold for functions in the set $L^\infty(\mathbb{R}^n) \cap BV(\mathbb{R}^n)$. 
%This space is important to deal with the gradient of very weak solutions to \eqref{eq:1.1} for smooth input data. Indeed, Section \ref{sec:smoothinput} shows that if $Q,A, \Gamma$ are suitably smooth, then $\nabla u \in L^\infty(\Omega) \cap BV(\Omega)$. 
Here we derive some important tools to deal with functions in this class. 

We first recall some facts about $BV$-functions.  
For  $v \in BV(\mathbb{R}^n)$ and $U \subset \mathbb{R}^n$ open one can define
\begin{equation}
    |Dv|(U) := \sup_{\phi \in C^1_c(U;\mathbb{R}^n), ||\phi|| \leq 1} \int_{\mathbb{R}^n} v \; \mathrm{div}(\phi) \; \mathrm{d}x.
\end{equation}
By \cite[Chapter 5]{EvGar} $|Dv|$ extends uniquely to a finite Radon measure on $\mathbb{R}^n$, called the \emph{total variation measure}. This measure has the property that there exists a Borel measurable function $\sigma_v : \mathbb{R}^n \rightarrow \mathbb{R}^n$, $|\sigma_v| = 1$ a.e. and 
\begin{equation}
    \int_{\mathbb{R}^n} v \; \mathrm{div}(\phi) \; \mathrm{d}x =  \int_{\mathbb{R}^n}  (\phi, \sigma_v) \; \mathrm{d}|Dv| \quad \forall \phi \in C_c^1(\mathbb{R}^n).   
\end{equation}
We further define  for $v \in BV(\mathbb{R}^n)$ the \emph{precise representative}
\begin{equation}
    v^*(x) := \lim_{r\rightarrow 0}  \fint_{B_r(x)} v(y) \; \mathrm{d}y,
\end{equation}
which exists for $\mathcal{H}^{n-1}$ a.e. $x \in \mathbb{R}^n$ by \cite[Theorem 5.20]{EvGar} (and not just Lebesgue almost-everywhere). If $(\phi_\epsilon)_{\epsilon > 0}$ is the standard mollifier (defined as in \cite[Section 4.2]{EvGar}) one can define $v_\epsilon := v * \phi_\epsilon$ which then  lies in $C^\infty(\mathbb{R}^n)\cap BV(\mathbb{R}^n)$. For $v \in BV(\mathbb{R}^n)$ it is shown in \cite[Theorem 5.20]{EvGar} that $v_\epsilon \rightarrow v^*$ as  $\mathcal{H}^{n-1}$ a.e. as $\epsilon \rightarrow 0$ (and not just Lebesgue almost everywhere). We will need also the following
\begin{lemma}
Let $v \in BV(\mathbb{R}^n)$ and $U \subset \mathbb{R}^n$ be open. Then for any $\epsilon > 0$ there holds 
\begin{equation}\label{lem:E1}
    \int_{U} |\nabla v_\epsilon(x)| \; \mathrm{d}x \leq |Dv|(B_\epsilon(U)) 
\end{equation}
\end{lemma}
\begin{proof}
Let $U,v$ be as in the statement. Then  
\begin{align}
    \int_U |\partial_i v_\epsilon(x)| \; \mathrm{d}x & = \int_U \left\vert \int_{\mathbb{R}^n} v(y) \partial_{x_i}\phi_{\epsilon}(x-y) \; \mathrm{d}y \right\vert \; \mathrm{d}x = 
   \int_U \left\vert \int_{\mathbb{R}^n} v(y) \partial_{y_i}\phi_{\epsilon}(x-y) \; \mathrm{d}y \right\vert \; \mathrm{d}x
   \\ & =  \int_U \left\vert \int_{\mathbb{R}^n} v(y) \mathrm{div}(\phi_{\epsilon}(x-\cdot) e_i)(y) \; \mathrm{d}y \right\vert \; \mathrm{d}x 
   \\ & = \int_U \left\vert \int_{\mathbb{R}^n} \phi_\epsilon(x-y) (\sigma_v(y), e_i)  \; \mathrm{d}|Dv|(y) \right\vert \; \mathrm{d}x
   \leq  \int_U \int_{\mathbb{R}^n} \phi_\epsilon(x-y) \mathrm{d}|Dv|(y) \; \mathrm{d}x
   \\ & = \int_{\mathbb{R}^n} \left( \int_U \phi_\epsilon(x-y) \; \mathrm{d}x \right) \; \mathrm{d}|Dv|(y) \leq \int_{\mathbb{R}^n} \chi_{\overline{B_\epsilon(U)}}(y) \; \mathrm{d}|Dv|(y) \\ & \leq   |Dv|(\overline{B_\epsilon(U)}),
\end{align}
where the last step is due to the fact that 
\begin{equation}
    0 \leq \int_{U} \phi_\epsilon (x-y) \; \mathrm{d}x \leq  \int_{\mathbb{R}^n} \phi_\epsilon (x-y) \; \mathrm{d}x = 1
\end{equation}
and if $y \not \in \overline{B_\epsilon(U)}$ then the fact that $\mathrm{spt}(\phi_\epsilon) \subset \overline{B_\epsilon(0)}$ implies 
 \begin{equation}
     \int_{U} \phi_\epsilon(x-y) \; \mathrm{d}x = 0. 
 \end{equation}
\end{proof}

\begin{lemma}[Radial integration for $BV$-functions] \label{lem:layercake}
Let $v \in BV(\mathbb{R}^n) \cap L^\infty(\mathbb{R}^n)$ have compact support in $\mathbb{R}^n$. Then 
\begin{equation}
    \int_{\mathbb{R}^n} v(x) \; \mathrm{d}x = \int_0^\infty \int_{\partial B_s(0)} v^*(x) \; \mathrm{d}\mathcal{H}^{n-1}(x) \; \mathrm{d}s
\end{equation}
\end{lemma}
\begin{proof}
Notice that  $\fint_{B_r(x)} v(y) \; \mathrm{d}y \rightarrow v^*(x)$ as $r \rightarrow 0 $ for Lebesgue a.e $x \in \mathbb{R}^n$ and $|v^*(x)| \leq ||v||_{L^\infty} \chi_{B_1(\mathrm{spt}(v))}(x)$ for all $x \in \mathbb{R}^n$ and $r \in (0,1)$. Thus the dominated convergence theorem yields
\begin{equation}
    \int_{\mathbb{R}^n} v(x) \; \mathrm{d}x = \int_{\mathbb{R}^n} v^*(x) \; \mathrm{d}x = \lim_{r \rightarrow 0} \int_{\mathbb{R}^n} \left( \fint_{B_r(x)} v(y) \; \mathrm{d}y \right) \; \mathrm{d}x. 
\end{equation}
Now observe that for fixed $r\in (0,1)$ the map $ f: \mathbb{R}^n \rightarrow \mathbb{R}, x \mapsto  \fint_{B_r(x)} v(y) \; \mathrm{d}y$ is continuous. Indeed this follows immediately from the estimate 
\begin{equation}
    |f(x_1) - f(x_2)|  \leq \frac{|B_r(x_1) \Delta B_r(x_2)|}{|B_r(0)|} ||v||_{L^\infty}. 
\end{equation}
Since radial integration is available for continuous functions one has 
\begin{equation}
     \int_{\mathbb{R}^n} v(x) \; \mathrm{d}x  = \lim_{r \rightarrow 0} \int_0^\infty \int_{\partial B_s(0)}  \left( \fint_{B_r(x)} v(y) \; \mathrm{d}y \right) \; \mathrm{d}\mathcal{H}^{n-1}(x) \; \mathrm{d}s.
\end{equation}
We can now use the dominated convergence theorem 
again to switch the limit and  the first two integrals (with the same dominating function as above) and obtain 
\begin{equation}
    \int_{\mathbb{R}^n} v(x) \; \mathrm{d}x = \int_0^\infty \int_{\partial B_s(0)} v^*(x) \; \mathrm{d}\mathcal{H}^{n-1}(x) \; \mathrm{d}s.
\end{equation}

\end{proof}

%For the proof of the next lemma we need the notion of \emph{essential variation} on intervals. For $f \in L^1((a,b);\mathbb{R})$ we call
%\begin{equation}
%    \mathrm{essV}_{(a,b)} f := \sup_{\{t_1,...,t_n\} \in \mathcal{T}} \left(   \sum_{k = 1}^n |f(t_{k+1}) - f(t_k)| \right)
%\end{equation}
%where $\mathcal{T}$ is the set of all partitions of $(a,b)$ that consits only of points of \emph{approximate continuity} of $f$ (defined as in \cite[Definition 5.8]{EvGar}. In \cite[Theorem 5.21]{EvGar} it is shown that $f \in BV((a,b))$ if and only if $\mathrm{essV}_{(a,b)} f < \infty$. In particular there holds 
%\begin{equation}\label{eq:lsc}
%   \mathrm{essV}_{(a,b)} f =  \sup_{\phi \in C^1_c(a,b),||\phi||_\infty \leq 1}  \left( \int_a^b f \phi' \; \mathrm{d}t \right)  
%\end{equation}
%If additionally $f \in W^{1,1}(a,b)$ there holds 
%\begin{equation}\label{eq:essVW11}
%   \mathrm{essV}_{(a,b)} f =  \int_a^b |f'| \; \mathrm{d}t 
%\end{equation}

\begin{lemma}\label{lem:contimeanval}
Let $v \in BV(\mathbb{R}^n) \cap L^\infty(\mathbb{R}^n)$ and $f \in C^0(\mathbb{R}^n \times \mathbb{R})$. 
Let \[ A := \{ s \in (0,\infty) : t \mapsto |Dv|(B_t(0)) \textrm{is continuous at $s$} \}. \] 
Then $(0,\infty) \setminus A$ is countable and the map $I : A \rightarrow \mathbb{R}$
\begin{equation}\label{eq:inteq}
   I(r) := \int_{\partial B_r(0)}  f(x,v^*(x)) \; \mathrm{d}\mathcal{H}^{n-1}(x)
\end{equation}
is continuous on $A$. 
\end{lemma}
\begin{proof}
Since $t \mapsto |Dv|(B_t(0))$ is monotone it can only have a countable set of discontinuities. 
Hence $(0,\infty) \setminus A$ is countable. 
We notice that at each point $s \in A$ there holds 
\begin{align}
    |Dv|(\partial B_s(0)) &  = |Dv|(\overline{B_s(0)} \setminus  B_s(0))  = \lim_{n \rightarrow \infty} |Dv| ( B_{s+1/n}(0) \setminus B_s(0)) \\ & = \lim_{n \rightarrow \infty} |Dv| ( B_{s+1/n}(0) )- |Dv|( B_s(0))  = 0. 
\end{align}
We first assume additionally that $f= f(x,p) \in C^1(\mathbb{R}^n \times \mathbb{R})$. Let $(\phi_\epsilon)_{\epsilon > 0}$ be the standard mollifier. Since by \cite[Theorem 5.20]{EvGar} $v_\epsilon := v * \phi_\epsilon$ converges to $v^*$ $\mathcal{H}^{n-1}$ a.e. we have (by the dominated convergence theorem)
\begin{equation}
    \int_{\partial B_r(0)} f(x, v^*(x)) \; \mathrm{d}\mathcal{H}^{n-1}(x) = \lim_{\epsilon\downarrow 0}  \int_{\partial B_r(0)} f(x,v_\epsilon(x)) \; \mathrm{d}\mathcal{H}^{n-1}(x). 
\end{equation}
To find an integrable dominating function we have used that by Young's convolution inequality $||v_\epsilon||_{L^\infty} = ||v * \phi_\epsilon||_{L^\infty} \leq  ||v||_{L^\infty} ||\phi_\epsilon||_{L^1} \leq ||v||_{L^\infty}$ and  $f$ is uniformly bounded on $\partial B_r(0) \times [- ||v||_{L^\infty}, ||v||_{L^\infty}]$. Now for $0 < \underline{r} \leq \overline{r} < \infty$ we estimate for $r,s \in A \cap [\underline{r},\overline{r}], r \leq s$
\begin{align}
    & I(r)- I(s) =  \left\vert \int_{\partial B_r(0)}  f(x, v^*(x)) \; \mathrm{d}\mathcal{H}^{n-1}(x) - \int_{\partial B_s(0)}  f(x, v^*(x)) \; \mathrm{d}\mathcal{H}^{n-1}(x) \right\vert 
    \\ & =  \lim_{ \epsilon \downarrow 0 } \left\vert \int_{\partial B_r(0)}  f(x, v_\epsilon(x)) \; \mathrm{d}\mathcal{H}^{n-1}(x) - \int_{\partial B_s(0)}  f(x, v_\epsilon(x)) \; \mathrm{d}\mathcal{H}^{n-1}(x) \right\vert
    \\ & = \lim_{ \epsilon \downarrow 0 } \left\vert  \frac{1}{r}\int_{\partial B_r(0)}  (f(x, v_\epsilon(x))x , \tfrac{x}{r} )  \; \mathrm{d}\mathcal{H}^{n-1}(x) - \frac{1}{s}\int_{\partial B_s(0)}  (f(x, v_\epsilon(x)) x, \tfrac{x}{s}) \; \mathrm{d}\mathcal{H}^{n-1}(x) \right\vert
    \\ & = \lim_{\epsilon \downarrow 0} \left\vert \frac{1}{r} \int_{B_r(0)} \mathrm{div}(f(x,v_\epsilon(x)) x) \; \mathrm{d}x - \frac{1}{s} \int_{B_s(0)} \mathrm{div}(f(x,v_\epsilon(x)) x) \; \mathrm{d}x \right\vert 
    \\ &  \leq \lim_{\epsilon \downarrow 0} \left\vert \frac{n}{r} \int_{B_r(0)}  f(x,v_\epsilon(x)) \; \mathrm{d}x - \frac{n}{s} \int_{B_s(0)} f(x,v_\epsilon(x)) \; \mathrm{d}x \right\vert 
    \\ & \quad  +  \lim_{\epsilon \downarrow 0} \left\vert \frac{1}{r} \int_{B_r(0)}  (D_xf(x,v_\epsilon(x)),x) \; \mathrm{d}x - \frac{1}{s} \int_{B_s(0)} (D_xf(x,v_\epsilon(x)),x) \; \mathrm{d}x \right\vert 
    \\ &  \quad  +  \lim_{\epsilon \downarrow 0} \left\vert \frac{1}{r} \int_{B_r(0)}  \partial_pf(x,v_\epsilon(x)) (\nabla v_\epsilon(x),x)  \; \mathrm{d}x - \frac{1}{s} \int_{B_s(0)} \partial_pf(x,v_\epsilon(x)) (\nabla v_\epsilon(x), x)  \; \mathrm{d}x \right\vert\\ & = \textrm{(I) + (II) + (III)}.
\end{align}
In the first two terms one can (thanks to the fact that $f$ and $ (x,p)  \mapsto (D_xf(x,p),x),$ is uniformly bounded on $\overline{B_{\overline{r}}(0)} \times [- ||v||_{L^\infty}, ||v||_{L^\infty}]$) let $\epsilon \rightarrow 0$ and obtain 
\begin{align}
    \textrm{(I)+ (II)}  & =  \left\vert \frac{n}{r} \int_{B_r(0)}  f(x,v^*(x)) \; \mathrm{d}x - \frac{n}{s} \int_{B_s(0)} f(x,v^*(x)) \; \mathrm{d}x \right\vert 
    \\ &  +   \left\vert \frac{1}{r} \int_{B_r(0)}  (D_xf(x,v^*(x)),x) \; \mathrm{d}x - \frac{1}{s} \int_{B_s(0)} (D_xf(x,v^*(x)),x) \; \mathrm{d}x \right\vert.
\end{align}
This expression is obviously continuous in $r,s$, again by boundedness of $f$ and $(x,p) \mapsto (D_xf (x,p),x)$ on $\overline{B_{\overline{r}}(0)} \times [- ||v||_{L^\infty}, ||v||_{L^\infty}]$. For $(III)$ we estimate 
\begin{align}
   \textrm{(III)} =  & \left\vert \frac{1}{r} \int_{B_r(0)}  \partial_pf(x,v_\epsilon(x)) (\nabla v_\epsilon(x),x)  \; \mathrm{d}x - \frac{1}{s} \int_{B_s(0)} \partial_pf(x,v_\epsilon(x)) (\nabla v_\epsilon(x), x)  \; \mathrm{d}x \right\vert
     \\ & \leq  \left\vert  \frac{1}{r} - \frac{1}{s} \right\vert \;   \int_{B_r(0)} |\partial_pf(x,v_\epsilon(x)) (\nabla v_\epsilon(x),x)|  \; \mathrm{d}x
    \\ & \quad  + \frac{1}{s} \int_{B_s(0) \setminus B_r(0)} |\partial_p f(x,v_\epsilon(x)) ( \nabla v_\epsilon(x), x)| \; \mathrm{d}x.
\end{align}
Now we let $M := \sup_{(x,p) \in \overline{B_{\overline{r}}(0)} \times [- ||v||_{L^\infty}, ||v||_{L^\infty}] } |\partial_p f(x,p)|$ and infer 
\begin{align}
    \textrm{(III)}  & \leq M \left\vert 1 - \frac{r}{s} \right\vert \int_{B_r(0)} |\nabla v_\epsilon(x)| \; \mathrm{d}x  + M  \int_{B_s(0) \setminus B_r(0)} |\nabla v_\epsilon(x)| \; \mathrm{d}x.
\end{align}
 We now use Lemma \ref{lem:E1} to estimate 
\begin{align}
    \textrm{(III)} & \leq M \left\vert 1 - \frac{r}{s} \right\vert |Dv|(\mathbb{R}^n) + M |Dv|(\overline{B_\epsilon(B_s(0) \setminus B_r(0)})  \\ & \leq M \left\vert 1 - \frac{r}{s} \right\vert |Dv|(\mathbb{R}^n) + M |Dv|(\overline{B_{s+\epsilon}(0)} \setminus B_{r-\epsilon}(0)).
\end{align}
Letting $\epsilon \downarrow 0$ continuity properties of Radon measures and the fact that $s \in A$ yield that 
\begin{align}
     \textrm{(III)}   & \leq M \left\vert 1 - \frac{r}{s} \right\vert |Dv|(\mathbb{R}^n) + M |Dv|(\overline{B_{s}(0)} \setminus B_{r}(0))
      \\ & = M \left\vert 1 - \frac{r}{s} \right\vert |Dv|(\mathbb{R}^n) + M (|Dv|(B_s(0)) - |Dv|(B_r(0))) + M |Dv|(\partial B_s(0))  \\  & = M \left\vert 1 - \frac{r}{s} \right\vert |Dv|(\mathbb{R}^n) + M (|Dv|(B_s(0)) - |Dv|(B_r(0))) .
\end{align}
Since $t \mapsto |Dv|(B_t(0))$ is continuous on $A$ we have thus found a continuous function $g: A \rightarrow \mathbb{R}$ such that 
\begin{equation}
    |I(r) - I(s)| \leq |g(r)- g(s)| \quad \forall r,s \in A. 
\end{equation}
This implies continuity of $I$ on $A$. Finally, to obtain the result for any continuius $f \in C^0(\mathbb{R}^n \times \mathbb{R})$, we approximate $f$ uniformly by a sequence of $C^1$ function and argue by the Definiton of $I$ that continuity of $I$ on $A$ is preserved by this uniform approximation. 
\end{proof}

%\section{The trace operator}
%(TODO)

\section{The signed distance function}\label{app:signdist}

In this section %et $\Omega \subset \mathbb{R}^n$ be a smooth domain and $\Omega' \subset \subset \Omega$. Further 
let $\Gamma = \partial \Omega' \in C^k$ for some $k \geq 2$. In this appendix we recall some properties of the \emph{signed distance function} of $\Gamma$ and derive some useful formulas. The signed distance function  $d_\Gamma : \Omega \rightarrow \mathbb{R}$ is defined as
\begin{equation}
    d_\Gamma(x) := \begin{cases}
      \mathrm{dist}(x, \Gamma) & x \in \Omega'' := \Omega \setminus  \overline{\Omega'}, \\ 0 & x \in \Gamma, \\ - \mathrm{dist}(x,\Gamma) & x \in \Omega'. 
    \end{cases}
\end{equation}
One readily checks that $d_\Gamma$ and $|d_\Gamma| = \mathrm{dist}(\cdot, \Gamma)$ are Lipschitz continuous functions. In a small tubular neighborhood $B_{\epsilon_0}(\Gamma) := \{ x \in \mathbb{R}^n : \mathrm{dist}(x,\Gamma) < \epsilon_0 \}$ one can however still obtain more regularity for $d_\Gamma$. By \cite[Lemma 14.16]{GilTru} one has that $d_\Gamma \in C^k(B_{\epsilon_0}(\Gamma))$ and $\nabla d_\Gamma = \nu \circ \pi_\Gamma$, where $\nu$ denotes the outward pointing unit normal of $\Omega'$ on $\Gamma$ and $\pi_\Gamma : B_{\epsilon_0}(\Gamma) \rightarrow \Gamma$ denotes the \emph{nearest point projection} on $\Gamma$ which lies in $C^{k-1}(B_{\epsilon_0}(\Gamma))$, cf. \cite[Eq. (14.96)]{GilTru}. Next we derive some helpful formulas for test functions involving $d_\Gamma$. We remark that $\epsilon_0= \epsilon_0(\Gamma)$ can be chosen independent of $k$. 
%For the observations to hold true we only look at $d_\Gamma$ on $B_\epsilon(\Gamma)$ for some $\epsilon < \min\{ \epsilon_0, \mathrm{dist}(\Omega',\Omega^C) \}$. 

\begin{lemma}\label{lem:distprep}
Let $Q_0 \in W^{2,s}(\Omega)$ for some $s>n$ and let $\Gamma  = \partial \Omega' \in C^2$ with outer unit normal $\nu$. Further let $d_\Gamma : B_{\epsilon_0}(\Gamma) \rightarrow \mathbb{R}$ be the signed distance function of $\Gamma$ and $\epsilon< \epsilon_0$. Then (distributionally in $C_0^\infty(B_\epsilon(\Gamma))'$) there holds 
\begin{equation}\label{eq:distancedistr}
    \partial^2_{ij}\left( \frac{Q_0}{2}|d_\Gamma| \right)  = Q_0 \nu_i \nu_j \mathcal{H}^{n-1}\mres \Gamma + \partial^2_{ij}( \frac{Q_0}{2} d_\Gamma) ( \chi_{\Omega''} - \chi_{\Omega'}).
\end{equation}
%for some $g_{ij} \in L^p(B_\epsilon(\Gamma))$. %that depends only on $Q$ and the principal curvatures of $\Gamma$. 
%In particular,
%\begin{equation}\label{eq:Laplace}
%   - \Delta \left( \frac{Q}{2}|d_\Gamma| \right) = - Q \; \mathcal{H}^{n-1} \mres \Gamma - \mathrm{tr}(g).
%\end{equation}
\end{lemma}
\begin{proof}
Let $\phi \in C_0^\infty(B_\epsilon(\Gamma))$ be arbitrary but fixed. Denote by $\nu = \nu^{\Omega'}$ and $\nu^{\Omega''}$ the outer unit normals of $\Omega', \Omega''$ respectively.
Then using integration by parts and $d_\Gamma \vert_\Gamma = 0$ 
\begin{align}
    \int \frac{Q_0}{2}|d_\Gamma| \partial^2_{ij} \phi \; \mathrm{d}x & = -\int_{\Omega'} \frac{Q_0}{2} d_\Gamma \partial^2_{ij} \phi \; \mathrm{d}x + \int_{\Omega''} \frac{Q_0}{2} d_\Gamma \partial^2_{ij} \phi \; \mathrm{d}x
    \\ & = - \int_{\partial \Omega'} \frac{Q_0}{2}d_\Gamma \nu_i^{\Omega'} \partial_j \phi \; \mathrm{d}\mathcal{H}^{n-1} + \int_{\partial \Omega''} \frac{Q_0}{2}d_\Gamma \nu_i^{\Omega''} \partial_j \phi  \; \mathrm{d}\mathcal{H}^{n-1}\\ 
    & \quad + \int_{\Omega'} \partial_i \left( \frac{Q_0}{2} d_\Gamma \right) \partial_j \phi \; \mathrm{d}x -  \int_{\Omega''} \partial_i \left( \frac{Q_0}{2} d_\Gamma \right) \partial_j \phi \; \mathrm{d}x
    \\ & = \int_{\Omega'} \partial_i \left( \frac{Q_0}{2} d_\Gamma \right) \partial_j \phi \; \mathrm{d}x -  \int_{\Omega''} \partial_i \left( \frac{Q_0}{2} d_\Gamma \right) \partial_j \phi \; \mathrm{d}x
    \\ & = \int_{\partial \Omega'} \partial_i \left(\frac{Q_0}{2}d_\Gamma\right) \nu^{\Omega'}_j  \phi \; \mathrm{d}\mathcal{H}^{n-1} - \int_{\partial \Omega''} \partial_i \left(\frac{Q_0}{2}d_\Gamma\right) \nu^{\Omega''}_j \phi \; \mathrm{d}\mathcal{H}^{n-1}\\
    & \quad - \int_{\Omega'} \partial^2_{ij} \left( \frac{Q_0}{2} d_\Gamma \right)  \phi \; \mathrm{d}x + \int_{\Omega''} \partial^2_{ij} \left( \frac{Q_0}{2} d_\Gamma \right)  \phi \; \mathrm{d}x.
\end{align}
Noticing that $\partial \Omega' \cap \mathrm{supp}(\phi) = \partial \Omega'' \cap \mathrm{supp}(\phi) = \Gamma$ and $\nu_j^{\Omega''} = - \nu_j^{\Omega'}$ on $\Gamma$ we infer 
\begin{equation}
     \int_\Omega \frac{Q_0}{2}|d_\Gamma| \partial^2_{ij} \phi \; \mathrm{d}x  = \int_\Gamma \partial_i (Q_0 d_\Gamma) \nu_j \; \mathrm{d} \mathcal{H}^{n-1} + \int g_{ij} \phi \; \mathrm{d}x,
\end{equation}
where $g_{ij} := \partial^2_{ij}( \frac{Q_0}{2} d_\Gamma) ( \chi_{\Omega''} - \chi_{\Omega'}) \in L^s(B_\epsilon(\Gamma))$. Now notice that on $\Gamma$ one has 
\begin{equation}
    \partial_i (Q_0 d_\Gamma) = \partial_i (Q_0) d_\Gamma + Q_0 \partial_i d_\Gamma = 0 + Q_0 \nu_i = Q_0 \nu_i.
\end{equation}
Thus we infer 
\begin{equation}
     \int_\Omega \frac{Q_0}{2}|d_\Gamma| \partial^2_{ij} \phi \; \mathrm{d}x = \int_\Gamma Q_0 \nu_i \nu_j  \phi \; \mathrm{d}\mathcal{H}^{n-1} + \int_\Omega g_{ij} \phi \; \mathrm{d}x,
\end{equation}
which was asserted.
%Formula \eqref{eq:Laplace} follows immediately from $- \Delta f  = - \sum_{i = 1}^n \partial^2_{ii} f$ and 
%and using $\sum_i \nu_i^2 = 1$. 
\end{proof}

\begin{cor}\label{cor:signdistBV}
Let $\Gamma = \partial \Omega' \in C^2$ and $Q_0 \in W^{2,s}(\Omega)$ for some $s > n$. Then $\nabla \left( \frac{Q_0}{2}|d_\Gamma| \right) \in BV(B_\epsilon(\Gamma))$.
\end{cor}
\begin{proof}
Let $\phi = (\phi_1,...,\phi_n) \in C_0^\infty(B_\epsilon(\Gamma);\mathbb{R}^n)$. Then we compute for $j = 1,...,n$ 
\begin{align}
    \int_\Omega \partial_j \left( \tfrac{Q_0}{2}|d_\Gamma| \right) \mathrm{div}(\phi) \; \mathrm{d}x  = \sum_{k = 1}^n \int_\Omega   \partial_j \left( \tfrac{Q_0}{2}|d_\Gamma| \right) \partial_k \phi_k  \; \mathrm{d}x 
    = - \sum_{k = 1}^n \int_\Omega    \left( \tfrac{Q_0}{2}|d_\Gamma| \right) \partial^2_{jk} \phi_k  \; \mathrm{d}x. 
\end{align}
Now we use \eqref{eq:distancedistr} to find (with the notation $g_{ij} := \partial^2_{ij} ( \frac{Q_0}{2} d_\Gamma ) ( \chi_{\Omega''} - \chi_{\Omega'}) \in L^s( B_\epsilon(\Gamma))$ 
that 
\begin{equation}
    \int_\Omega \partial_j \left( \tfrac{Q_0}{2}|d_\Gamma| \right) \mathrm{div}(\phi) \; \mathrm{d}x = - \sum_{k =1}^n  \left( \int_\Gamma  Q_0 \nu_j \nu_k \phi_k \; \mathrm{d}\mathcal{H}^{n-1} + \int_\Omega g_{ik}\phi_k \right). 
\end{equation}
One readily checks that the term on the right hand side is bounded by 
\begin{equation}
     \left( \sum_{k =1}^n ||Q_0||_{L^\infty} \mathcal{H}^{n-1}(\Gamma)  + ||g_{jk}||_{L^1(B_\epsilon(\Gamma))} \right) ||\phi||_{L^\infty} =: C_j ||\phi||_{L^\infty}.
\end{equation}
Notice that $C_j < \infty$ since $L^s(B_\epsilon(\Gamma)) \subset L^1(B_\epsilon(\Gamma))$. Using the fact that each $C^1_c(B_\epsilon(\Gamma))$-function can be uniformly approximated by $C^\infty_c(B_\epsilon(\Gamma))$-functions we find  
\begin{equation}
     \int_\Omega \partial_j \left( \tfrac{Q_0}{2}|d_\Gamma| \right) \mathrm{div}(\phi) \; \mathrm{d}x \leq C_j ||\phi||_\infty \quad \forall \phi \in C^1_c(B_\epsilon(\Gamma);\mathbb{R}^n),
\end{equation}
implying that  $\partial_j \left( \tfrac{Q_0}{2}|d_\Gamma| \right) \in BV(B_\epsilon(\Gamma))$. Since $j \in \{ 1,...,n\}$ was arbitrary the claim is shown. 
\end{proof}

\begin{lemma}\label{lem:F3}
Let $\Gamma = \partial \Omega' \in C^k$, $k \geq 2$ and $Q_0 \in W^{k,s}(\Omega)$, $s> n$. Then $\left( \frac{Q_0}{2}|d_\Gamma| \right) \in W^{k,s}(\Omega' \cap B_\epsilon(\Gamma))  \cap W^{k,s}(\Omega'' \cap B_\epsilon(\Gamma))$.
\end{lemma}
\begin{proof}
This follows immediately from the fact that on $\Omega' \cap B_\epsilon(\Gamma)$ there holds 
\begin{equation}
     \left( \frac{Q_0}{2}|d_\Gamma| \right) =  - \frac{Q_0}{2}d_\Gamma, 
\end{equation}
and $d_\Gamma$ extends to a function in $C^k(\overline{\Omega}' \cap B_\epsilon(\Gamma)) = C^k((\Omega' \cap B_\epsilon(\Gamma)) \cup \Gamma)$. Hence all derivatives up to order $k$ are bounded. The product rule in $W^{k,s}$ yields then the desired $W^{k,s}(\Omega' \cap B_\epsilon(\Gamma))$-regularity.
On $\overline{\Omega''}\cap B_\epsilon(\Gamma)$ the same regularity can be obtained analogously. 
\end{proof}

\section{A maximum principle for $BV$-solutions}\label{app:maxpr}

In this appendix we introduce a nonclassical maximum principle for so-called \emph{$BV$-solutions} of
\begin{equation}\label{eq:AharmBV}
\begin{cases}
  -\mathrm{div}(A(x) \nabla u ) = 0  & \textrm{in } D, \\ u = h & \textrm{on } \partial D.
\end{cases}
\end{equation}
This concept of solution is defined as follows. 
\begin{definition}\label{def:BVsol}
Let $D \subset \mathbb{R}^n$ be a $C^\infty$-smooth domain, $A \in W^{2,s}(D;\mathbb{R}^{n\times n})$ $(s > n)$ and $h \in L^1(\partial D)$. We say $u \in BV(D)$ is a \emph{$BV$-solution} of  \eqref{eq:AharmBV} if $\mathrm{tr}_{\partial D}(u) = h$ and 
\begin{equation}
    \int_D    u(x) \; \mathrm{div}(A(x) \nabla \eta(x)) \; \mathrm{d}x = 0 \quad \forall \eta \in C_0^\infty(D). 
\end{equation}
 \end{definition}
If a $BV$-solution $u$ lies additionally in $W^{1,p}(D)$ for some $p \in (1,\infty)$ then it is easy to see that it is a usual weak solution of the problem and hence established methods like regularity theory, maximum principles etc. can be used. %Lying in the space $BV(D)$ is however not obviously sufficient to obtain further  regularity. 
However it is not clear whether a $BV$-solution lies in some higher Sobolev space. 
It was a longstanding conjecture by Serrin (cf. \cite{Serrin}) under which conditions on the coefficient matrix $A(x)$ further interior regularity can be obtained. In \cite[Theorem A.1.2]{Ancona}, Brezis presents a presumably optimal condition: For Dini-continuous coefficients $A(x)$ further interior regularity can be obtained. We will use this further regularity to obtain a maximum principle.

\begin{lemma}\label{lem:maxBV} Let $D \subset \mathbb{R}^n$ be a $C^\infty$-smooth domain, $A \in W^{2,s}(D;\mathbb{R}^{n\times n})$, $(s >n)$ and suppose that $\psi \in BV(D)$ is a BV-solution of \eqref{eq:AharmBV} for some $h \in W^{1,s}(\mathbb{R}^n)$.
 %for some $h \in W^{2,s}(\mathbb{R}^n)$,  and 
 %\begin{equation}
 %    \int_\Omega  \psi(x) \; \mathrm{div}(A(x) \nabla \eta(x)) \; \mathrm{d}x = 0 \quad \forall \eta \in C_0^\infty(\Omega).  
 %\end{equation}
 Then 
\begin{equation}
    ||\psi||_{L^\infty(D)} \leq ||h||_{L^\infty(\partial D)}
\end{equation}
\end{lemma}
\begin{proof}
Observe first that by \cite[Theorem A.1.2]{Ancona} $\psi \in BV(D) \cap W^{1,2}_{loc}(D)$. This in particular implies (cf. Lemma \ref{lem:B5}) that $\psi \in W^{1,1}(D)$ and 
\begin{equation}\label{eq:weakW11}
    \int_\Omega (A(x) \nabla \psi(x) , \nabla \eta(x) ) \; \mathrm{d}x = 0 \quad \forall \eta \in C_0^\infty(\Omega).
\end{equation}
We infer that $\psi$ is a \emph {weak $W^{1,1}$-solution} of 
\begin{equation}\label{eq:G2}
    \begin{cases}
        \mathrm{div}(A(x) \nabla \psi(x)) = 0  & \textrm{in } D, \\ \psi = h & \textrm{on } \partial D.
    \end{cases}
\end{equation}
%We claim that such weak $W^{1,1}$-solution $\psi$ is uniquely determined by the requirements
in the sense that 
$\mathrm{tr}_{\partial D}(\psi) = h$ and \eqref{eq:weakW11} holds. We next observe that the concept of a $W^{1,1}$-solution just introduced is a notion that characterises $\psi$ uniquely.
Indeed, if $\psi_1,\psi_2$ are two $W^{1,1}$-solutions of \eqref{eq:G2} (i.e. satisfying \eqref{eq:weakW11} and having the same boundary trace) then $\psi_1 - \psi_2$ lies in $W_0^{1,1}(D)$ and also satisfies \eqref{eq:weakW11}. By \cite[Theorem A5.1, second sentence]{Ancona} we infer $\psi_1- \psi_2 = 0$. 
%We remark that this uniqueness result is not at all trivial and needs some regularity of the coefficient matrix $A$ as \cite[Appendix A{Ancona}. 
%\textbf{Step 1.} Proof under smoothness assumption on $h$. 
%First assume that $h = \tilde{h}\big\vert_{\partial D}$ for some $\tilde{h} \in C^\infty(\mathbb{R}^n)$.
%\textbf{Step 1.} We first show the claim under the additional assumption that $h \in W^{2,s}(\mathbb{R}^n)$
We claim next that for each $h \in W^{1,s}(\mathbb{R}^n)$ there exists a weak $W^{1,1}$-solution of \eqref{eq:G2} which additionally lies in $W^{1,s}(D)$. To this end notice that $A\nabla h \in L^s(D)$ and hence by \cite[Theorem 1.1]{Auscher} there exists a unique $\tilde{\psi} \in W_0^{1,s}(D)$ weak solution of 
\begin{equation}\label{eq:divdiv}
    \begin{cases}
      -\mathrm{div}(A \nabla \tilde{\psi}) = \mathrm{div}(A  \nabla h ) & \textrm{in } D, \\ \tilde{\psi} = 0 & \textrm{on } \partial D.
    \end{cases}
\end{equation}
We now define $\bar{\psi} := \tilde{\psi} + h$ and assert that $\bar{\psi}$ is a weak $W^{1,1}$-solution. Indeed, $\mathrm{tr}_{\partial D} (\bar{\psi}) = \mathrm{tr}_{\partial D} (\tilde{\psi} + h) = 0 + h \vert_{\partial D} = h$ on $\partial D$. Moreover for each $\eta \in C_0^\infty(\Omega)$ there holds by \eqref{eq:divdiv}
\begin{align}
    \int_D (A(x) \nabla \bar{\psi}(x), \nabla \eta(x) ) \; \mathrm{d}x  & =  \int_D (A(x) \nabla \tilde{\psi}(x) , \nabla \eta(x) ) \; \mathrm{d}x + \int_D (A \nabla h(x), \nabla \eta(x))  \; \mathrm{d}x \\ & = -\int_D (A \nabla h(x), \nabla \eta(x)) \; \mathrm{d}x +  \int_D (A \nabla h(x), \nabla \eta(x)) \; \mathrm{d}x = 0.
\end{align}
%One readily checks that such $\bar{\psi}$ satisfies \eqref{eq:weakW11}.
By the previous uniqueness discussion for $W^{1,1}$-solutions of \eqref{eq:G2} we have $\psi = \bar{\psi}$ and thus $\psi \in W^{1,s}(D)$. In particular, $\psi \in  C(\overline{D})$ by Sobolev embedding. Now, the classical maximum principle implies the claim.  
%Now fix $K > ||h||_{L^\infty(\partial D)}$ arbitrary. Then $(\psi-K)^+$ has compact support in $D$ and hence (since $s>n \geq 2$) $(\psi - K)^+ \in W_0^{1,2}(D)$. Let $(\psi_j)_{j \in \mathbb{N}} \subset C_0^\infty(D)$ be such that $\psi_j \rightarrow (\psi - K)^+$ in $W^{1,2}(D)$. Then by the fact that $\psi$ is a weak $W^{1,1}$-solution and $\psi \in W^{1,2}(D)$ one has
%\begin{align}
%    0 & =   \lim_{j \rightarrow \infty} \int_D ( A \nabla \psi , \nabla \psi_j) \; \mathrm{d}x  = \int_D (A \nabla \psi , \nabla(\psi-K)^+ ) \; \mathrm{d}x = \int_{\psi > K} (A \nabla \psi , \nabla \psi ) \; \mathrm{d}x  \\ &  = \int_D (A \nabla ( \psi- K)^+ , \nabla (\psi- K)^+) \; \mathrm{d}x \geq \lambda \int_D |\nabla (\psi - K)^+|^2 \; \mathrm{d}x.
%\end{align}
%We infer that $\nabla ( \psi - K)^+ = 0$ a.e. and hence $(\psi - K)^+ = 0$ a.e. in $D$. Similarly one shows that $\psi(x) \geq - K$ for a.e. $x \in D$. Choosing a sequence $K_n \rightarrow ||h||_{L^\infty(\partial D)}$ we infer
%\begin{equation}
%    ||\psi||_{L^\infty(D)} \leq ||h||_{L^\infty(\partial D)}.
%\end{equation}
\end{proof}

\end{document}